\documentclass[12pt,reqno]{amsart}
\usepackage[foot]{amsaddr}
\usepackage{amsmath}
\usepackage{amssymb}
\usepackage{verbatim}
\usepackage{bbm}
\usepackage{graphicx}
\usepackage[font=footnotesize]{caption}
\usepackage{subcaption}
\usepackage{epstopdf}
\usepackage{multirow}
\usepackage{url}
\usepackage{makecell}
\usepackage{tabularx}
\usepackage{longtable}
\usepackage{graphicx}
\usepackage{tikz}
\usepackage{pgfplots}
\usepackage{xcolor}
\usepackage{pgfplots}
\usepackage{pgfplotstable}
\usepackage{wavelet}
\allowdisplaybreaks

\usepackage[toc]{appendix}
\usepackage{color}
\allowdisplaybreaks

\usepackage[top=0.5in,bottom=0.5in,left=0.7in,right=0.7in]{geometry}

\newtheorem{theorem}{Theorem}

\newtheorem{example}{Example}

\usepackage[capitalise]{cleveref} 

\newcommand{\cI}{\mathcal{I}}
\newcommand{\nv}{\mathbf{\nu}} 
\newcommand{\ka}{\kappa}
\newcommand{\ia}{\textsf{i}}
\def\Xint#1{\mathchoice
	 {\XXint\displaystyle\textstyle{#1}}%
	{\XXint\textstyle\scriptstyle{#1}}%
	 {\XXint\scriptstyle\scriptscriptstyle{#1}}%
	 {\XXint\scriptscriptstyle\scriptscriptstyle{#1}}%
	\!\int}
\def\XXint#1#2#3{{\setbox0=\hbox{$#1{#2#3}{\int}$ }
		\vcenter{\hbox{$#2#3$ }}\kern-.6\wd0}}
\def\ddashint{\Xint=}

\renewcommand{\ge}{\geqslant}
\renewcommand{\le}{\leqslant}

\renewcommand{\lrs}[3]{(l_{#1}(\mathbb{Z}))^{#2\times #3}} 
\newcommand{\LpI}[1]{L^{#1}(\mathcal{I})}
\newcommand{\LpId}[1]{L^{#1}([0,1]^d)}
\newcommand{\LpO}[1]{L^{#1}(\Omega)}

\numberwithin{equation}{section}
\numberwithin{theorem}{section}
\numberwithin{example}{section}

\begin{document}
	
	\title{Wavelet Galerkin Method for an Electromagnetic Scattering Problem}
	
	\author{Bin Han}
	\address{Department of Mathematical and Statistical Sciences, University of Alberta, Edmonton, Alberta, Canada T6G 2G1.}
	\email{bhan@ualberta.ca}
	
	\author{Michelle Michelle}
	\address{Department of Mathematics, Purdue University, West Lafayette, IN, USA 47907.}
	\email{mmichell@purdue.edu}
	
	\thanks{Research supported in part by
		Natural Sciences and Engineering Research Council (NSERC) of Canada under grant RGPIN-2024-04991, NSERC Postdoctoral Fellowship, and the Digital Research Alliance of Canada.\\
		Data Availability: Enquiries about data availability should be directed to the authors.\\
		Conflict of interest: The authors have not disclosed any competing interests.}
	
	\makeatletter \@addtoreset{equation}{section} \makeatother
	
	\begin{abstract}
	The Helmholtz equation with variable wavenumbers is challenging to solve numerically due to the pollution effect, which often results in a huge ill-conditioned linear system. In this paper, we present a high-order wavelet Galerkin method to numerically solve an electromagnetic scattering from a large cavity problem modeled by the 2D Helmholtz equation with variable wavenumbers. The high approximation order and the sparse linear system with uniformly bounded condition numbers offered by wavelets are useful in dealing with the pollution effect. Using the direct approach in \cite{HM21a}, we present various optimized spline biorthogonal wavelets on a bounded interval. We provide a self-contained proof to show that the tensor product of such wavelets forms a 2D Riesz wavelet in the appropriate Sobolev space. Compared to the coefficient matrix of the finite element method (FEM), when an iterative scheme is applied to the coefficient matrix of our wavelet Galerkin method, much fewer iterations are needed for the relative residuals to be within a tolerance level. Furthermore, for a given bounded variable wavenumber, the number of required iterations is practically independent of the size of the wavelet coefficient matrix, due to the small uniformly bounded condition numbers of such wavelets. In contrast, when an iterative scheme is applied to the FEM coefficient matrix, the number of required iterations doubles as the mesh size for each axis is halved. The implementation can also be done conveniently thanks to the simple structure, the refinability property, and the analytic expression of our wavelet bases.
	\end{abstract}
	\maketitle
	
	\pagenumbering{arabic}
	
	\section{Introduction and Motivations}\label{sec:intro}

	In this paper, we consider an electromagnetic scattering from a large cavity problem presented in \cite{ABW02,BL14,BS05,DLSY18,DSZ13}, which is modeled by the following 2D Helmholtz equation
	\be \label{cavity:model}
	\begin{aligned}
		&\Delta u + \ka^2  u = f \quad \text{in} \quad \Omega:=(0,1)^2,\\
		&u=0 \quad \text{on} \quad \partial \Omega \backslash \Gamma,\\
		&\frac{\partial u}{\partial \nv} = \mathcal{T}(u) + g \quad \text{on} \quad \Gamma,
	\end{aligned}
	\ee
	where $\ka \in L^{\infty}(\Omega)$ is a variable wavenumber, $\Gamma:=(0,1) \times \{1\}$, $f \in \LpO{2}$, $g \in H^{1/2}(\Gamma)$, $\nv$ is the unit outward normal, and
	\be \label{cavity:Tu}
	\mathcal{T}(u):=\frac{\ia \ka_0}{2} \ddashint_0^1 \frac{1}{|x-x'|}H^{(1)}_1(\ka_0 |x-x'|) u(x',1) dx',
	\ee
where $\kappa_0 > 0$, $\ddashint$ denotes the Hadamard finite part integral, and $H^{(1)}_1$ is the Hankel function of the first kind of degree $1$. We shall explain how $\ka$ is related to $\ka_0$ in the next section, where we review the model derivation.

In practice, such a scattering problem is often encountered in stealth/tracking technology. The Radar Cross Section (RCS) measures the detectability of an object by a radar system. The RCS of cavities in an object (e.g., a jet engine's inlet ducts, exhaust nozzles) contributes the most to the overall RCS of an object. Therefore, accurate measurements of the RCS of these cavities are important. This is where numerical methods for the scattering problem come into play.

The Helmholtz equation with variable wavenumbers is challenging to solve numerically due to its sign indefinite (non-coercive) standard weak formulation and the pollution effect.
Define
\be \label{H}
\mathcal{H}(\Omega):=\{u\in H^{1}(\Omega):u=0 \text{ on } \partial \Omega \backslash \Gamma\}.
\ee
The weak formulation of the model problem in \eqref{cavity:model} is to find $u \in \mathcal{H}(\Omega)$ such that
\be \label{cavity:weak}
a(u,v):=\langle \nabla u, \nabla v\rangle_{\Omega} - \langle \ka^2 u,v \rangle_{\Omega} - \langle \mathcal{T}(u), v \rangle_{\Gamma} = \langle g, v \rangle_{\Gamma} - \langle f, v\rangle_{\Omega}, \quad \forall  v \in \mathcal{H}(\Omega).
\ee
The existence and uniqueness of the solution to \eqref{cavity:weak} have been studied in \cite[Theorem 4.1]{ABW02}. Relevant wavenumber-explicit stability
bounds have also been derived in
\cite{BY16, BYZ12, LMS13}, and in \cite{DLS15, HM22} with the non-local boundary operator approximated by the first-order absorbing boundary condition.
As the wavenumber $\ka$ gets larger, the solution becomes more oscillating and the mesh size requirement becomes exponentially demanding. Hence, the linear system associated with the discretization is often huge and ill-conditioned. Iterative schemes are usually preferred over direct solvers due to the expensive computational cost of the latter. It has been shown that high-order schemes are better in tackling the pollution effect (e.g, \cite{MS11}). Various high-order finite difference, Galerkin, and spectral methods have been proposed in \cite{BS05, DSZ13, HMW21, HMS16, LMS13, ZQT11}. In this paper, we are interested in using a wavelet basis to numerically solve \eqref{cavity:model} due to the following advantages. Our wavelet bases have high approximation orders, which help in alleviating the pollution effect. They produce a sparse coefficient matrix, which in a sense is more well-conditioned than that of the FEM. The sparsity aids in the efficient storage of the coefficient matrix, while the well-conditioned linear system results in a much fewer number of iterations needed for iterative schemes to be within a tolerance level.
	
	\subsection{Wavelets in the Sobolev space $\mathcal{H}(\Omega)$ with $\Omega=(0,1)^2$}
Let $\phi:=\{\phi^{1}, \dots, \phi^{r}\}^{\mathsf{T}}$ and $\psi:=\{\psi^{1}, \dots, \psi^{s}\}^{\mathsf{T}}$ be in $\Lp{2}$ of square integrable functions. For $J_0\in \Z$, define a wavelet system by
	\begin{equation}\label{Bphipsi}
		 \mathcal{B}_{J_0}(\phi;\psi):=  \left\{\phi^{\ell}_{J_0;k}: k \in \mathbb{Z},  \ell =1,\ldots,r \right\}
		\cup \left\{\psi^{\ell}_{j;k}: j \ge J_0, k \in \mathbb{Z}, \ell=1,\ldots, s \right\},
	\end{equation}
	where $\phi^{\ell}_{J_0;k}:=2^{J_0/2} \phi^{\ell}(2^{J_0}\cdot-k)$ and $\psi^{\ell}_{j;k}:=2^{j/2} \psi^{\ell}(2^{j}\cdot-k)$. Recall that
$\mathcal{B}_{J_0}(\phi;\psi)$ is \emph{a Riesz basis for $\Lp{2}$} if
(1) the linear span of $\mathcal{B}_{J_0}(\phi;\psi)$ is dense in $\Lp{2}$, and (2) there exist $C_{1}, C_{2} >0$ such that
	\begin{equation}\label{Riesz:L2}
	C_1 \sum_{\eta \in \mathcal{B}_{J_0}(\phi;\psi)} |c_{\eta}|^2 \le \Big\| \sum_{\eta \in \mathcal{B}_{J_0}(\phi;\psi)} c_{\eta} \eta \Big\|^2_{L^2(\R)} \le C_2 \sum_{\eta \in \mathcal{B}_{J_0}(\phi;\psi)} |c_{\eta}|^2
	\end{equation}
	for all finitely supported sequences $\{c_\eta\}_{\eta \in \mathcal{B}_{J_0}(\phi;\psi)}$. It is known (e.g., \cite[Theorem~6]{han12}) that \eqref{Riesz:L2} holds for some $J_0\in \Z$ if and only if \eqref{Riesz:L2} holds for all $J_0\in \Z$ with the same positive constants $C_1$ and $C_2$.
Consequently, we say that $\{\phi;\psi\}$ is \emph{a Riesz multiwavelet in $\Lp{2}$} if $\mathcal{B}_{0}(\phi;\psi)$ is a Riesz basis for $\Lp{2}$.
Let $\tilde{\phi}:=\{\tilde{\phi}^{1}, \dots, \tilde{\phi}^{r}\}^{\mathsf{T}}$ and $\tilde{\psi}:=\{\tilde{\psi}^{1}, \dots, \tilde{\psi}^{s}\}^{\mathsf{T}}$ be in $\Lp{2}$. We call $(\{\tilde{\phi};\tilde{\psi}\},\{\phi;\psi\})$ \emph{a biorthogonal multiwavelet} in $\Lp{2}$ if $(\mathcal{B}_{0}(\tilde{\phi};\tilde{\psi}), \mathcal{B}_{0}(\phi;\psi))$ is a biorthogonal basis in $\Lp{2}$, i.e.,
(1) $\mathcal{B}_{0}(\tilde{\phi};\tilde{\psi})$ and $\mathcal{B}_{0}(\phi;\psi)$ are Riesz bases in $\Lp{2}$, and (2) $\mathcal{B}_{0}(\tilde{\phi};\tilde{\psi})$ and $\mathcal{B}_{0}(\phi;\psi)$ are biorthogonal to each other.
	%
We say that the wavelet function $\psi$ has \emph{$m$ vanishing moments} if $\int_\R x^j \psi(x)dx=0$ for all $j=0,\ldots,m-1$. Furthermore, we define $\vmo(\psi):=m$ with $m$ being the largest of such an integer.

Many problems such as numerical PDEs and image processing are defined on bounded domains such as $(0,1)^d$. To use wavelets for such problems, we have to adapt biorthogonal wavelet bases on the real line to a bounded interval.
See \cref{sec:1DRiesz} for the literature review on this topic. Recently, \cite[Section 4]{HM21a} developed a systematic direct approach to adapt any univariate compactly supported biorthogonal multiwavelet $(\{\tilde{\phi};\tilde{\psi}\},\{\phi;\psi\})$
from the real line $\R$ to the interval $\cI:=(0,1)$ with or without any homogeneous boundary conditions. The direct approach presented in \cite[Section 4]{HM21a} constructs all possible locally supported biorthogonal bases $(\tilde{\mathcal{B}}, \mathcal{B})$ in $\LpI{2}$, where
\be \label{Binterval}
\mathcal{B}:=\Phi_{J_0}\cup \cup_{j=J_0}^\infty \Psi_j \subseteq \LpI{2},
\qquad \tilde{\mathcal{B}}:=\tilde{\Phi}_{J_0}\cup \cup_{j=J_0}^\infty \tilde{\Psi}_j \subseteq \LpI{2}
\ee
with $J_0\in \N$ being the coarsest resolution level and
\be \label{Bx:PhiPsi}
\begin{aligned}
	&\Phi_{J_0}=\{\phi^{L}_{J_0;0}\}\cup
			\{\phi_{J_0;k} \setsp n_{l,\phi} \le k\le 2^{J_0}-n_{h,\phi}\}\cup \{ \phi^{R}_{J_0;2^{J_0}-1}\},\\
			 &\Psi_{j}=\{\psi^{L}_{j;0}\}\cup
			\{\psi_{j;k} \setsp n_{l,\psi} \le k\le 2^{j}-n_{h,\psi}\}\cup \{ \psi^{R}_{j;2^{j}-1}\},\quad j\ge J_0,
		\end{aligned}
		\ee
		where the boundary refinable functions $\phi^L, \phi^R$ and the boundary wavelets $\psi^L, \psi^R$ are finite subsets of functions in $\LpI{2}$. The set $\tilde{\mathcal{B}}$ is defined similarly by adding $\sim$ to all the elements in $\mathcal{B}$.
In addition to giving us all possible boundary wavelets with or without boundary conditions, all boundary wavelets $\psi^L,\psi^R$ obtained from the direct approach in \cite{HM21a} have the same order of vanishing moments as that of the wavelet $\psi$. More importantly, the calculation in \cite{HM21a} does not explicitly involve the duals in $\tilde{\mathcal{B}}$.

Because we shall develop tensor product wavelets for the Sobolev space $\mathcal{H}(\Omega)$ to numerically solve \eqref{cavity:weak},
we consider the subspaces $H^{1,x}(\cI)$ and $H^{1,y}(\cI)$ of $H^1(\cI)$, where $\cI:=(0,1)$ and
\begin{equation}\label{Hxy}
H^{1,x}(\cI):=\{ f\in H^1(\cI) \setsp f(0)=f(1)=0\},
\quad
H^{1,y}(\cI):=\{ f\in H^1(\cI) \setsp f(0)=0\}.
\end{equation}
%

We have the following result, which is proved in \cref{sec:appendix},
on Riesz wavelet bases in the Sobolev space $H^{1,x}(\cI)$ or $H^{1,y}(\cI)$ satisfying the homogeneous Dirichlet boundary conditions in \eqref{Hxy}.

\begin{theorem}\label{thm:H1}
Let $(\{\tilde{\phi};\tilde{\psi}\},\{\phi;\psi\})$ be any compactly supported biorthogonal (multi)wavelet in $\Lp{2}$ such that every entry of $\phi$ belongs to the Sobolev space $H^1(\R)$.
		Let $(\tilde{\mathcal{B}},\mathcal{B})$ in \eqref{Binterval} and \eqref{Bx:PhiPsi} be a biorthogonal wavelet basis in $\LpI{2}$ with $\cI:=(0,1)$ (e.g., constructed
		by the direct approach in \cite{HM21a})
		from the given biorthogonal wavelet $(\{\tilde{\phi};\tilde{\psi}\},\{\phi;\psi\})$ in $\Lp{2}$
		such that $\mathcal{B} \subseteq H^{1,x}(\cI)$.
		Then
		\be \label{BH1}
		\mathcal{B}_{H^{1,x}}:=[2^{-J_0} \Phi_{J_0}]\cup \cup_{j=0}^\infty [2^{-j}\Psi_j]
		\ee
(i.e., $\mathcal{B}_{H^{1,x}}$ is the $H^{1}(\cI)$-normalized version of $\mathcal{B}$)
		must be a Riesz basis of the Sobolev space $H^{1,x}(\cI)$, i.e.,
		there exist positive constants $C_1$ and $C_2$ such that every function $f\in H^{1,x}(\cI)$ has a decomposition
		\be \label{f:H1}
		f=\sum_{\alpha \in \Phi_{J_0}} c_\alpha 2^{-J_0} \alpha+\sum_{j=J_0}^\infty\sum_{\beta_j\in \Psi_j} c_{\beta_j} 2^{-j} \beta_j
		\ee
		with the above series absolutely converging in $H^{1,x}(\cI)$, and the coefficients $\{c_\alpha\}_{\alpha\in \Phi_{J_0}}\cup \{c_{\beta_j}\}_{\beta_j\in \Psi_j, j\ge J_0}$ satisfy the following norm equivalence property:
		\be \label{rz:H1}
		C_1 \Big(\sum_{\alpha \in \Phi_{J_0}} |c_\alpha|^2+\sum_{j=J_0}^\infty\sum_{\beta\in \Psi_j} |c_{\beta_j}|^2\Big)
		\le \| f \|^2_{H^1(\cI)}
		\le C_2 \Big(\sum_{\alpha \in \Phi_{J_0}} |c_\alpha|^2+\sum_{j=J_0}^\infty\sum_{\beta_j\in \Psi_j} |c_{\beta_j}|^2\Big),
		\ee
		where $\|f\|_{H^1(\cI)}^2:=\|f\|^2_{\LpI{2}}+\|f'\|^2_{\LpI{2}}$.
		The same conclusion holds if $\mathcal{B}\subseteq H^{1,y}(\cI)$ and $H^{1,x}$ is replaced by
 $H^{1,y}$.
	\end{theorem}

One common approach to handle PDEs in $\Omega=(0,1)^2$ such as the model problem in \eqref{cavity:model} is to form a 2D Riesz wavelet basis by taking the tensor product of 1D Riesz wavelets on bounded intervals.
Now we show that the tensor product of Riesz wavelets in $\LpI{2}$, after appropriately normalized, forms a 2D Riesz basis in $\mathcal{H}(\Omega)$.
For 1D functions $f_1$ and $f_2$, define $(f_1 \otimes f_2) (x,y) := f_1(x)f_2(y)$ for $x,y\in\R$. Furthermore, if $F_1,F_2$ are sets containing 1D functions, then $F_1 \otimes F_2 := \{f_1 \otimes f_2 : f_1 \in F_1, f_2 \in F_2\}$.

The following result is for tensor product wavelet bases in the Sobolev space $\mathcal{H}(\Omega)$ in \eqref{H} for the model problem in \eqref{cavity:model}. The proof is deferred to \cref{sec:appendix}.

\begin{theorem}\label{thm:H1:2D}
	Let $(\{\tilde{\phi};\tilde{\psi}\},\{\phi;\psi\})$ be any compactly supported biorthogonal (multi)wavelet in $\Lp{2}$ such that every entry of $\phi$ belongs to the Sobolev space $H^1(\R)$.
	Let $(\tilde{\mathcal{B}}^x,\mathcal{B}^x)$ and $(\tilde{\mathcal{B}}^y,\mathcal{B}^y)$  be  biorthogonal wavelets in $\LpI{2}$ with $\cI:=(0,1)$ (e.g., constructed by the approach in \cite{HM21a}) from the given biorthogonal wavelet $(\{\tilde{\phi};\tilde{\psi}\},\{\phi;\psi\})$ such that $\mathcal{B}^x:=\Phi^x_{J_0}\cup \cup_{j=J_0}^\infty \Psi^x_j\subseteq H^{1,x}(\cI)$ and $\mathcal{B}^y:=\Phi^y_{J_0}\cup \cup_{j=J_0}^\infty \Psi^y_j\subseteq H^{1,y}(\cI)$ which are similarly defined as in \eqref{Bx:PhiPsi}. Define $\mathcal{B}^{2D}:=\Phi^{2D}_{J_0}
	 \cup \cup_{j=J_0}^{\infty} \Psi^{2D}_{j}$ with
	\be \label{PhiPsi2D}
	\Phi^{2D}_{J_0}:=\{ \Phi^{x}_{J_0} \otimes \Phi^{y}_{J_0}\} \quad \text{and} \quad
	\Psi^{2D}_{j}:= \{\Phi^{x}_j \otimes \Psi^{y}_j,  \Psi^{x}_{j}  \otimes \Phi^{y}_{j}, \Psi^{x}_j \otimes \Psi^{y}_j\},
	\ee
and define $\tilde{\mathcal{B}}^{2D}$ similarly using the dual functions. Then $(\tilde{\mathcal{B}}^{2D}, \mathcal{B}^{2D})$ is a biorthogonal wavelet basis in $\LpO{2}$ with $\Omega:=(0,1)^2$ and
\be \label{BH}
\mathcal{B}_{\mathcal{H}}:=[2^{-J_0}\Phi^{2D}_{J_0}]
	 \cup \cup_{j=J_0}^{\infty} [2^{-j} \Psi^{2D}_{j}]
\ee
(i.e., $\mathcal{B}_{\mathcal{H}}$ is the $H^1(\Omega)$-normalized version of $\mathcal{B}^{2D}$) must be a Riesz basis of the Sobolev space $\mathcal{H}(\Omega)$, i.e., there exist positive constants $C_1$ and $C_2$ such that every function $f\in \mathcal{H}(\Omega)$ has a decomposition
	\be \label{f:H1:2D}
	f=\sum_{\alpha \in \Phi^{2D}_{J_0}} c_\alpha 2^{-J_0} \alpha+\sum_{j=J_0}^\infty\sum_{\beta_j\in \Psi^{2D}_{j}} c_{\beta_j} 2^{-j} \beta_j,
	\ee
	which converges absolutely in $\mathcal{H}(\Omega)$ and whose coefficients $\{c_\alpha\}_{\alpha \in \Phi^{2D}_{J_0}}\cup \{c_{\beta_j}\}_{\beta_j\in \Psi^{2D}_{j}, j\ge J_0}$ satisfy
	\begin{align} \label{rz:H1:2D}
	C_1 \Big(\sum_{\alpha \in  \Phi^{2D}_{J_0}} |c_\alpha|^2+\sum_{j=J_0}^\infty\sum_{\beta_j\in \Psi^{2D}_{j}} |c_{\beta_j}|^2\Big)
	\le \| f \|^2_{H^1(\Omega)}
	\le C_2 \Big(\sum_{\alpha \in  \Phi^{2D}_{J_0}} |c_\alpha|^2+\sum_{j=J_0}^\infty\sum_{\beta_j\in \Psi^{2D}_{j}} |c_{\beta_j}|^2\Big),
	\end{align}
	where $\|f\|^2_{H^1(\Omega)}:=\|f\|^2_{\LpO{2}}+
	\|\tfrac{\partial}{\partial x} f\|^2_{\LpO{2}}+\|\tfrac{\partial}{\partial y} f\|^2_{\LpO{2}}$.
\end{theorem}
	
\cref{thm:H1,thm:H1:2D} can be directly applied to all our constructed spline Riesz wavelets in $\LpI{2}$  in \cref{sec:1DRiesz}.
Even though we only consider the Sobolev space $\mathcal{H}(\Omega)$ in \cref{thm:H1:2D}, a similar proof idea can be applied to the Sobolev space $H^{m}(\Omega)$ with $\Omega=(0,1)^d$, where $d, m \in \N$.


	\subsection{Advantages and shortcomings of wavelets}
	To numerically solve 2D PDEs, the FEM uses $\Phi^{2D}_J$ with a fine scale level $J \ge J_0$ as test and trial functions. Meanwhile, our wavelet Galerkin method uses $\mathcal{B}^{2D}_{J_0,J}:= \Phi^{2D}_{J_0}\cup \cup_{j=J_0}^{J-1} \Psi^{2D}_j$. Let $N$ be the number of elements in $\mathcal{B}^{2D}_{J_0,J}$. Then, a numerical solution $u_N:=\sum_{v\in \mathcal{B}^{2D}_{J_0,J}} c_v v$ can be obtained by solving $Ac=b$, where $A$ is the $N \times N$ coefficient matrix coming from the discretization (the weak formulation), $b$ is the $N \times 1$ vector containing inner products of the source term and the boundary condition with $v\in \mathcal{B}^{2D}_{J_0,J}$, and $c:=(c_v)_{v\in \mathcal{B}^{2D}_{J_0,J}}$.
	We refer interested readers to \cite{cer19, CS12, cohbook, dah97, DS10, K01, LC13, S09, U09} and references therein for a review of wavelet-based methods in numerically solving PDEs. Note that $\text{span}(\Phi^{2D}_J) = \text{span}(\mathcal{B}^{2D}_{J_0,J})$ and therefore, the numerical solution $u_N$ is the same if $\mathcal{B}^{2D}_{J_0,J}$ is replaced by $\Phi^{2D}_J$ (but then the condition numbers of the $N\times N$ matrix $A$ increases exponentially as $N$ increases due to diminishing smallest singular values of $A$).
Some advantages of using spline wavelets in numerical PDEs are their analytic expressions and sparsity.
In order to effectively solve a linear system with an $N\times N$ coefficient matrix $A$, using an iterative scheme, the coefficient matrices $A$ using wavelets have the following two key properties:
	\begin{enumerate}
		\item[(i)] The condition numbers of the matrix $A$ are relatively small and uniformly bounded. In particular, the smallest singular values of $A$ are uniformly bounded away from zero.
   		 \item[(ii)]
   The $N\times N$ matrices $A$ have sparsity (i.e., the numbers of all nonzero entries of $A$ are $\bo(N)$ for constant wavenumbers and $\bo(N\log N)$ for variable wavenumbers) and certain desirable/exploitable structures for efficient implementation.
	\end{enumerate}

The desired property in (i) is achieved by \cref{thm:H1:2D} by using Riesz wavelet bases in $\mathcal{H}(\Omega)$. Item (i) remains a key consideration even if we use direct solvers. It is well known (e.g., see \cite{cohbook,dah97,DS10}) that the optimal sparsity in (ii) for constant wavenumbers can be achieved by considering a spline refinable vector function $\phi$ such that $\phi$ is a piecewise polynomial of degree less than $m$ and its derived wavelet $\psi$ has at least order $m$ vanishing moments. Because $\mathcal{B}^{2D}_{J_0,J}$ has the refinable structure (e.g., see \cite{hanbook,HM21a}), the matrices $A$ with desired structures can be computed efficiently by fast multiwavelet transforms.
In this paper, we are particularly interested in constructing spline Riesz wavelets $\{\phi;\psi\}$ and adapting them to the interval $\mathcal{I}$ such that the spline refinable vector function $\phi$ is a piecewise polynomial of degree less than $m$ and the wavelet $\psi$ has at least order $m$ vanishing moments.

Despite the fact that the $\mathcal{H}(\Omega)$-normalized version of the wavelets $\mathcal{B}^{2D}_{J_0,J}$ has sparsity and uniformly bounded condition numbers (independent of $J$) in the Sobolev space $\mathcal{H}(\Omega)$, one shortcoming of wavelets is that their construction on a general domain $\Omega$ can be challenging partly due to the number of boundary elements. More precisely, because the supports of elements in $\mathcal{B}^{2D}_{J_0,J}$ vary from being highly localized for large $j$ to almost global for small $j$. Consequently, there are much more boundary elements in  $\mathcal{B}^{2D}_{J_0,J}$ touching $\partial \Omega$ than those in $\Phi^{2D}_J$. That is why we focus on domains where we can take the tensor product of 1D Riesz wavelets; e.g., a rectangular domain.
	
	\subsection{Main contributions of this paper.} We present a high-order wavelet Galerkin method to solve the model problem in \eqref{cavity:model}. First, we present several optimized B-spline scalar wavelets and spline multiwavelets on the interval $\cI:=(0,1)$, which can be used to numerically solve various PDEs. All spline wavelets presented in this paper are constructed by using our direct approach in \cite{HM21a}, which allows us to find all possible biorthogonal multiwavelets in $\LpI{2}$ from any compactly supported biorthogonal multiwavelets in $\Lp{2}$. Since all possible biorthogonal multiwavelets in $\LpI{2}$ can be found, we can obtain an optimized wavelet on $\cI$ with a simple structure that is well-conditioned. Constructing a 1D Riesz wavelet on an interval is not the only task. We also need to carefully optimize the boundary wavelets such that their structures remain simple and the coefficient matrix associated with the discretization of a problem is in a sense as well-conditioned as possible.
	
	Second, we provide self-contained proofs showing that all the constructed wavelets on $\cI$ form 1D Riesz wavelets in the appropriate Sobolev space; additionally, via the tensor product, they form 2D Riesz wavelets in the appropriate Sobolev space. In the literature (e.g. see \cite{dah96}), the Riesz basis property is only guaranteed under the assumption that both the Jackson and Bernstein inequalities for the wavelet system hold, which may not be easy to establish (particularly the Bernstein inequality). Our proof does not involve the Jackson and Bernstein inequalities. We provide a direct and relatively simple proof, which does not require any unnecessary conditions on the wavelet systems.
	
	Third, we apply our wavelet Galerkin method to several test problems, where the wavenumbers can be constant, continuously varying, and piecewise smooth (the latter can be considered as an example of the Helmholtz interface problem). Our experiments show that the wavelet coefficient matrices $A$ are in a sense much more well-conditioned than those of the FEM. The smallest singular values of wavelet matrices $A$ are uniformly bounded away from the zero instead of becoming arbitrarily small for FEM matrices as the matrix size increases. Compared to the FEM coefficient matrix, when an iterative scheme is applied to the wavelet coefficient matrices $A$, much fewer iterations are needed for the relative residuals to be within a tolerance level. For a given bounded variable wavenumber, the number of required iterations is practically independent of the size of the matrices $A$; i.e, the number of iterations is bounded above. In contrast, the number of required iterations for the FEM coefficient matrix doubles as the mesh size for each axis is halved. Spline multiwavelets generally have shorter supports compared to B-splines wavelets. The former has much fewer boundary wavelets and their structures are much simpler than those of B-spline wavelets.
	Thus, we favor the use of spline multiwavelets over B-spline wavelets. Finally, the refinability structure of our wavelet basis makes the implementation of our wavelet Galerkin method efficient.
	
	\subsection{Organization of this paper.} In \cref{sec:model}, we recall the derivation of the model problem in \eqref{cavity:model}. In \cref{sec:1DRiesz}, we present some optimized 1D Riesz wavelets on the interval $\mathcal{I}$. In \cref{sec:implement}, we discuss the implementation of our wavelet Galerkin method. In \cref{sec:exp}, we present our numerical experiments showcasing the performance of our wavelets. In \cref{sec:appendix}, we present the proofs of our main results.

	\section{Model Derivation}
	\label{sec:model}
	We recall the derivation of the model problem in \eqref{cavity:model} as explained in \cite{ABW02,BL14,BS05,DSZ13}. Several simplifying physical assumptions are needed. We assume that the cavity is embedded in an infinite ground plane. The ground plane and cavity walls are perfect electric conductors (PECs). The medium is non-magnetic with a constant permeability, $\mu$, and a constant permittivity, $\varepsilon$. Furthermore, we assume that no currents are present and the fields are source free. Let $E$ and $H$ respectively denote the total electric and magnetic fields. So far, our current setup can be modelled by the following Maxwell's equation with time dependence $e^{-\ia \omega t}$, where $\omega$ stands for the angular frequency
	\be \label{cavity:maxwell}
	\begin{aligned}
		& \nabla \times E - \ia \omega \mu H =0,\\
		& \nabla \times H + \ia \omega \varepsilon E=0.
	\end{aligned}
	\ee
	Since we assume that the ground plane and cavity walls are PECs, we equip the above problem with the boundary condition $\nv \times E=0$ on the surface of PECs, where $\nv$ is again the unit outward normal. We further assume that the medium and the cavity are invariant with respect to the $z$-axis. The cross-section of the cavity, denoted by $\Omega$, is rectangular.
	Meanwhile, 
	$\Gamma$ corresponds to the top of the cavity or the aperture. We restrict our attention to the transverse magnetic (TM) polarization. This means that the magnetic field is transverse/perpendicular to the $z$-axis; moreover, the total electric and magnetic fields take the form $E=(0,0,u(x,y))$ and $H=(H_x,H_y,0)$ for some functions $u(x,y)$, $H_x$, and $H_y$. Plugging these particular $E,H$ into \eqref{cavity:maxwell} and recalling the boundary condition, we obtain the 2D homogeneous Helmholtz equation defined on the cavity and the upper half space with the homogeneous Dirichlet boundary condition at the surface of PECs, and the scattered field satisfying the Sommerfeld's radiation boundary condition at infinity. By using the half-space Green's function with homogeneous Dirichlet boundary condition or the Fourier transform, we can introduce a non-local boundary condition on $\Gamma$ such that the previous unbounded problem is converted to a bounded problem. See \cref{fig:cavity} for an illustration.
	
	For the standard scattering problem, we want to determine the scattered field $u^{s}$ in the half space and the cavity given an incident plane wave $u^{inc}=e^{\ia \alpha x-\ia \beta (y-1)}$, where $\alpha=\ka_0 \sin(\theta)$, $\beta=\ka_0 \cos(\theta)$, and the incident angle $\theta \in (-\pi/2,\pi/2)$. In particular, $u^{s}=u-u^{inc}+e^{\ia \alpha x+\ia \beta (y-1)}$, where $u$ is found by solving the following problem
	\begin{align}
		\nonumber
		&\Delta u + \ka_0^2 \varepsilon_r u = 0 \quad \text{in} \quad \Omega,\\
		\label{model:scatter}
		&u=0 \quad \text{on} \quad \partial \Omega \backslash \Gamma,\\
		\nonumber
		&\frac{\partial u}{\partial \nv} = \mathcal{T}(u) - 2\ia \beta e^{\ia \alpha x} \quad \text{on} \quad \Gamma,
	\end{align}
	where $\varepsilon_r$ is the medium's relative permittivity and the non-local boundary operator $\mathcal{T}$ is defined in \eqref{cavity:Tu}. In the model problem \eqref{cavity:model}, we let $\ka^2 := \ka_0^2 \varepsilon_r$, and allow the source and boundary data to vary. For simplicity, we let $\Omega = (0,1)^2$ in our model problem and numerical experiments.
	\begin{figure}[htbp]
		\begin{center}
			\begin{tikzpicture}
				\draw[->] (0,2) -- (0.8,0.3);
				\draw[dashed] (1,2) -- (1,0.3);
				\draw (-2,0) -- (0,0) -- (0,-2) -- (2,-2) -- (2,0) -- (4,0);
				\draw[dashed] (0,0) -- (2,0);
				\node[] at (1,-0.2) {$\Gamma$};
				\node[] at (1,-1) {$\Omega$};
				\node[] at (0.7,1.4) {$\theta$};
				\draw[] (0.7,0.6) to[bend left=45] (0.97,0.75);
			\end{tikzpicture}
		\end{center}
		\caption{Geometry of the scattering from a rectangular cavity.}
		\label{fig:cavity}
	\end{figure}
	%
	
	\section{1D Locally Supported Spline Riesz Wavelets in $\LpI{2}$ with $\cI:=(0,1)$}
	\label{sec:1DRiesz}

	In this section, we define $\cI:=(0,1)$ and construct Riesz wavelets $\mathcal{B}^x\subseteq H^{1,x}(\cI)$ and $\mathcal{B}^y\subseteq H^{1,x}(\cI)$. Consequently, by \cref{thm:H1,thm:H1:2D}, we can obtain Riesz wavelets in suitable subspaces of $H^1(\Omega)$ to numerically solve PDEs in the domain $\Omega:=\cI^d=(0,1)^d$ with dimension $d$.

	 	
As we discussed in \cref{sec:intro}, the construction of 1D locally supported spline Riesz wavelets in the Sobolev space $\mathcal{H}(\Omega)$ has three major steps. First, one has to construct a compactly supported biorthogonal wavelet $(\{\tilde{\phi};\tilde{\psi}\},\{\phi;\psi\})$ in $\Lp{2}$ such that the condition number, i.e., the ratio $C_2/C_1$ in \eqref{Riesz:L2}, of $\mathcal{B}_{J_0}(\phi;\psi)$ is as small as possible. We first review a fundamental result on compactly supported biorthogonal wavelets in $\Lp{2}$. In what follows, the Fourier series of $u=\{u(k)\}_{k\in \Z}\in \lrs{0}{r}{s}$ is defined by $\wh{u}(\xi):=\sum_{k\in \Z} u(k) e^{- \ia k\xi}$ for $\xi\in \R$, which is an $r\times s$ matrix of $2\pi$-periodic trigonometric polynomials. Additionally, the Fourier transform of $f \in \Lp{1}$ is defined to be $\wh{f}(\xi) := \int_{\R} f(x) e^{-ix \xi} dx$, $\xi \in \R$ and is extended naturally to functions in $\Lp{2}$.

	\begin{theorem} \label{thm:bw} (\cite[Theorem~6.4.6]{hanbook} and \cite[Theorem~7]{han12})
		Let $\phi,\tilde{\phi}$ be $r\times 1$ vectors of compactly supported distributions and $\psi,\tilde{\psi}$ be $s\times 1$ vectors of compactly supported distributions on $\R$. Then $(\{\tilde{\phi};\tilde{\psi}\},\{\phi;\psi\})$ is a biorthogonal wavelet in $\Lp{2}$ if and only if the following are satisfied
		\begin{enumerate}
			\item[(1)] $\phi,\tilde{\phi}\in (\Lp{2})^r$ and $\ol{\wh{\phi}(0)}^\tp \wh{\tilde{\phi}}(0)=1$.
			\item[(2)] $\phi$ and $\tilde{\phi}$ are biorthogonal to each other:
$\la \phi,\tilde{\phi}\ra=I_r$
and $\la \phi,\tilde{\phi}(\cdot-k)\ra=0$ for all $k\in \Z\setminus \{0\}$.
			\item[(3)] There exist low-pass filters $a,\tilde{a}\in \lrs{0}{r}{r}$ and high-pass filters
			$b,\tilde{b}\in \lrs{0}{s}{r}$ such that
			\begin{align*}
				&\phi=2\sum_{k\in \Z} a(k)\phi(2\cdot-k),\qquad
				\psi=2\sum_{k\in \Z} b(k)\phi(2\cdot-k),
				\\
				 &\tilde{\phi}=2\sum_{k\in \Z}\tilde{a}(k)
				 \tilde{\phi}(2\cdot-k),\qquad
				 \tilde{\psi}=2\sum_{k\in \Z} \tilde{b}(k)
				 \tilde{\phi}(2\cdot-k),
			\end{align*}
			and $(\{\tilde{a};\tilde{b}\},\{a;b\})$ is a biorthogonal wavelet filter bank, i.e., $s=r$
and
			\[
			\left [ \begin{matrix}
				\wh{\tilde{a}}(\xi) &\wh{\tilde{a}}(\xi+\pi)\\
				\wh{\tilde{b}}(\xi) &\wh{\tilde{b}}(\xi+\pi)
			\end{matrix}\right]
			\left[ \begin{matrix}
				\ol{\wh{a}(\xi)}^\tp &\ol{\wh{b}(\xi)}^\tp\\
				 \ol{\wh{a}(\xi+\pi)}^\tp &\ol{\wh{b}(\xi+\pi)}^\tp
			\end{matrix}\right]
			=I_{2r}, \qquad \xi\in \R.
			\]
			\item[(4)] $\wh{\psi}(0)=0$ and $\wh{\tilde{\psi}}(0)=0$, i.e., every element in $\psi$ and $\tilde{\psi}$ has at least one vanishing moment.
		\end{enumerate}
	\end{theorem}
	
A filter $a\in \lrs{0}{r}{r}$ has \emph{order $m$ sum rules with a (moment) matching filter $\vgu\in \lrs{0}{1}{r}$} if
	\[
	 [\wh{\vgu}(2\cdot)\wh{a}]^{(j)}(0)=\wh{\vgu}^{(j)}(0)
	\quad \mbox{and}\quad [\wh{\vgu}(2\cdot)\wh{a}(\cdot+\pi)]^{(j)}(0)=0,\qquad \forall\; j=0,\ldots,m-1
	\]
with $\wh{\vgu}(0)\ne 0$.
	More specifically, we define $\sr(a)=m$ with $m$ being the largest such nonnegative integer. The sum rule order of a filter is closely related to the approximation order of its corresponding refinable function. It is also well known (e.g., \cite{cdf92,hanbook}) that $\vmo(\psi)=\sr(\tilde{a})$ and $\vmo(\tilde{\psi})=\sr(a)$. Moreover, all
finitely supported dual masks $\tilde{a}$ of a given primal mask $a$ with $\sr(\tilde{a})\ge m$ can be constructed via the coset-by-coset (CBC) algorithm in \cite[page~33]{han01} (also see \cite[Algorithm~6.5.2]{hanbook}).


Next, we have to adapt a given biorthogonal wavelet
$(\{\tilde{\phi};\tilde{\psi}\},\{\phi;\psi\})$ in $\Lp{2}$ into a biorthogonal wavelet $(\tilde{\mathcal{B}},\mathcal{B})$ in $\LpI{2}$ as in \eqref{Binterval} and \eqref{Bx:PhiPsi} such that the condition number of $\mathcal{B}$ is not much larger than that of $\mathcal{B}_{J_0}(\phi;\psi)$ and the number of boundary wavelets $\psi^{L},\psi^{R}$ is as small as possible for simple structure and efficient implementation. Finally, \cref{thm:H1,thm:H1:2D} can be applied to obtained 1D and 2D Riesz wavelets for various Sobolev spaces for the numerical solutions of PDEs.

For comprehensive discussions on existing constructions of wavelets on a bounded interval, we refer interested readers to \cite{cer19,HM21a}. Compactly supported biorthogonal B-spline wavelets based on \cite{cdf92} were adapted to $\cI$ in the pivotal study \cite{dku99}. Subsequent studies were done to address its shortcomings (see \cite{cer19,HM21a}). Some infinitely supported B-spline wavelets have also been constructed on $\cI$ (see \cite{cer19} and references therein). An example of compactly supported biorthogonal spline multiwavelets was constructed on $\cI$ in the key study \cite{dhjk00}. For compactly supported biorthogonal wavelets with symmetry, we can use the approach in \cite{hm18} to construct wavelets on $\cI$, but the boundary wavelets have reduced vanishing moments.  Many constructions (e.g., \cite{dku99, hm18}) are special cases of \cite{HM21a}.

	In the following examples, we define $f_{j;k}:=2^{j/2}f(2^{j}\cdot-k)$. Given a refinable function $\phi$, define $\sm(\phi):=\sup\{\tau \in \R: \phi \in (H^{\tau}(\R))^r\}$. We include the technical quantity $\sm(a)$, whose definition can be found in \cite[(5.6.44)]{hanbook}, and is closely related to the smoothness of a refinable vector function $\phi$ via the inequality $\sm(\phi) \ge \sm(a)$. We define $\text{fsupp}(\phi)$ to be the shortest interval with integer endpoints such that $\phi$ vanishes outside $\text{fsupp}(\phi)$. The superscript $bc$ in the left boundary wavelet $\psi^{L,bc}$ means $\psi^{L,bc}$ satisfies the homogeneous Dirichlet boundary condition at the left endpoint $0$; i.e., $\psi^{L,bc}(0)=0$. Since $\psi^{R,bc}=\psi^{L,bc}(1-\cdot)$, we have $\psi^{R,bc}(1)=0$. The same notation holds for $\phi^{L,bc}$ and $\phi^{R,bc}$.
	
	We do not include any information on the dual boundary refinable functions and wavelets in $\tilde{\mathcal{B}}^x$ and $\tilde{\mathcal{B}}^y$ for all our examples, largely because
both $\tilde{\mathcal{B}}^x$ and $\tilde{\mathcal{B}}^y$ do not play explicit roles in the Galerkin scheme and are uniquely determined/recovered by their primal Riesz wavelet bases $\mathcal{B}^x$ and $\mathcal{B}^y$.

	\subsection{Scalar B-spline wavelets on $\cI$}
	We present three B-spline wavelets on $\cI$.
	
	\begin{example}\label{ex:B2}
	\normalfont
	Consider the scalar biorthogonal wavelet $(\{\tilde{\phi};\tilde{\psi}\},\{\phi;\psi\})$ in \cite{cdf92} (see also \cite[Example 7.5]{HM21a}) with $\wh{\phi}(0)=\wh{\tilde{\phi}}(0)=1$ and a biorthogonal wavelet filter bank $(\{\tilde{a};\tilde{b}\},\{a;b\})$  given by
	\begin{align*}
		 a=&\left\{\tfrac{1}{4},\tfrac{1}{2},\tfrac{1}{4}\right\}_{[-1,1]}, \quad b=\left\{-\tfrac{1}{8},-\tfrac{1}{4},\tfrac{3}{4},-\tfrac{1}{4},-\tfrac{1}{8}\right\}_{[-1,3]},\\
		\tilde{a}=&\left\{-\tfrac{1}{8}, \tfrac{1}{4}, \tfrac{3}{4}, \tfrac{1}{4}, -\tfrac{1}{8} \right\}_{[-2,2]}, \quad \tilde{b}=\left\{-\tfrac{1}{4}, \tfrac{1}{2}, -\tfrac{1}{4}\right\}_{[0,2]}.
	\end{align*}
	The analytic expression of the hat function $\phi$ is $\phi:=(x+1)\chi_{[-1,0)} + (1-x)\chi_{[0,1]}$, which has been widely used in numerical PDEs and approximation theory. Note that $\text{fsupp}(\phi)=[-1,1]$, $\text{fsupp}(\psi)=\text{fsupp}(\tilde{\psi})=[-1,2]$, and $\text{fsupp}(\tilde{\phi})=[-2,2]$. Moreover, $\sm(a)=1.5$, $\sm(\tilde{a}) \approx 0.440765$, and $\sr(a)=\sr(\tilde{a})=2$. Let $\phi^{L}:=\phi\chi_{[0,\infty)}=\phi^{L}(2\cdot) + \tfrac{1}{2} \phi(2\cdot-1)$. The direct approach in \cite[Theorem 4.2]{HM21a} yields
	\[
	\psi^{L} = \phi^{L}(2\cdot) - \tfrac{5}{6} \phi(2\cdot-1) + \tfrac{1}{3} \phi(2\cdot -2)
	\quad \text{and}
	\quad
	\psi^{L,bc} =  \tfrac{1}{2} \phi(2\cdot-1) - \phi(2\cdot-2) + \tfrac{1}{2}\phi(2\cdot-3).
	\]
	For $J_0 \ge 2$ and $j \ge J_0$, define
	\[
	\Phi^{x}_{J_0} := \{\phi_{J_0;k}: 1\le k \le 2^{J_0}-1\},
	\quad
	\Psi^{x}_{j} := \{\psi^{L,bc}_{j;0}\} \cup \{\psi_{j;k}:1\le k \le 2^{j}-2\} \cup \{\psi^{R,bc}_{j;2^j-1}\},
	\]
	where $\psi^{R,bc}=\psi^{L,bc}(1-\cdot)$, and
	\[
	\Phi^{y}_{J_0} := \Phi^{x}_{J_0} \cup \{\phi^{R}_{J_0;2^{J_0}-1}\}, \quad
	\Psi^{y}_{j} := \left(\Psi^{x}_{j} \backslash \{\psi^{R,bc}_{j;2^j-1}\}\right) \cup \{\psi^{R}_{j;2^j-1}\},
	\]
	where $\phi^{R}=\phi^{L}(1-\cdot)$ and $\psi^{R}=\psi^{L}(1-\cdot)$. Then, $\mathcal{B}^x:=\Phi^{x}_{J_0} \cup \{\Psi_{j}^x : j\ge J_0\}$ and $\mathcal{B}^{y}:=\Phi^{y}_{J_0} \cup \{\Psi_{j}^y : j\ge J_0\}$ with $J_0 \ge 2$ are Riesz wavelets in $\LpI{2}$.
	See \cref{fig:hat} for their generators. 
	\begin{figure}[htbp]
		\centering
		 \begin{subfigure}[b]{0.24\textwidth} \includegraphics[width=\textwidth]{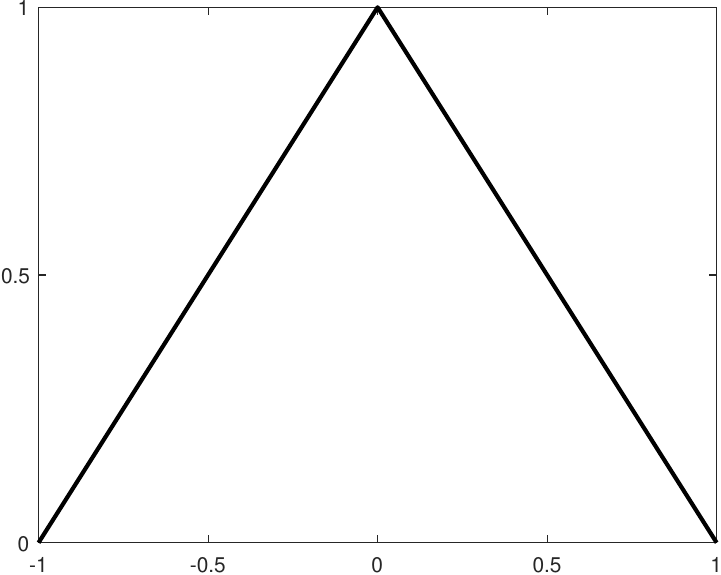}
			\caption{$\phi$}
		\end{subfigure}
		 \begin{subfigure}[b]{0.24\textwidth} \includegraphics[width=\textwidth]{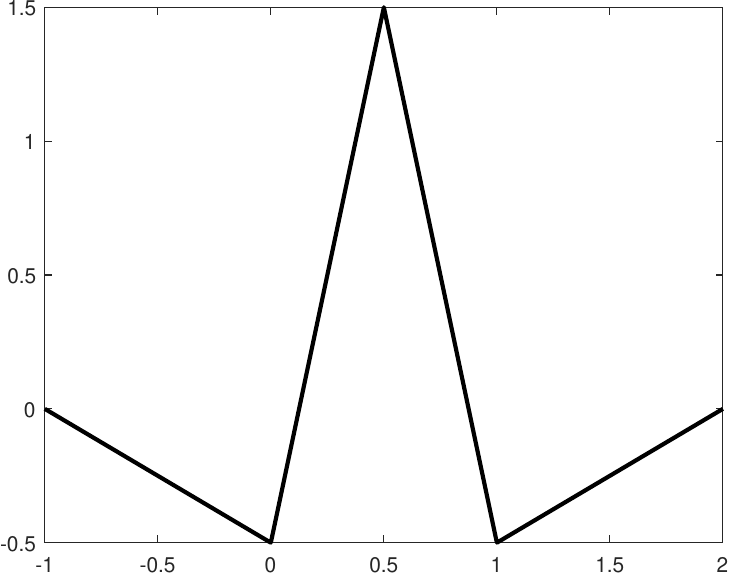}
			\caption{$\psi$}
		\end{subfigure}
		 \begin{subfigure}[b]{0.24\textwidth} \includegraphics[width=\textwidth]{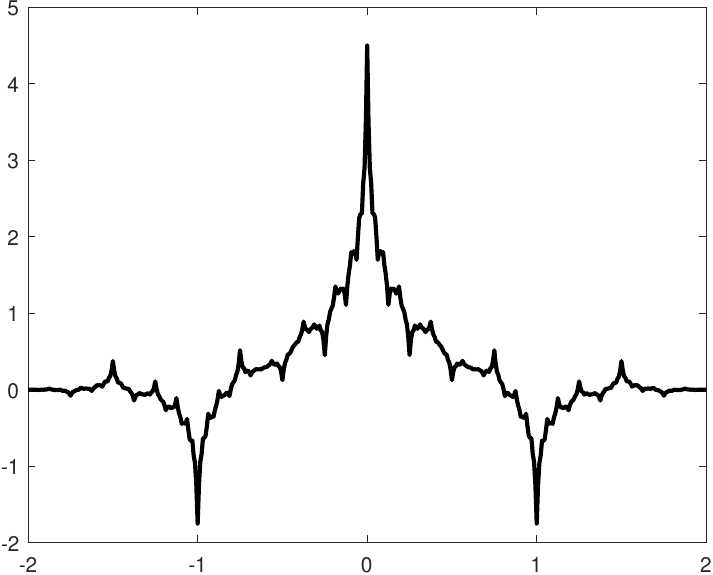}
			\caption{$\tilde{\phi}$}
		\end{subfigure}
		 \begin{subfigure}[b]{0.24\textwidth} \includegraphics[width=\textwidth]{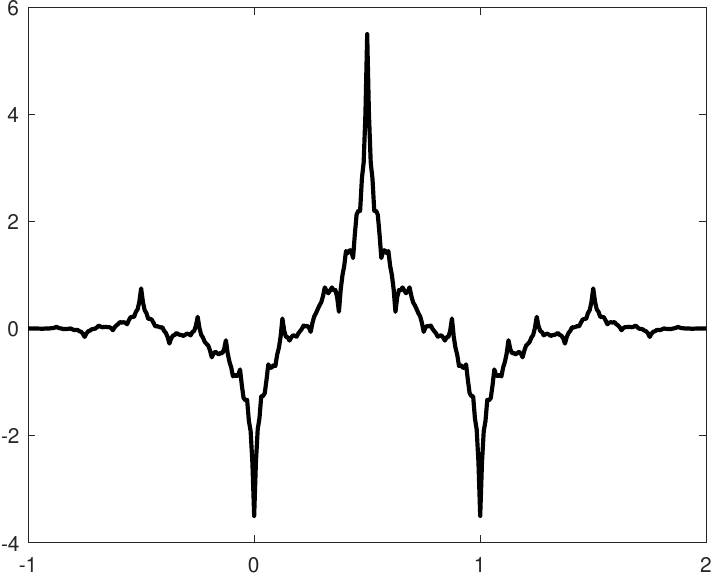}
			\caption{$\tilde{\psi}$}
		\end{subfigure}
		 \begin{subfigure}[b]{0.24\textwidth}
			 \includegraphics[width=\textwidth]{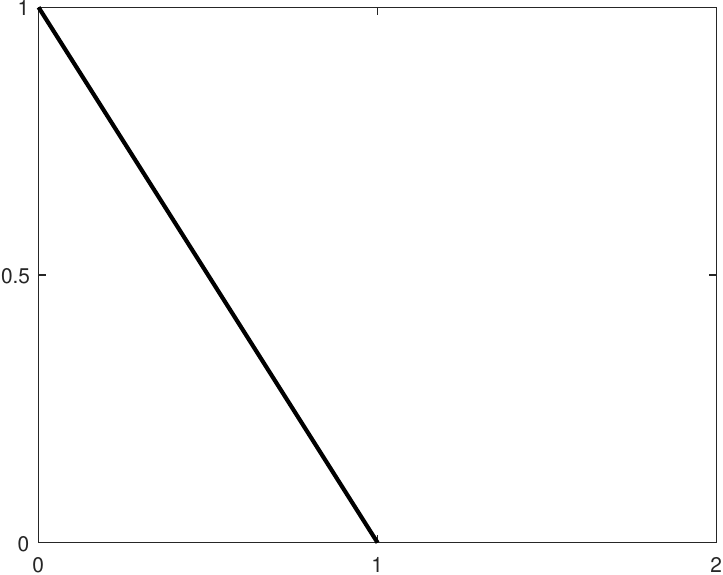}
			\caption{$\phi^{L}$}
		\end{subfigure}
		 \begin{subfigure}[b]{0.24\textwidth}
			 \includegraphics[width=\textwidth]{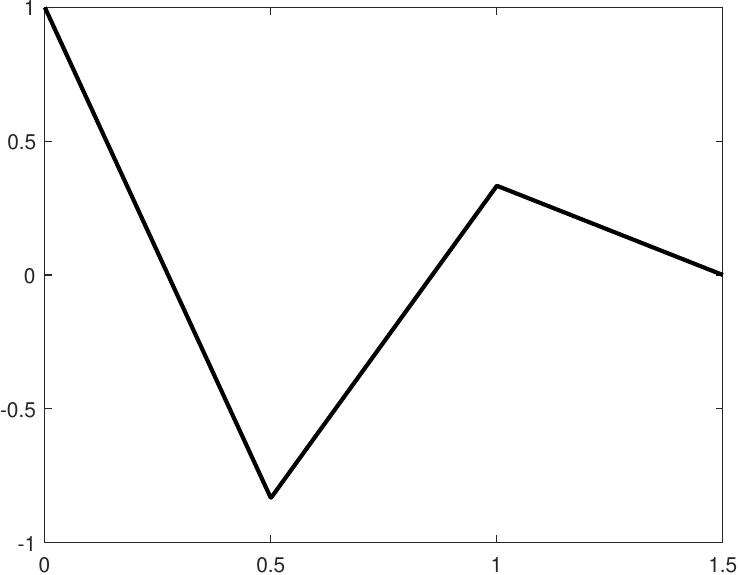}
			\caption{$\psi^{L}$}
		\end{subfigure}
		 \begin{subfigure}[b]{0.24\textwidth}
			 \includegraphics[width=\textwidth]{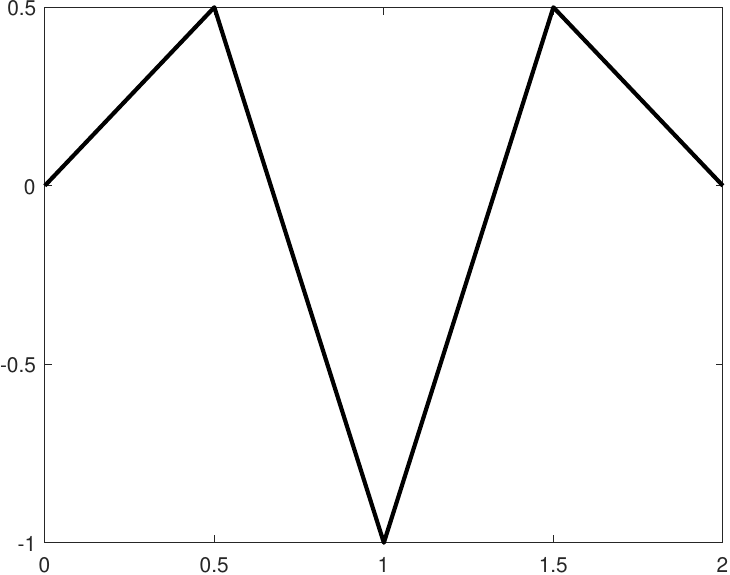}
			\caption{$\psi^{L,bc}$}
		\end{subfigure}
		\caption{The generators of Riesz wavelets $\Phi^{x}_{J_0} \cup \{\Psi_{j}^x : j\ge J_0\}$ and $\Phi^{y}_{J_0} \cup \{\Psi_{j}^y : j\ge J_0\}$ of $\LpI{2}$ 
			for $J_0 \ge 2$.}
		\label{fig:hat}
	\end{figure}	
	\end{example}
		
 	\begin{example} \label{ex:B3}
 	\normalfont
 	Consider a biorthogonal wavelet $(\{\tilde{\phi};\tilde{\psi}\},\{\phi;\psi\})$ in $\Lp{2}$ with $\wh{\phi}(0) = \wh{\tilde{\phi}}(0) =1$ in \cite{cdf92} and a biorthogonal wavelet filter bank $(\{\tilde{a};\tilde{b}\},\{a;b\})$ given by
 	\begin{align*}
 		 a & = \{\tfrac{1}{8},\tfrac{3}{8},\tfrac{3}{8},\tfrac{1}{8}\}_{[-1,2]}, \quad b=\{\tfrac{3}{64},\tfrac{9}{64},-\tfrac{7}{64},-\tfrac{45}{64},\tfrac{45}{64},\tfrac{7}{64},-\tfrac{9}{64},-\tfrac{3}{64}\}_{[-3,4]},\\
 		 \tilde{a} & = \{\tfrac{3}{64},-\tfrac{9}{64},-\tfrac{7}{64},\tfrac{45}{64},\tfrac{45}{64},-\tfrac{7}{64},-\tfrac{9}{64},\tfrac{3}{64}\}_{[-3,4]}, \quad
 		 \tilde{b} = \{\tfrac{1}{8},-\tfrac{3}{8},\tfrac{3}{8},-\tfrac{1}{8}\}_{[-1,2]}.	 
 	\end{align*}
 	The analytic expression of $\phi$ is $\phi:=\tfrac{1}{2}(x+1)^2 \chi_{[-1,0)} + \tfrac{1}{2}(-2x^2 +2x + 1) \chi_{[0,1)} + \tfrac{1}{2}(-2+x)^2 \chi_{[1,2]}$. Note that $\text{fsupp}(\phi) = [-1,2]$, $\text{fsupp}(\psi)=\text{fsupp}(\tilde{\psi})=[-2,3]$, and $\text{fsupp}(\tilde{\phi})=[-3,4]$. Furthermore, $\sm(a)=2.5$, $\sm(\tilde{a}) \approx 0.175132$, and $\sr(a)=\sr(\tilde{a})=3$. This implies $\phi \in H^{2}(\R)$ and $\tilde{\phi} \in \Lp{2}$. Let $\phi^{L}:=(\phi(\cdot+1) + \phi)\chi_{[0,\infty)} = \phi^{L} + \tfrac{3}{4} \phi(2\cdot -1) + \tfrac{1}{4} \phi(2\cdot -2)$ and $\phi^{L,bc} := \tfrac{1}{2}(-\phi(\cdot + 1) +\phi)\chi_{[0,\infty)} = \tfrac{1}{2} \phi^{L,bc}(2\cdot) + \tfrac{3}{8} \phi(2\cdot -1) + \tfrac{1}{8} \phi(2\cdot -2)$. The direct approach in \cite[Theorem 4.2]{HM21a} yields
 	\begin{align*}
 	\psi^{L} & = \phi^{L}(2\cdot) - \tfrac{11}{4} \phi(2\cdot -1) + \tfrac{31}{12} \phi(2\cdot -2) - \tfrac{5}{6} \phi(2\cdot -3),\\
 	\psi^{L,bc1} & = 2\phi^{L,bc}(2\cdot) - \tfrac{47}{30} \phi(2\cdot -1) + \tfrac{13}{10} \phi(2\cdot -2) - \tfrac{2}{5} \phi(2\cdot -3),\\
 	\psi^{L,bc2} & = 2\phi(2\cdot -1)  - 6 \phi(2\cdot -2) + 6 \phi(2\cdot-3) - 2\phi(2\cdot-4).
 	\end{align*}
	%
	For $J_0 \ge 2$ and $j \ge J_0$, define
	\begin{align*}
	& \Phi^{x}_{J_0} := \{\phi^{L,bc}_{J_0;0}\} \cup \{\phi_{j;k}:1\le k \le 2^{J_0}-2\} \cup \{\phi^{R,bc}_{J_0;2^{J_0}-1}\}, \quad \text{and}\\
	& \Psi^{x}_{j} := \{\psi^{L,bc1}_{j;0},\psi^{L,bc2}_{j;0}\} \cup \{\psi_{j;k}: 2\le k \le 2^{j}-3\} \cup \{\psi^{R,bc1}_{j;2^j-1},\psi^{R,bc2}_{j;2^j-1}\},
	\end{align*}
	where $\phi^{R,bc} = \phi^{L,bc}(1-\cdot)$, $\psi^{R,bci} = \psi^{L,bci}(1-\cdot)$ for $i=1,2$, and
	\[
	\Phi^{y}_{J_0} := \Phi^{x}_{J_0} \cup \{\phi^{R}_{J_0;2^{J_0}-1}\}, \quad
	\Psi^{y}_{j} := \left(\Psi^{x}_j \backslash \{\psi^{R,bc2}_{j;2^j-1}\}\right) \cup \{\psi^{R}_{j;2^j-1}\},
	\]
    where $\phi^{R} = \phi^{L}(1-\cdot)$ and $\psi^{R} = \psi^{L}(1-\cdot)$. Then, $\mathcal{B}^{x}:=\Phi^{x}_{J_0} \cup \{\Psi_{j}^x : j\ge J_0\}$ and
    $\mathcal{B}^{y}:=\Phi^{y}_{J_0} \cup \{\Psi_{j}^y : j\ge J_0\}$ with $J_0 \ge 2$ are Riesz wavelets in $\LpI{2}$.
    See \cref{fig:B3} for their generators.

    \begin{figure}[htbp]
    	\centering
    	 \begin{subfigure}[b]{0.24\textwidth} \includegraphics[width=\textwidth]{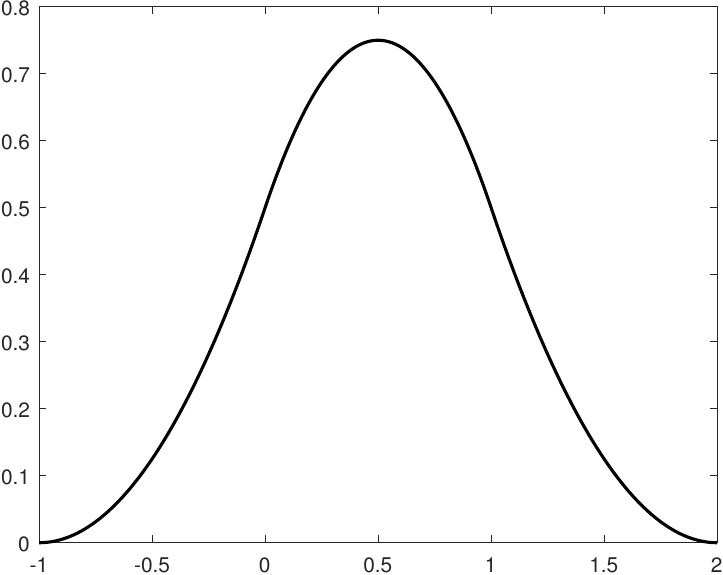}
    		\caption{$\phi$}
    	\end{subfigure}
    	 \begin{subfigure}[b]{0.24\textwidth} \includegraphics[width=\textwidth]{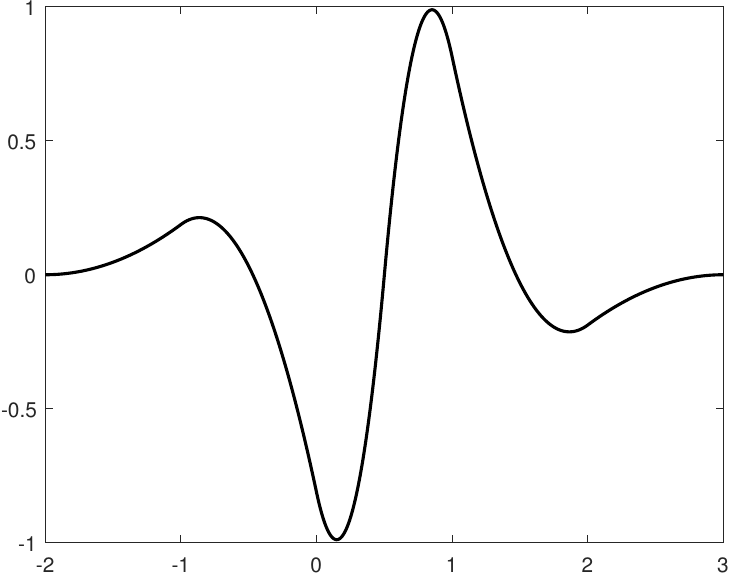}
    		\caption{$\psi$}
    	\end{subfigure}
    	 \begin{subfigure}[b]{0.24\textwidth} \includegraphics[width=\textwidth]{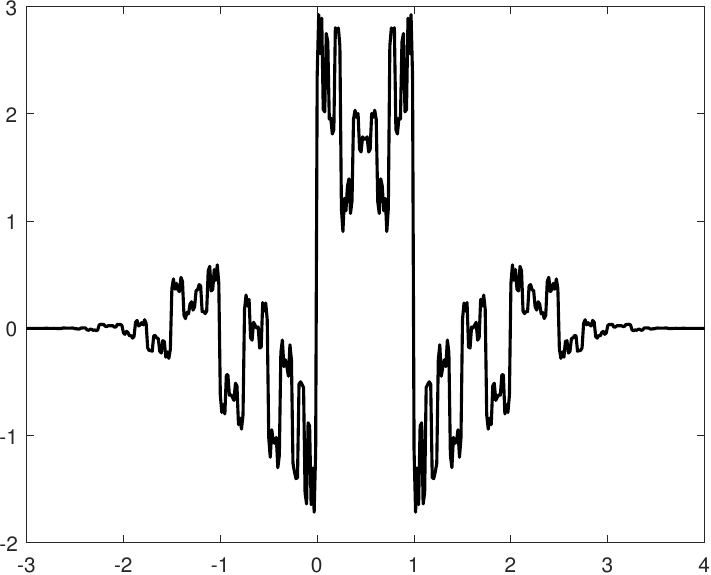}
    		\caption{$\tilde{\phi}$}
    	\end{subfigure}
    	 \begin{subfigure}[b]{0.24\textwidth} \includegraphics[width=\textwidth]{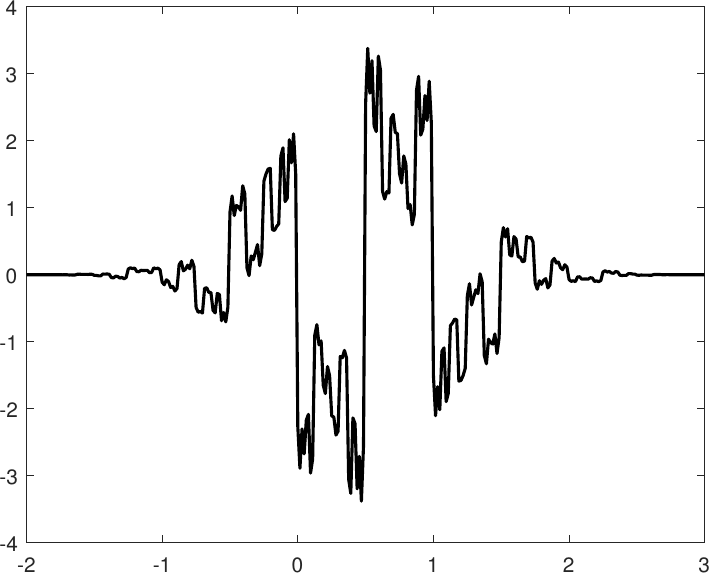}
    		\caption{$\tilde{\psi}$}
    	\end{subfigure}
    	 \begin{subfigure}[b]{0.24\textwidth}
    		 \includegraphics[width=\textwidth]{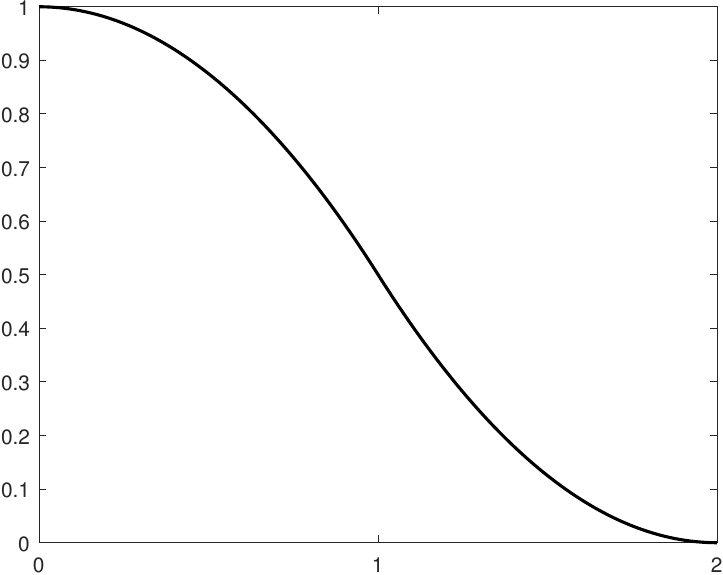}
    		\caption{$\phi^{L}$}
    	\end{subfigure}
    	 \begin{subfigure}[b]{0.24\textwidth}
    		 \includegraphics[width=\textwidth]{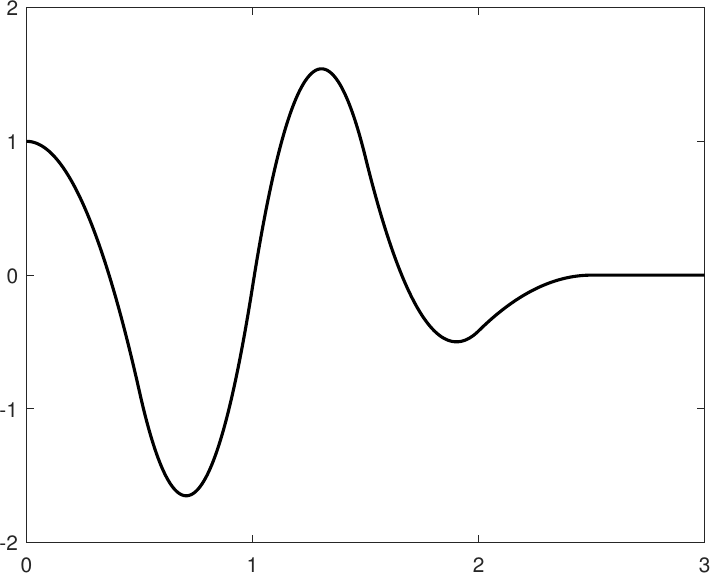}
    		 \caption{$\psi^L$}
    	\end{subfigure}
    	 \begin{subfigure}[b]{0.24\textwidth}
    		 \includegraphics[width=\textwidth]{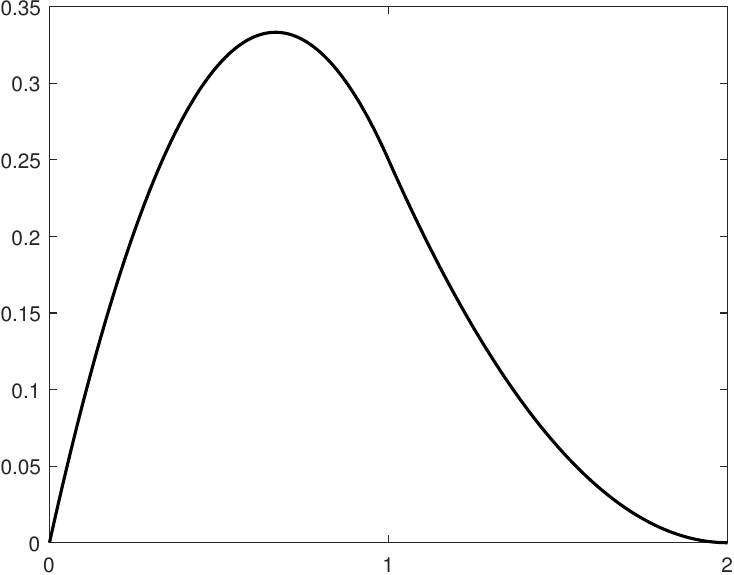}
    		\caption{$\phi^{L,bc}$}
    	\end{subfigure}
    	 \begin{subfigure}[b]{0.24\textwidth}
    		 \includegraphics[width=\textwidth]{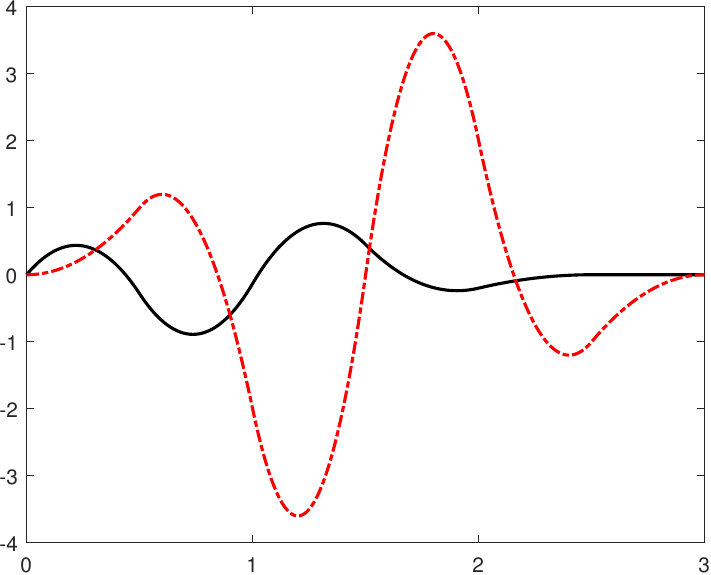}
    		 \caption{$\psi^{L,bc1},\psi^{L,bc2}$}
    	\end{subfigure}
    	\caption{The generators of Riesz wavelets $\Phi^{x}_{J_0} \cup \{\Psi_{j}^x : j\ge J_0\}$ and $\Phi^{y}_{J_0} \cup \{\Psi_{j}^y : j\ge J_0\}$ of $\LpI{2}$ 
    		for $J_0 \ge 2$.
    		The black (solid) and red (dotted dashed) lines correspond to the first and second components of a vector function respectively.}
    	\label{fig:B3}
    \end{figure}
	\end{example}	


	\begin{example} \label{ex:B4}
		\normalfont
		Consider a biorthogonal wavelet $(\{\tilde{\phi};\tilde{\psi}\},\{\phi;\psi\})$ in $\Lp{2}$ with $\wh{\phi}(0) = \wh{\tilde{\phi}}(0) =1$ and a biorthogonal wavelet filter bank $(\{\tilde{a};\tilde{b}\},\{a;b\})$ given by
		\begin{align*}
			a & = \{\tfrac{1}{16},\tfrac{1}{4},\tfrac{3}{8},\tfrac{1}{4},\tfrac{1}{16}\}_{[-2,2]}, \quad b=\{\tfrac{-1}{128},\tfrac{-1}{32},\tfrac{-1}{256},\tfrac{9}{64},\tfrac{31}{256},\tfrac{-11}{32},\tfrac{-23}{64},\tfrac{31}{32},\tfrac{-23}{64},\tfrac{-11}{32},\tfrac{31}{256},\tfrac{9}{64},\tfrac{-1}{256},\tfrac{-1}{32},\tfrac{-1}{128}\}_{[-6,8]},\\
			\tilde{a} & =\{\tfrac{1}{128},\tfrac{-1}{32},\tfrac{1}{256},\tfrac{9}{64},\tfrac{-31}{256},\tfrac{-11}{32},\tfrac{23}{64},\tfrac{31}{32},\tfrac{23}{64},\tfrac{-11}{32},\tfrac{-31}{256},\tfrac{9}{64},\tfrac{1}{256},\tfrac{-1}{32},\tfrac{1}{128}\}_{[-7,7]}, \quad
			\tilde{b} = \{\tfrac{1}{16},\tfrac{-1}{4},\tfrac{3}{8},\tfrac{-1}{4},\tfrac{1}{16}\}_{[-1,3]}.	 
		\end{align*}
		The analytic expression of $\phi$ is $\phi:=\tfrac{1}{6}(x+2)^3\chi_{[-2,-1)} + \tfrac{1}{6}(-3x^3 - 6x^2 + 4)\chi_{[-1,0)} + \tfrac{1}{6}(3x^3 - 6x^2 + 4)\chi_{[0,1)} - \tfrac{1}{6}(-2+x)^3\chi_{[1,2]}$. Note that $\text{fsupp}(\phi)=[-2,2]$, $\text{fsupp}(\psi)=\text{fsupp}(\tilde{\psi})=[-4,5]$, and $\text{fsupp}(\tilde{\phi})=[-7,7]$. Furthermore, $\sm(a)=3.5$, $\sm(\tilde{a})=0.858627$, and $\sr(a) = \sr(\tilde{a}) = 4$. This implies $\phi \in H^{3}(\R)$ and $\tilde{\phi} \in \Lp{2}$. Let $\phi^{L} = (\phi(\cdot + 1) + \phi + \phi(\cdot -1)) \chi_{[0,\infty)}$, $\phi^{L,bc1} = (-\phi(\cdot + 1) + \phi(\cdot -1)) \chi_{[0,\infty)}$, and $\phi^{L,bc2}=(\tfrac{2}{3}\phi(\cdot + 1) -\tfrac{1}{3} \phi + \tfrac{2}{3}\phi(\cdot -1)) \chi_{[0,\infty)}$. Also,
		\begin{align*}
		 [\phi^{L},\phi^{L,bc1},\phi^{L,bc2}]^{\tp} & =\text{diag}(1,\tfrac{1}{2},\tfrac{1}{4})[\phi^{L}(2\cdot),\phi^{L,bc1}(2\cdot),\phi^{L,bc2}(2\cdot)]^{\tp} + [\tfrac{7}{8},\tfrac{3}{4},\tfrac{11}{24}]^{\tp} \phi(2\cdot -2)\\
		& \quad + [\tfrac{1}{2},\tfrac{1}{2},\tfrac{1}{3}]^{\tp} \phi(2\cdot -3) + [\tfrac{1}{8},\tfrac{1}{8},\tfrac{1}{12}]^{\tp} \phi(2\cdot -4).
		\end{align*}
	 	The direct approach in \cite[Theorem 4.2]{HM21a} yields
	 	\begin{align*}
	 		\psi^{L} & :=\phi^{L}(2\cdot) - \tfrac{527}{69} \phi^{L,bc1}(2\cdot) + \tfrac{278}{23} \phi^{L,bc2}(2\cdot) -\tfrac{61}{69} \phi(2\cdot - 2) + \tfrac{22339}{43470} \phi(2\cdot - 3) -\tfrac{47113}{173880} \phi(2\cdot - 4)\\
	 		& \quad + \tfrac{3832}{21735} \phi(2\cdot -5) - \tfrac{18509}{173880} \phi(2\cdot -6) + \tfrac{2}{69} \phi(2\cdot -7),\\
	 		\psi^{L,bc1} & := \tfrac{1}{8} \phi(2\cdot -3) - \tfrac{1}{2}\phi(2\cdot -4) + \tfrac{3}{4}\phi(2\cdot -5) - \tfrac{1}{2} \phi(2\cdot -6) + \tfrac{1}{8} \phi(2\cdot -7).\\
	 		\psi^{L,bc2} & := \tfrac{1}{8} \phi(2\cdot -2) - \tfrac{1}{2}\phi(2\cdot -3) + \tfrac{3}{4}\phi(2\cdot -4) - \tfrac{1}{2} \phi(2\cdot -5) + \tfrac{1}{8} \phi(2\cdot -6),\\
	 		\psi^{L,bc3} & := -\tfrac{31}{2500} \phi^{L,bc1}(2\cdot) + \tfrac{147}{5000} \phi^{L,bc2}(2\cdot) - \tfrac{381}{80000} \phi(2\cdot -2) + \tfrac{4517}{3600000} \phi(2\cdot -3) -\tfrac{169}{1800000} \phi(2\cdot -4)\\
	 		& \quad + \tfrac{1757}{1800000} \phi(2\cdot -5) - \tfrac{3493}{3600000} \phi(2\cdot -6) + \tfrac{21}{80000} \phi(2\cdot -7),\\
	 		\psi^{L,bc4} & := \tfrac{6}{125} \phi^{L,bc1}(2\cdot) - \tfrac{307}{2500} \phi^{L,bc2}(2\cdot) + \tfrac{1869}{80000} \phi(2\cdot -2) + \tfrac{3453}{2800000} \phi(2\cdot -3) -\tfrac{661}{175000} \phi(2\cdot -4)\\
	 		& \quad - \tfrac{8663}{1400000} \phi(2\cdot -5) -\tfrac{29521}{2800000} \phi(2\cdot -6) + \tfrac{351}{16000} \phi(2\cdot -7) -\tfrac{69}{8000} \phi(2\cdot -8),\\
	 		\psi^{L,bc5} & := \tfrac{11}{1000} \phi^{L,bc1}(2\cdot) - \tfrac{4}{125} \phi^{L,bc2}(2\cdot) + \tfrac{69}{5000}\phi(2\cdot - 2) - \tfrac{1361}{126000} \phi(2\cdot - 3) - \tfrac{6931}{1260000} \phi(2\cdot - 4)\\
	 		& \quad + \tfrac{12163}{630000} \phi(2\cdot -5) - \tfrac{9253}{630000} \phi(2\cdot -6) + \tfrac{19}{5000} \phi(2\cdot -7),\\
	 		\psi^{L,bc6} & := \tfrac{63}{1000} \phi^{L,bc1}(2\cdot) + \tfrac{129}{1000} \phi^{L,bc2}(2\cdot) -\tfrac{993}{2500} \phi(2\cdot - 2) +\tfrac{18497}{70000} \phi(2\cdot - 3) + \tfrac{64081}{140000} \phi(2\cdot - 4)\\
	 		& \quad - \tfrac{6441}{14000} \phi(2\cdot -5) - \tfrac{17351}{35000} \phi(2\cdot -6) + \tfrac{1913}{2500} \phi(2\cdot -7) - \tfrac{641}{2500} \phi(2\cdot -8).
 		\end{align*}
 		%
 		For $J_0 \ge 3$ and $j \ge J_0$, define
 		\begin{align*}
 			& \Phi^{x}_{J_0} := \{\phi^{L,bc1}_{J_0;0},\phi^{L,bc2}_{J_0;0}\} \cup \{\phi_{j;k}:2\le k \le 2^{J_0}-2\} \cup \{\phi^{R,bc1}_{J_0;2^{J_0}-1},\phi^{R,bc2}_{J_0;2^{J_0}-1}\}, \quad \text{and}\\
 			& \Psi^{x}_{j} := \{\psi^{L,bc1}_{j;0},\psi^{L,bc2}_{j;0},\psi^{L,bc3}_{j;0},\psi^{L,bc4}_{j;0}\} \cup \{\psi_{j;k}: 4\le k \le 2^{j}-5\} \cup \{\psi^{R,bc1}_{j;2^j-1},\psi^{R,bc2}_{j;2^j-1},\psi^{R,bc3}_{j;2^j-1},\psi^{R,bc4}_{j;2^j-1}\},
 		\end{align*}
 		where $\phi^{R,bc i} = \phi^{L,bc i}(1-\cdot)$ for $i=1,2$, $\psi^{R,bc i} = \psi^{L,bc i}(1-\cdot)$ for $i=1,\dots,4$ , and
 		\[
 		\Phi^{y}_{J_0} := \Phi^{x}_{J_0} \cup \{\phi^{R}_{J_0;2^{J_0}-1}\}, \quad
 		\Psi^{y}_{j} := \left(\Psi^{x}_j \backslash \{\psi^{R,bc2}_{j;2^j-1},\psi^{R,bc3}_{j;2^j-1},\psi^{R,bc4}_{j;2^j-1}\}\right) \cup \{\psi^{R,bc5}_{j;2^j-1},\psi^{R,bc6}_{j;2^j-1},\psi^{R}_{j;2^j-1}\},
 		\]
 		where $\phi^{R} = \phi^{L}(1-\cdot)$ and $\psi^{R,bci} = \psi^{L,bci}(1-\cdot)$ for $i=5,6$. Then, $\mathcal{B}^x:=\Phi^{x}_{J_0} \cup \{\Psi_{j}^x : j\ge J_0\}$
 		and $\mathcal{B}^y:=\Phi^{y}_{J_0} \cup \{\Psi_{j}^y : j\ge J_0\}$ with $J_0 \ge 3$
 		are Riesz wavelets in $\LpI{2}$.
 		See \cref{fig:B4} for their generators. 
 		
 		 \begin{figure}[htbp]
 		 	\centering
 		 	 \begin{subfigure}[b]{0.24\textwidth} \includegraphics[width=\textwidth]{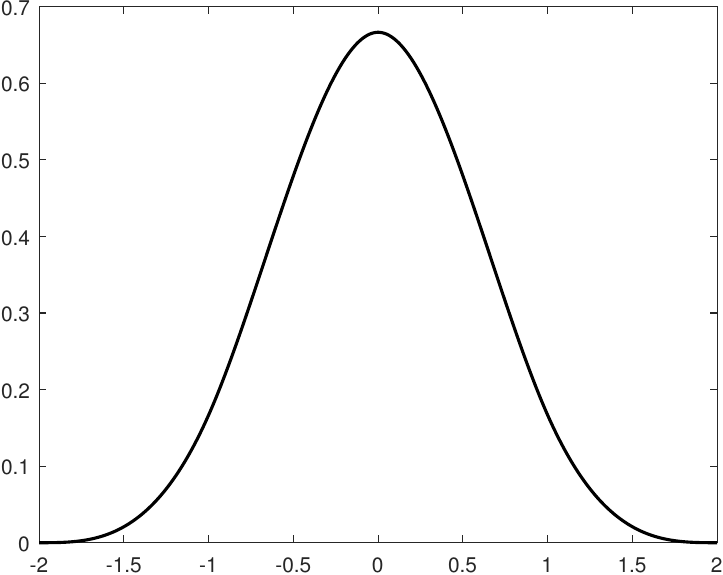}
 		 		\caption{$\phi$}
 		 	\end{subfigure}
 		 	 \begin{subfigure}[b]{0.24\textwidth} \includegraphics[width=\textwidth]{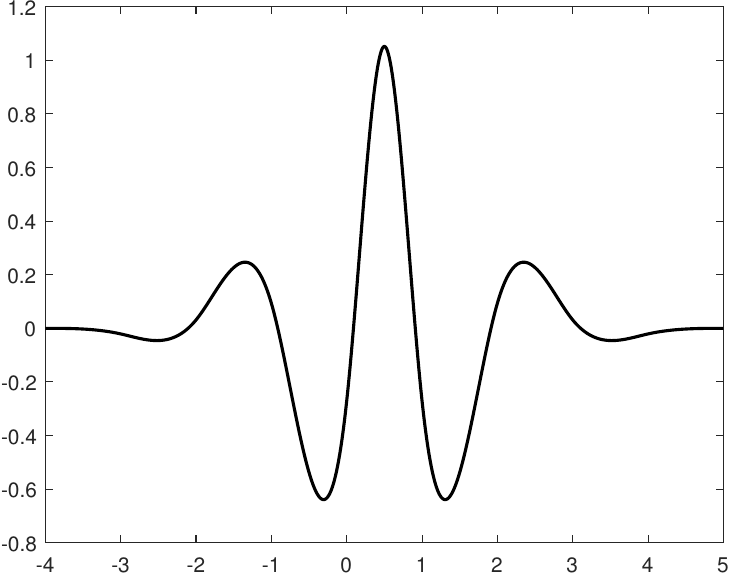}
 		 		\caption{$\psi$}
 		 	\end{subfigure}
 		 	 \begin{subfigure}[b]{0.24\textwidth} \includegraphics[width=\textwidth]{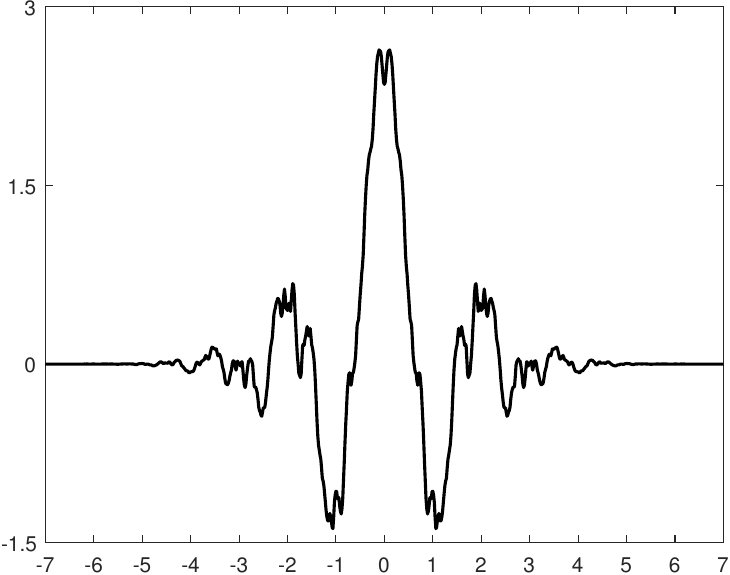}
 		 		 \caption{$\tilde{\phi}$}
 		 	\end{subfigure}
 		 	 \begin{subfigure}[b]{0.24\textwidth} \includegraphics[width=\textwidth]{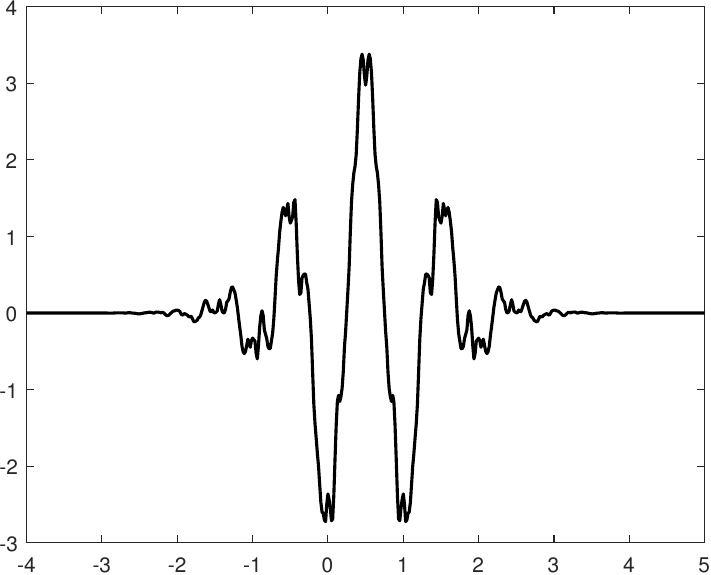}
 		 		 \caption{$\tilde{\psi}$}
 		 	\end{subfigure}
 		 	 \begin{subfigure}[b]{0.24\textwidth}
 		 		 \includegraphics[width=\textwidth]{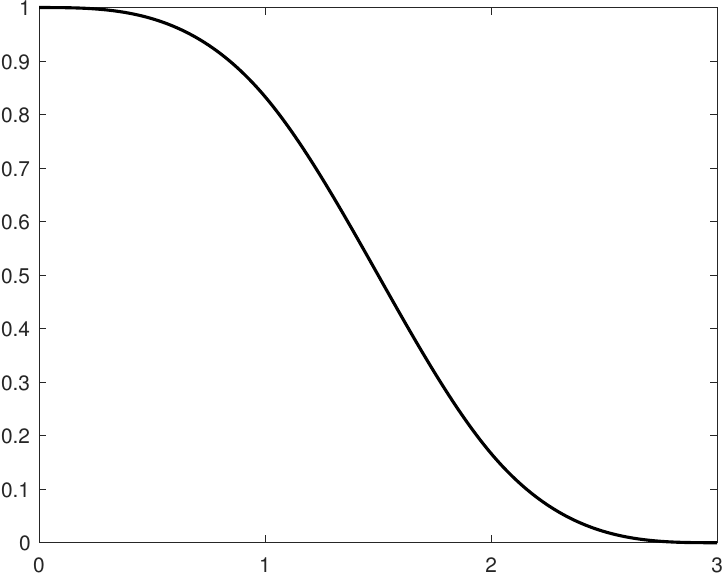}
 		 		 \caption{$\phi^{L}$}
 		 	\end{subfigure}
 		 	 \begin{subfigure}[b]{0.24\textwidth}
 		 		 \includegraphics[width=\textwidth]{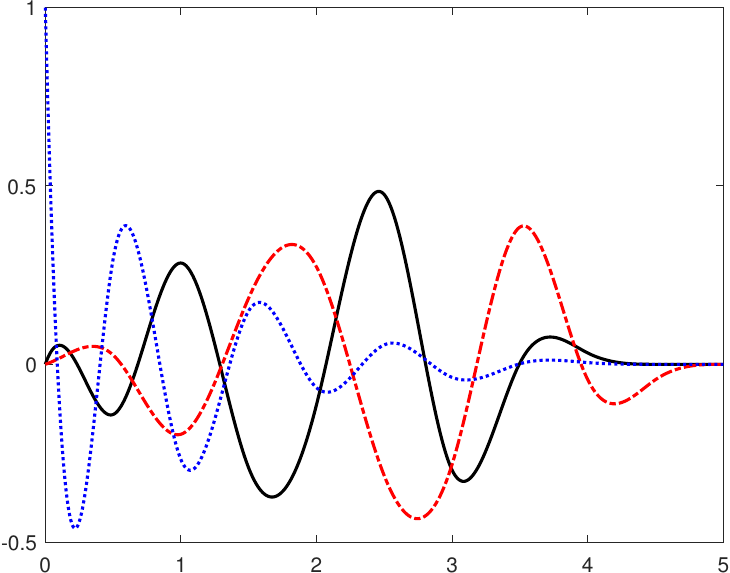}
 		 		 \caption{$\psi^{L,bc5},\psi^{L,bc6},\psi^L$}
 		 	\end{subfigure}
 		 	 \begin{subfigure}[b]{0.24\textwidth}
 		 		 \includegraphics[width=\textwidth]{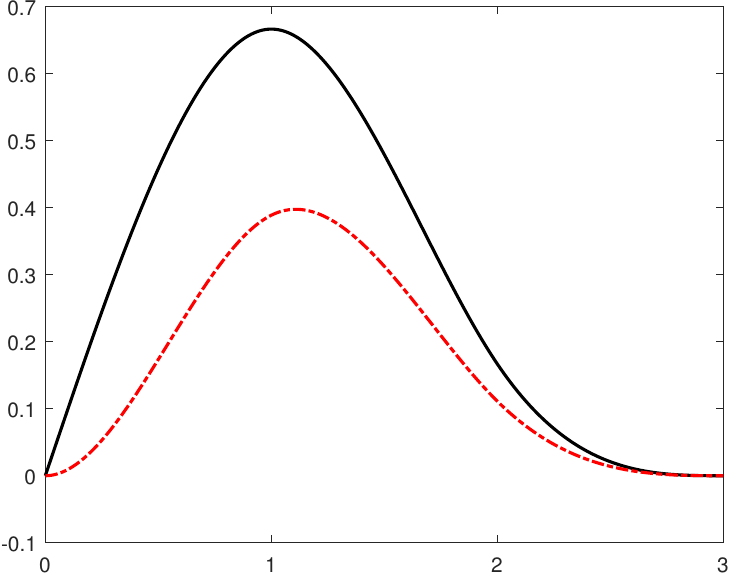}
 		 		 \caption{$\phi^{L,bc1},\phi^{L,bc2}$}
 		 	\end{subfigure}
 		 	 \begin{subfigure}[b]{0.24\textwidth}
 		 		 \includegraphics[width=\textwidth]{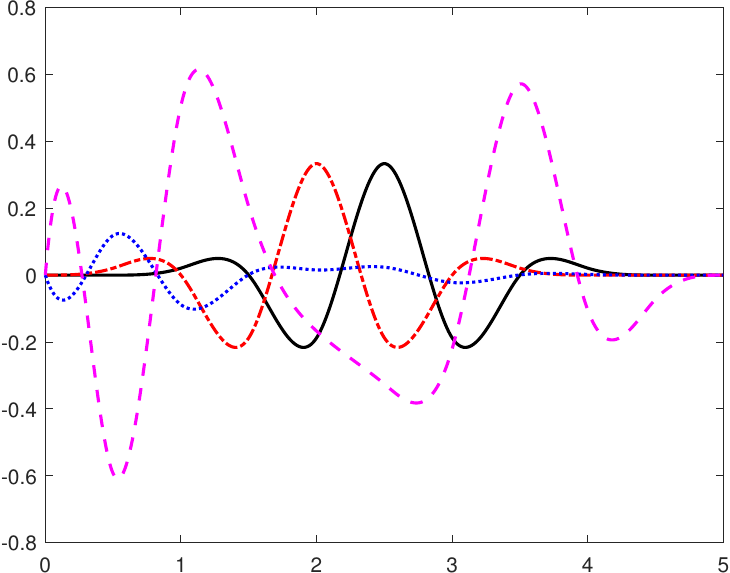}
 		 		\caption{$\psi^{L,bci}, i=1,\dots,4$}
 		 	\end{subfigure}
 		 	\caption{The generators of Riesz wavelets $\Phi^{x}_{J_0} \cup \{\Psi_{j}^x : j\ge J_0\}$ and $\Phi^{y}_{J_0} \cup \{\Psi_{j}^y : j\ge J_0\}$ of $\LpI{2}$ 
 		 		for $J_0 \ge 3$. The black (solid), red (dotted dashed), blue (dotted), and purple (dashed) lines correspond to the first, second, third, and fourth components of a vector function respectively.}
 		 	\label{fig:B4}
 		 \end{figure}
 	 \end{example}	
	
	\subsection{Spline multiwavelets on $\cI$} We present three spline multiwavelets on $\cI$.
		
	\begin{example} \label{ex:sr3}
	\normalfont
	Consider a biorthogonal wavelet $(\{\tilde{\phi};\tilde{\psi}\},\{\phi;\psi\})$ in $\Lp{2}$ with $\wh{\phi}(0)=(\tfrac{1}{3},\tfrac{2}{3})^{\tp}$, $\wh{\tilde{\phi}}(0)=(1,1)^\tp$, and a biorthogonal wavelet filter bank $(\{\tilde{a},\tilde{b}\},\{a,b\})$ given by
	\begin{align*}
		a & =\left\{
		\begin{bmatrix}
			0 & -\frac{1}{16}\\
			0 & 0
		\end{bmatrix},
		\begin{bmatrix}
			0 & \frac{3}{16}\\
			0 & 0
		\end{bmatrix},
		\begin{bmatrix}
			\frac{1}{2} & \frac{3}{16}\\
			0 & \frac{3}{8}
		\end{bmatrix},
		\begin{bmatrix}
			0 & -\frac{1}{16}\\
			\frac{1}{2} & \frac{3}{8}
		\end{bmatrix}
		\right\}_{[-2,1]},\\
		b & =\left\{
		\begin{bmatrix}
			0 & -\frac{1}{32}\\
			0 & -\frac{1}{8}
		\end{bmatrix},
		\begin{bmatrix}
			\frac{3}{8} & -\frac{9}{32}\\
			-\frac{3}{2} & \frac{15}{8}
		\end{bmatrix},
		\begin{bmatrix}
			\frac{1}{2} & -\frac{9}{32}\\
			0 & -\frac{15}{8}
		\end{bmatrix},
		\begin{bmatrix}
			\frac{3}{8} & -\frac{1}{32}\\
			\frac{3}{2} & \frac{1}{8}
		\end{bmatrix}
		\right\}_{[-2,1]},\\
		\tilde{a} & =\left\{
		\begin{bmatrix}
			\frac{3}{32} & -\frac{1}{8}\\
			0 & 0
		\end{bmatrix},
		\begin{bmatrix}
			-\frac{3}{16} & \frac{3}{8}\\
			0 & 0
		\end{bmatrix},
		\begin{bmatrix}
			\frac{11}{16} & \frac{3}{8}\\
			-\frac{3}{32} & \frac{3}{8}
		\end{bmatrix},
		\begin{bmatrix}
			-\frac{3}{16} & -\frac{1}{8}\\
			\frac{7}{16} & \frac{3}{8}
		\end{bmatrix},
		\begin{bmatrix}
			\frac{3}{32} & 0\\
			-\frac{3}{32} & 0
		\end{bmatrix}
		\right\}_{[-2,2]},\\
		\tilde{b} & =\left\{
		\begin{bmatrix}
			-\frac{3}{32} & \frac{1}{8}\\
			\frac{3}{128} & -\frac{1}{32}
		\end{bmatrix},
		\begin{bmatrix}
			\frac{3}{16} & -\frac{3}{8}\\
			-\frac{3}{64} & \frac{3}{32}
		\end{bmatrix},
		\begin{bmatrix}
			\frac{5}{16} & -\frac{3}{8}\\
			0 & -\frac{3}{32}
		\end{bmatrix},
		\begin{bmatrix}
			\frac{3}{16} & \frac{1}{8}\\
			\frac{3}{64} & \frac{1}{32}
		\end{bmatrix},
		\begin{bmatrix}
			-\frac{3}{32} & 0\\
			-\frac{3}{128} & 0
		\end{bmatrix}
		\right\}_{[-2,2]}.
	\end{align*}
	The analytic expression of $\phi=(\phi^1,\phi^2)^{\tp}$ is
	\[
	\phi^{1}(x) = (2x^2 + 3x + 1) \chi_{[-1,0)} + (2x^2 -3x +1)\chi_{[0,1]}
	\quad \text{and} \quad
	\phi^{2}(x) = (-4x^2 + 4x)\chi_{[0,1]}.
	\]
	Note that $\text{fsupp}(\phi)=\text{fsupp}(\psi)=[-1,1]$ and $\text{fsupp}(\tilde{\phi})=\text{fsupp}(\tilde{\psi})=[-2,2]$. Furthermore, $\sm(a)=\sm(\tilde{a})=1.5$ and $\sr(a)=\sr(\tilde{a})=3$, and its matching filters $v,\tilde{v} \in (l_0(\Z))^{1 \times 2}$ with $\wh{v}(0)\wh{\phi}(0)=\wh{\tilde{v}}(0)\wh{\tilde{\phi}}(0)=1$ are given by $\wh{v}(0)=(1,1)$, $\wh{v}'(0)=\ia (0,\tfrac{1}{2})$, $\wh{v}''(0)=(0,-\tfrac{1}{4})$, $\wh{\tilde{v}}(0)=(\tfrac{1}{3},\tfrac{2}{3})$, $\wh{\tilde{v}}'(0)=\ia (0,\tfrac{1}{3})$, and $\wh{\tilde{v}}''(0)=(\tfrac{1}{30},-\tfrac{1}{5})$. This implies $\phi,\tilde{\phi} \in H^{1}(\R)$. Let $\phi^{L} := \phi^{1}\chi_{[0,\infty)}$ and $\phi^{L,bc} :=\phi^{2}\chi_{[0,\infty)}$. Note that $\phi^{L}=\phi^{L}(2\cdot) + \tfrac{3}{8} \phi^{2}(2\cdot) - \tfrac{1}{8} \phi^{2}(2\cdot-1)$ and $\phi^{L,bc}=\tfrac{3}{4}\phi^{L,bc}(2\cdot) + [1,\tfrac{3}{4}] \phi(2\cdot-1)$. The direct approach in \cite[Theorem 4.2]{HM21a} yields
	\begin{align*}
		\psi^{L} & := \phi^{L}(2\cdot) -\tfrac{9}{16}\phi^{L,bc}(2\cdot) + [\tfrac{3}{4},- \tfrac{1}{16}] \phi(2\cdot-1),\\
		\psi^{L,bc} & :=\phi^{L,bc}(2\cdot) + [-\tfrac{2121}{512}, \tfrac{657}{4096}] \phi(2\cdot-1) + [\tfrac{3877}{1024}, -\tfrac{4023}{4096}] \phi(2\cdot-2).
	\end{align*}
	For $J_0 \ge 1$ and $j \ge J_0$, define
	\[
	\Phi^{x}_{J_0} := \{\phi^{L,bc}_{J_0;0}\} \cup \{\phi_{J_0;k}: 1\le k \le 2^{J_0}-1\},
	\quad
	\Psi^{x}_{j} := \{\psi^{L,bc}_{j;0}\} \cup \{\psi_{j;k}:1\le k \le 2^{j}-1\} \cup \{\psi^{R,bc}_{j;2^j-1}\},
	\]
	where $\psi^{R,bc}=\psi^{L,bc}(1-\cdot)$, and
	\[
	\Phi^{y}_{J_0} := \Phi^{x}_{J_0} \cup \{\phi^{R}_{J_0;2^{J_0}-1}\}, \quad
	\Psi^{y}_{j} := \left(\Psi^{x}_{j} \backslash \{\psi^{R,bc}_{j;2^j-1}\}\right) \cup \{\psi^{R}_{j;2^j-1}\},
	\]
	where $\phi^{R}=\phi^{L}(1-\cdot)$ and $\psi^{R}=\psi^{L}(1-\cdot)$. Then, $\mathcal{B}^{x}:=\Phi^{x}_{J_0} \cup \{\Psi_{j}^x : j\ge J_0\}$
	and $\mathcal{B}^{y}:=\Phi^{y}_{J_0} \cup \{\Psi_{j}^y : j\ge J_0\}$
	are Riesz wavelets in $\LpI{2}$. See \cref{fig:cavity:1D} for their generators.
	
	\begin{figure}[htbp]
		\centering
		 \begin{subfigure}[b]{0.24\textwidth} \includegraphics[width=\textwidth]{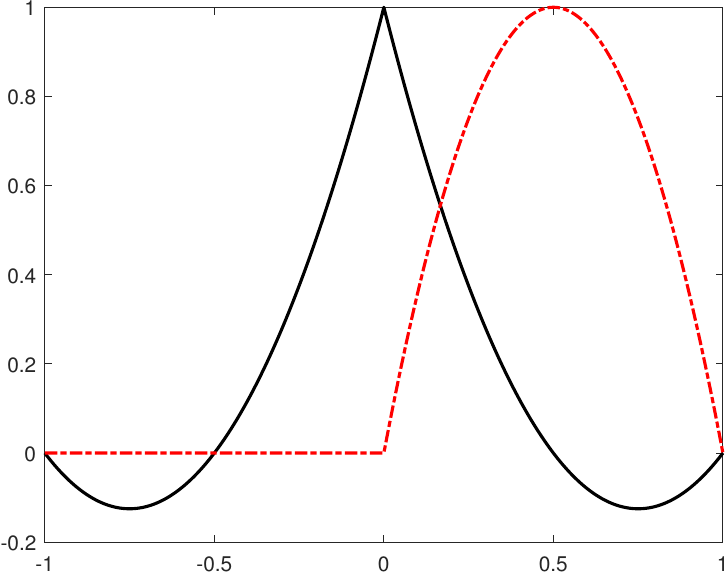}
			 \caption{$\phi=(\phi^{1},\phi^{2})^{\tp}$}
		\end{subfigure}
		 \begin{subfigure}[b]{0.24\textwidth} \includegraphics[width=\textwidth]{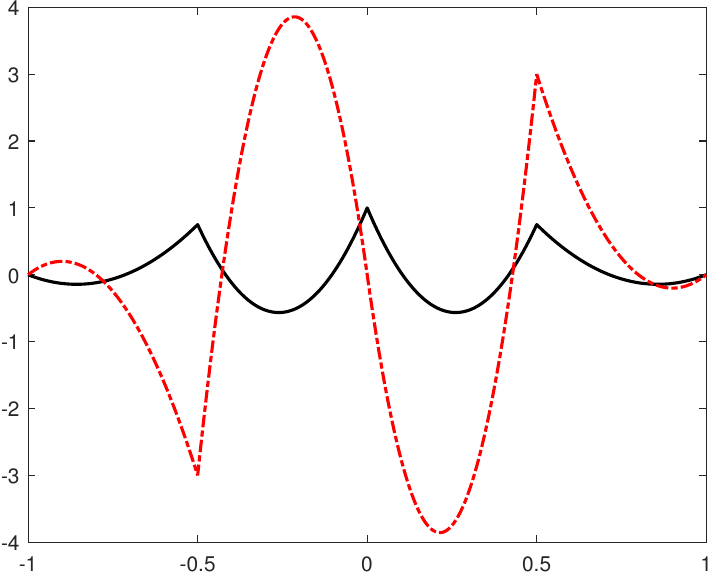}
			 \caption{$\psi=(\psi^{1},\psi^{2})^{\tp}$}
		\end{subfigure}
		 \begin{subfigure}[b]{0.24\textwidth} \includegraphics[width=\textwidth]{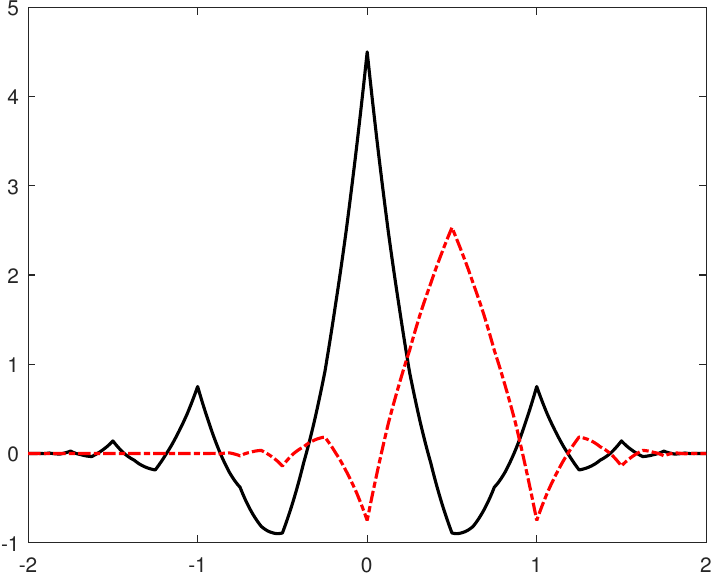}
			 \caption{$\tilde{\phi}=(\tilde{\phi}^{1},\tilde{\phi}^{2})^{\tp}$}
		\end{subfigure}
		 \begin{subfigure}[b]{0.24\textwidth} \includegraphics[width=\textwidth]{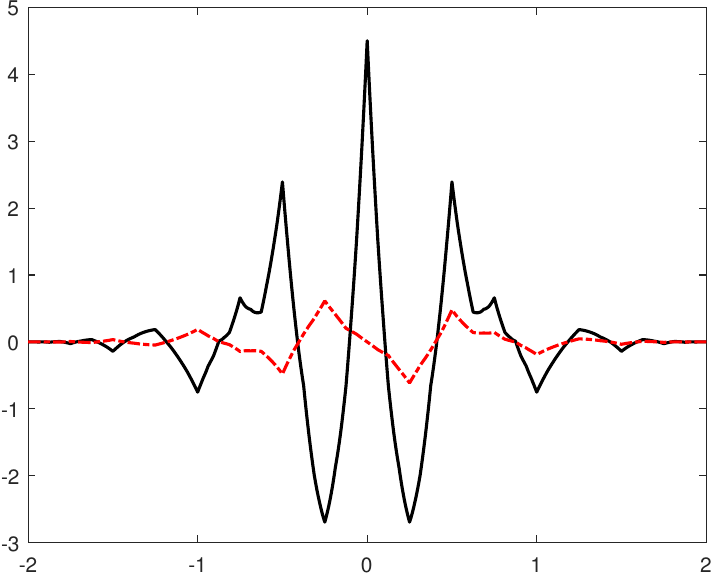}
			 \caption{$\tilde{\psi}=(\tilde{\psi}^{1},\tilde{\psi}^{2})^{\tp}$}
		\end{subfigure}
		 \begin{subfigure}[b]{0.24\textwidth}
			 \includegraphics[width=\textwidth]{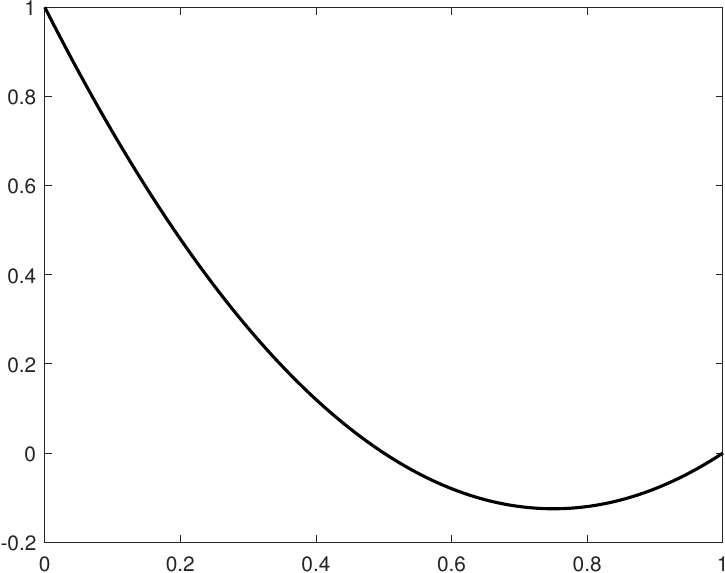}
			\caption{$\phi^{L}$}
		\end{subfigure}
		 \begin{subfigure}[b]{0.24\textwidth}
			 \includegraphics[width=\textwidth]{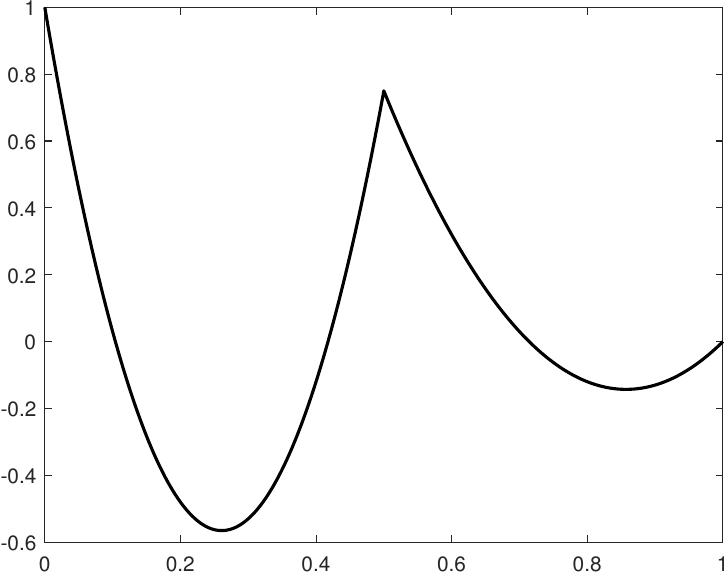}
			\caption{$\psi^{L}$}
		\end{subfigure}
		 \begin{subfigure}[b]{0.24\textwidth}
			 \includegraphics[width=\textwidth]{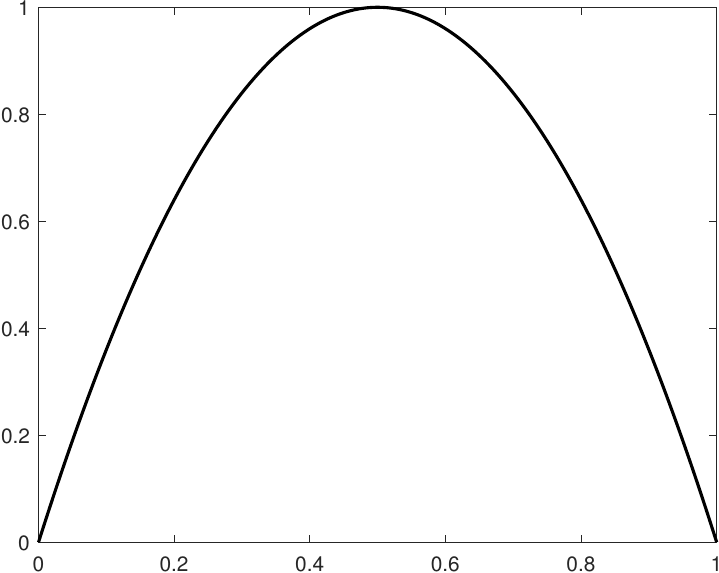}
			\caption{$\phi^{L,bc}$}
		\end{subfigure}
		 \begin{subfigure}[b]{0.24\textwidth}
			 \includegraphics[width=\textwidth]{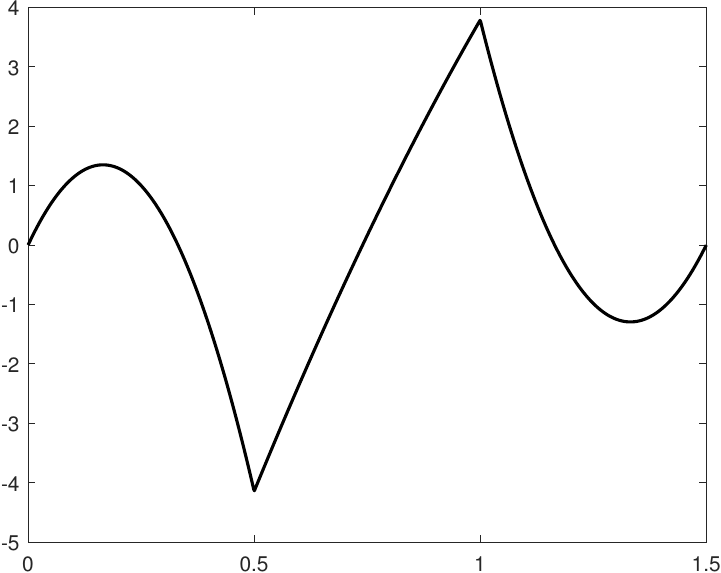}
			\caption{$\psi^{L,bc}$}
		\end{subfigure}
		\caption{The generators of Riesz wavelets $\Phi^{x}_{J_0} \cup \{\Psi_{j}^x : j\ge J_0\}$ and $\Phi^{y}_{J_0} \cup \{\Psi_{j}^y : j\ge J_0\}$ of $\LpI{2}$ 
			for $J_0 \ge 1$. The black (solid) and red (dotted dashed) lines correspond to the first and second components of a vector function.}
		\label{fig:cavity:1D}
	\end{figure}
	\end{example}
	
	\begin{example}\label{ex:hmt}
	\normalfont
	Consider a biorthogonal wavelet $(\{\tilde{\phi};\tilde{\psi}\},\{\phi;\psi\})$ in $\Lp{2}$ with $\wh{\phi}(0)=\wh{\tilde{\phi}}(0)=(1,0)^\tp$ and a biorthogonal wavelet filter bank $(\{\tilde{a};\tilde{b}\},\{a;b\})$ given by
	\begin{align*} 
		a=&{\left\{ \begin{bmatrix} \tfrac{1}{4} &\tfrac{3}{8}\\[0.3em]
				-\tfrac{1}{16} &-\tfrac{1}{16}\end{bmatrix},
			\begin{bmatrix} \tfrac{1}{2} &0 \\[0.3em]
				0 &\tfrac{1}{4}\end{bmatrix},
			\begin{bmatrix} \tfrac{1}{4} &-\tfrac{3}{8}\\[0.3em]
				\tfrac{1}{16} &-\tfrac{1}{16}\end{bmatrix}\right\}_{[-1,1]}},\\
		b=&{\left\{ \begin{bmatrix} 0 & 0\\[0.3em]
				\tfrac{2}{97} & \tfrac{24}{679}\end{bmatrix},
			\begin{bmatrix} -\tfrac{1}{2} & -\tfrac{15}{4}\\[0.3em]
				\tfrac{77}{1164} & \tfrac{2921}{2761}\end{bmatrix},
			\begin{bmatrix} 1 & 0 \\[0.3em]
				0 & 1 \end{bmatrix},
			\begin{bmatrix} -\tfrac{1}{2} & \tfrac{15}{4}\\[0.3em]
				-\tfrac{77}{1164} & \tfrac{2921}{2761}\end{bmatrix},
			\begin{bmatrix} 0 & 0\\[0.3em]
				-\tfrac{2}{97} & \tfrac{24}{679}\end{bmatrix}\right\}_{[-2,2]}},\\ 
		\tilde{a}=
		&{\left\{\begin{bmatrix} -\tfrac{13}{2432} & -\tfrac{91}{29184}\\[0.3em]
				\tfrac{3}{152} & \tfrac{7}{608}\end{bmatrix},
			\begin{bmatrix} \tfrac{39}{2432} &\tfrac{13}{3648}\\[0.3em]
				-\tfrac{9}{152} &-\tfrac{1}{76}\end{bmatrix},
			\begin{bmatrix} -\tfrac{1}{12} & -\tfrac{1699}{43776} \\[0.3em]
				\tfrac{679}{1216} & \tfrac{4225}{14592}\end{bmatrix},
			\begin{bmatrix} \tfrac{569}{2432} & \tfrac{647}{10944}\\[0.3em]
				-\tfrac{1965}{1216} & -\tfrac{37}{96}\end{bmatrix},
			\begin{bmatrix} \tfrac{2471}{3648} & 0\\[0.3em]
				0 &\tfrac{7291}{7296}\end{bmatrix},\right.}
		\\
		&\;\; {\left.\begin{bmatrix} \tfrac{569}{2432} & -\tfrac{647}{10944}\\[0.3em]
				\tfrac{1965}{1216} & -\tfrac{37}{96}\end{bmatrix},
			\begin{bmatrix} -\tfrac{1}{12} & \tfrac{1699}{43776} \\[0.3em]
				-\tfrac{679}{1216} & \tfrac{4225}{14592}\end{bmatrix},
			\begin{bmatrix} \tfrac{39}{2432} &-\tfrac{13}{3648}\\[0.3em]
				\tfrac{9}{152} &-\tfrac{1}{76}\end{bmatrix},
			\begin{bmatrix} -\tfrac{13}{2432} & \tfrac{91}{29184}\\[0.3em]
				-\tfrac{3}{152} & \tfrac{7}{608}\end{bmatrix}\right\}_{[-4,4]}},\\
		 \tilde{b}=&{\left\{\begin{bmatrix} -\tfrac{1}{4864} & -\tfrac{7}{58368}\\[0.3em]
				0 & 0\end{bmatrix},
			\begin{bmatrix} \tfrac{3}{4864} & \tfrac{1}{7296}\\[0.3em]
				0 & 0\end{bmatrix},
			\begin{bmatrix} \tfrac{1}{24} & \tfrac{2161}{87552}\\[0.3em]
				-\tfrac{679}{4864} & -\tfrac{4753}{58368}\end{bmatrix},
			\begin{bmatrix} -\tfrac{611}{4864} & -\tfrac{605}{21888}\\[0.3em]
				\tfrac{2037}{4864} & \tfrac{679}{7296}\end{bmatrix},
			\begin{bmatrix} \tfrac{1219}{7296} & 0\\[0.3em]
				0 & \tfrac{7469}{29814}\end{bmatrix},\right.}\\
		& \;\;{\left. \begin{bmatrix} -\tfrac{611}{4864} & \tfrac{605}{21888}\\[0.3em]
				-\tfrac{2037}{4864} & \tfrac{679}{7296}\end{bmatrix},
			\begin{bmatrix} \tfrac{1}{24} & -\tfrac{2161}{87552}\\[0.3em]
				\tfrac{679}{4864} & -\tfrac{4753}{58368}\end{bmatrix},
			\begin{bmatrix} \tfrac{3}{4864} & -\tfrac{1}{7296}\\[0.3em]
				0 & 0\end{bmatrix},
			\begin{bmatrix} -\tfrac{1}{4864} & \tfrac{7}{58368}\\[0.3em]
				0 & 0\end{bmatrix}\right\}_{[-4,4]}}.
	\end{align*}
	The analytic expression of the well-known Hermite cubic splines $\phi=(\phi^1,\phi^2)^{\tp}$ is
	\[
	\phi^1  := (1-x)^2 (1+2x) \chi_{[0,1]} + (1+x)^2 (1-2x) \chi_{[-1,0)}, \quad \text{and} \quad
	\phi^{2} := (1-x)^2x \chi_{[0,1]} + (1+x)^2x \chi_{[-1,0)}.
	\]
	Note that $\fs(\phi)=[-1,1]$, $\fs(\psi)=[-2,2]$, and $\fs(\tilde{\phi})=\fs(\tilde{\psi})=[-4,4]$.
	Then $\sm(a)=2.5$, $\sm(\tilde{a})=0.281008$,
	$\sr(a)=\sr(\tilde{a})=4$,
	and the matching filters $\vgu, \tilde{\vgu} \in \lrs{0}{1}{2}$ with $\wh{\vgu}(0)\wh{\phi}(0)=\wh{\tilde{\vgu}}(0)\wh{\tilde{\phi}}(0)=1$ are given by
	\begin{align*}
	&\wh{\vgu}(0,0)=(1,0), \quad  \wh{\vgu}'(0)=(0,\ia), \quad
	 \wh{\vgu}''(0)=\wh{\vgu}'''(0)=(0,0), \quad \text{and}\\
	&\wh{\tilde{\vgu}}(0)=(1,0),\quad \wh{\vgu}'(0)=\ia(0,\tfrac{1}{15}),\quad \wh{\tilde{\vgu}}''(0)=(-\tfrac{2}{15},0),\quad
	 \wh{\vgu}'''(0)=\ia(0,-\tfrac{2}{105}).
	\end{align*}
	This implies $\phi \in H^{2}(\R)$ and $\tilde{\phi} \in \Lp{2}$. Let $\phi^{L}:=\phi^{1} \chi_{[0,\infty)}$ and $\phi^{L,bc}:=\phi^{2}\chi_{[0,\infty)}$. Note that $\phi^{L}=\phi^{L}(2\cdot) + [\frac{1}{2},-\frac{3}{4}] \phi(2\cdot -1)$ and $\phi^{L,bc}=\frac{1}{2} \phi^{L,bc}(2\cdot) + [\frac{1}{8},-\frac{1}{8}] \phi(2\cdot -1)$. The direct approach in \cite[Theorem 4.2]{HM21a} yields
	\begin{align*}
		\psi^{L} & := \phi^{L}(2\cdot) - \tfrac{27}{4} \phi^2(2\cdot)
		+ [\tfrac{4139}{26352},\tfrac{215}{144}]\phi(2\cdot -1)
		- [\tfrac{623}{6588},\tfrac{119}{1098}]\phi(2\cdot-2)
		+ [0,\tfrac{27}{122}] \phi(2\cdot-3),\\
		\psi^{L,bc1}& := -\tfrac{21}{2}\phi^{L,bc}(2\cdot) + [\tfrac{17}{24},-\tfrac{5847}{488}]\phi(2\cdot -1) +
		 [\tfrac{115}{366},\tfrac{233}{61}]\phi(2\cdot -2) +
		 [-\tfrac{9}{61},0]\phi(2\cdot-3),\\
		\psi^{L,bc2}& := \tfrac{93}{16} \phi^{L,bc}(2\cdot) +
		 [-\tfrac{235}{2112},\tfrac{30351}{3904}] \phi(2\cdot -1) +
		 [\tfrac{8527}{32208},\tfrac{3571}{488}]\phi(2\cdot -2) +
		 [-\tfrac{428}{671},\tfrac{195}{44}]\phi(2\cdot -3),\\
		\psi^{L,bc3}& := \phi^{L,bc}(2\cdot) - [\tfrac{41}{144},\tfrac{121}{488}] \phi(2\cdot -1) + [\tfrac{341}{2196},-\tfrac{1987}{732}] \phi(2\cdot -2) + [\tfrac{45}{976},0] \phi(2\cdot-3).
	\end{align*}
	For $J_0 \ge 2$ and $j \ge J_0$, define
	\begin{align*}
	\Phi^{x}_{J_0} & := \{\phi^{L,bc}_{J_0;0}\} \cup \{\phi_{J_0;k}: 1\le k \le 2^{J_0}-1\} \cup \{\phi^{R,bc}_{J_0;2^{J_0}-1}\},\\
	\Psi^{x}_{j} & := \{\psi^{L,bc1}_{j;0},\psi^{L,bc2}_{j;0},\psi^{L,bc3}_{j;0}\} \cup \{\psi_{j;k}:2\le k \le 2^{j}-2\} \cup \{\psi^{R,bc1}_{j;2^j-1},\psi^{R,bc2}_{j;2^j-1},\psi^{R,bc3}_{j;2^j-1}\},
	\end{align*}
	where $\phi^{R,bc} = -\phi^{L,bc}(1-\cdot)$, $\psi^{R,bci}=\psi^{L,bci}(1-\cdot)$ for $i=1,2,3$, and
	\[
	\Phi^{y}_{J_0} := \Phi^{x}_{J_0} \cup \{\phi^{R}_{J_0;2^{J_0}-1}\}, \quad
	\Psi^{y}_{j} := \left(\Psi^{x}_{j} \backslash \{\psi^{R,bc3}_{j;2^j-1}\}\right) \cup \{\psi^{R}_{j;2^j-1}\},
	\]
	where $\phi^{R}=\phi^{L}(1-\cdot)$ and $\psi^{R}=\psi^{L}(1-\cdot)$. Then, $\mathcal{B}^x:=\Phi^{x}_{J_0} \cup \{\Psi_{j}^x : j\ge J_0\}$ and
	$\mathcal{B}^y:=\Phi^{y}_{J_0} \cup \{\Psi_{j}^y : j\ge J_0\}$ with $J_0 \ge 2$
	are Riesz wavelets in $\LpI{2}$.
	See \cref{fig:hmt} for their generators.
	
	\begin{figure}[htbp]
		\centering
		 \begin{subfigure}[b]{0.24\textwidth} \includegraphics[width=\textwidth]{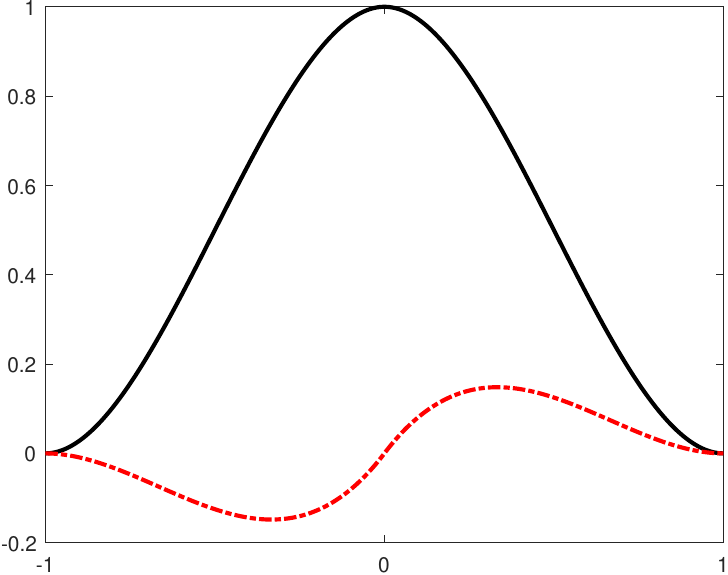}
			 \caption{$\phi=(\phi^{1},\phi^{2})^{\tp}$}
		\end{subfigure}
		 \begin{subfigure}[b]{0.24\textwidth} \includegraphics[width=\textwidth]{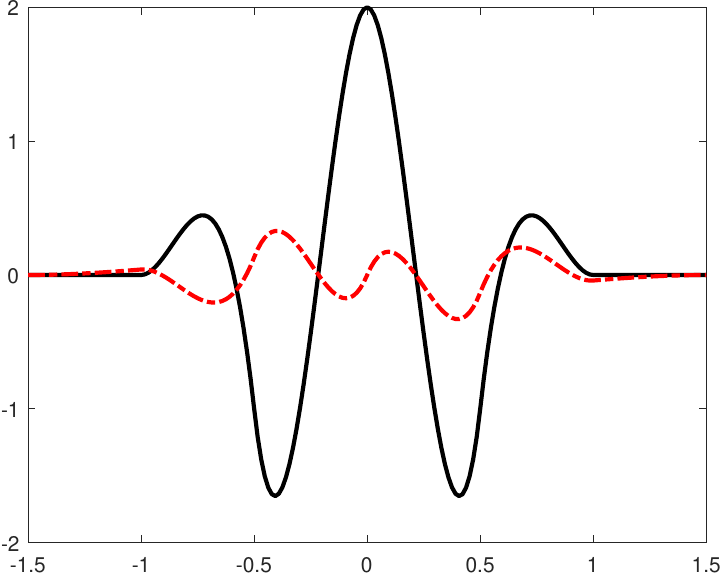}
			 \caption{$\psi=(\psi^{1},\psi^{2})^{\tp}$}
		\end{subfigure}
		 \begin{subfigure}[b]{0.24\textwidth} \includegraphics[width=\textwidth]{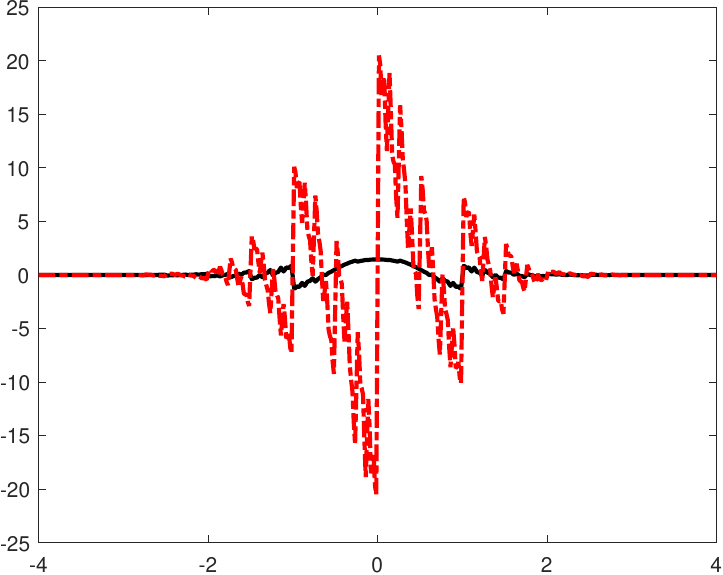}
			 \caption{$\tilde{\phi}=(\tilde{\phi}^{1},\tilde{\phi}^{2})^{\tp}$}
		\end{subfigure}
		 \begin{subfigure}[b]{0.24\textwidth} \includegraphics[width=\textwidth]{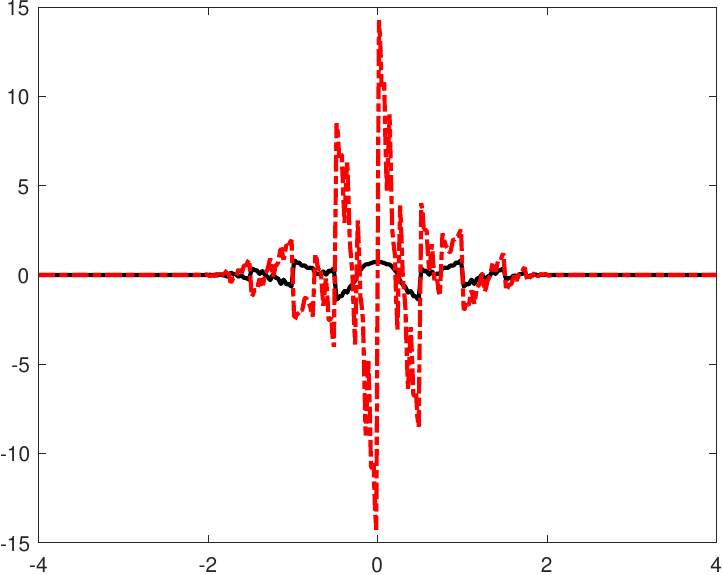}
			 \caption{$\tilde{\psi}=(\tilde{\psi}^{1},\tilde{\psi}^{2})^{\tp}$}
		\end{subfigure}
		 \begin{subfigure}[b]{0.24\textwidth}
			 \includegraphics[width=\textwidth]{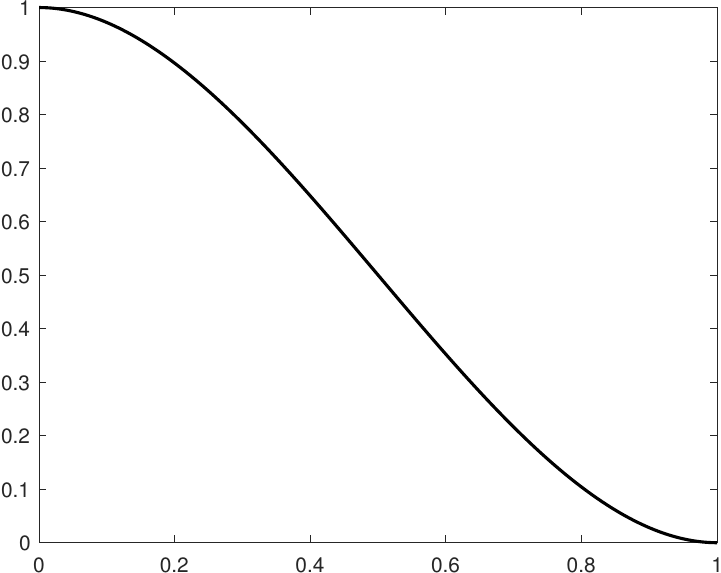}
			\caption{$\phi^{L}$}
		\end{subfigure}
		 \begin{subfigure}[b]{0.24\textwidth}
			 \includegraphics[width=\textwidth]{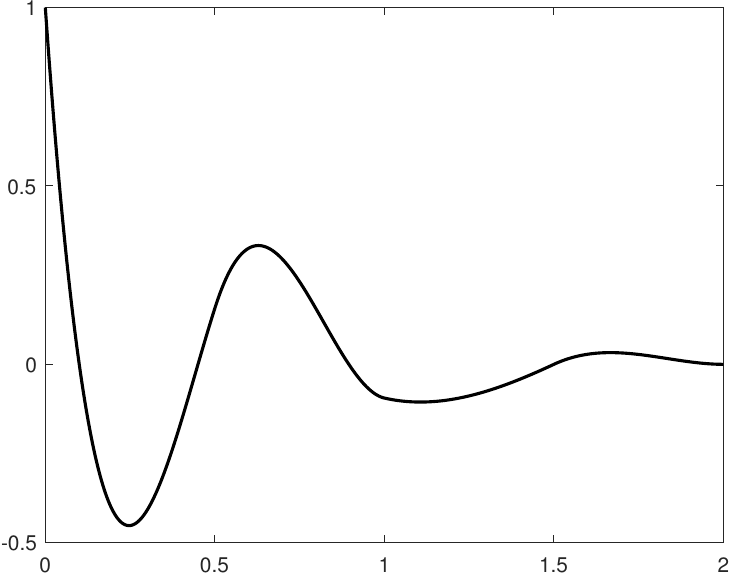}
			 \caption{$\psi^{L}$}
		\end{subfigure}
		 \begin{subfigure}[b]{0.24\textwidth}
			 \includegraphics[width=\textwidth]{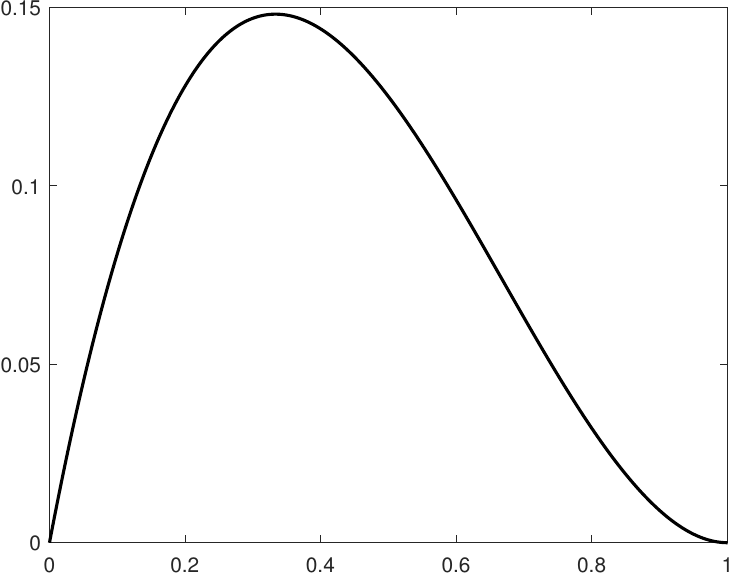}
			\caption{$\phi^{L,bc}$}
		\end{subfigure}
		 \begin{subfigure}[b]{0.24\textwidth}
			 \includegraphics[width=\textwidth]{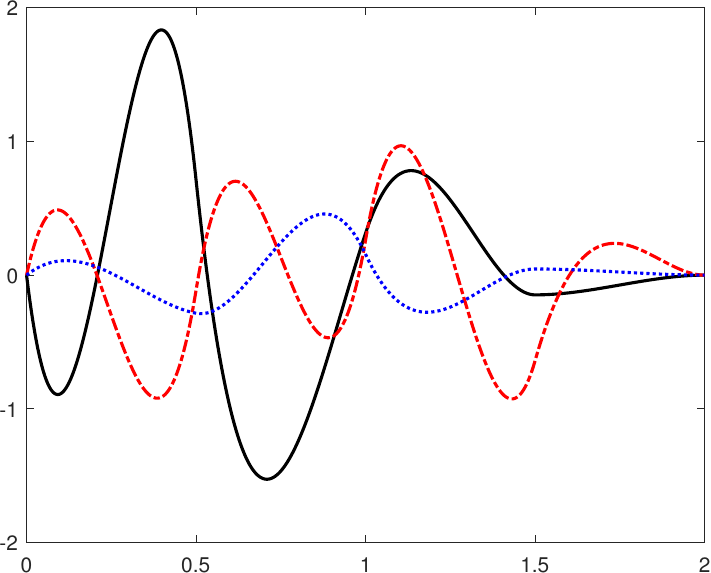}
			 \caption{$\psi^{L,bc1},\psi^{L,bc2},\psi^{L,bc3}$}
		\end{subfigure}
		\caption{The generators of Riesz wavelets $\Phi^{x}_{J_0} \cup \{\Psi_{j}^x : j\ge J_0\}$ and $\Phi^{y}_{J_0} \cup \{\Psi_{j}^y : j\ge J_0\}$ of $\LpI{2}$ 
			for $J_0 \ge 2$. The black (solid), red (dotted dashed), and blue (dotted) lines correspond to the first, second, and third components of a vector function respectively.}
		\label{fig:hmt}
	\end{figure}
	\end{example}
	
	\begin{example}\label{ex:r3}
	\normalfont
	Consider a biorthogonal wavelet $(\{\tilde{\phi};\tilde{\psi}\},\{\phi;\psi\})$ in $\Lp{2}$ with $\wh{\phi}(0)=(\tfrac{2}{5},\tfrac{3}{5},\tfrac{3}{5})^{\tp}$, $\wh{\tilde{\phi}}(0)=(\tfrac{5}{8},\tfrac{5}{8},\tfrac{5}{8})^{\tp}$, and a biorthogonal wavelet filter bank $(\{\tilde{a};\tilde{b}\},\{a;b\})$ given by
	\begin{align*}
		a=&{\left\{ \begin{bmatrix}  0 &\tfrac{1}{32} & 0 \\[0.3em]
			0 & 0 & 0\\[0.3em]
			0 & 0 & 0 \end{bmatrix},
			\begin{bmatrix}  -\tfrac{1}{32} & 0 & \tfrac{5}{32} \\[0.3em]
				0 & 0 & 0\\[0.3em]
				0 & 0 & 0 \end{bmatrix},
			\begin{bmatrix}  \tfrac{1}{2} & \tfrac{5}{32} & 0 \\[0.3em]
				0 & \tfrac{15}{32} & \tfrac{1}{2}\\[0.3em]
				0 & -\tfrac{5}{32} & 0 \end{bmatrix},
			\begin{bmatrix}  -\tfrac{1}{32} & 0 & \tfrac{1}{32} \\[0.3em]
				\tfrac{9}{32} & 0 & -\tfrac{5}{32}\\[0.3em]
				\tfrac{9}{32} & \tfrac{1}{2} & \tfrac{15}{32} \end{bmatrix}
			\right\}_{[-2,1]}},\\
		b=&{\left\{ \begin{bmatrix}  0 &\tfrac{1}{64} & -\tfrac{125}{6032} \\[0.3em]
				0 & 0 & 0\\[0.3em]
				0 & 0 & 0 \end{bmatrix},
			\begin{bmatrix}  -\tfrac{4335}{24128} & \tfrac{365}{2262} & -\tfrac{13453}{72384} \\[0.3em]
				-\tfrac{1}{4} & \tfrac{13}{36} & -\tfrac{11}{18}\\[0.3em]
				0 & 0 & 0 \end{bmatrix},
			\begin{bmatrix}  \tfrac{2703}{6032} & -\tfrac{13453}{72384} & \tfrac{365}{2262} \\[0.3em]
				0 & \tfrac{11}{18} & -\tfrac{13}{36}\\[0.3em]
				0 & \tfrac{1}{64} & \tfrac{1}{8} \end{bmatrix},
			\begin{bmatrix}  -\tfrac{4335}{24128} & -\tfrac{125}{6032} & \tfrac{1}{64} \\[0.3em]
				\tfrac{1}{4} & 0 & 0\\[0.3em]
				-\tfrac{27}{64} & \tfrac{1}{8} & \tfrac{1}{64} \end{bmatrix}
			\right\}_{[-2,1]}},\\
		\tilde{a}=&\left\{ \begin{bmatrix}  -\tfrac{33}{512} &\tfrac{47}{512} & \tfrac{7}{64} \\[0.3em]
				\tfrac{11}{3392} & -\tfrac{47}{10176} & -\tfrac{7}{1272}\\[0.3em]
				\tfrac{1375}{461312} & -\tfrac{5875}{1383936} & -\tfrac{875}{172992} \end{bmatrix},
			\begin{bmatrix}  -\tfrac{17}{256} & -\tfrac{209}{512} & \tfrac{259}{512} \\[0.3em]
				\tfrac{17}{5088} & \tfrac{209}{10176} & -\tfrac{259}{10176}\\[0.3em]
				\tfrac{125}{40704} & \tfrac{26125}{1383936} & -\tfrac{32375}{1383936} \end{bmatrix},
			\begin{bmatrix}  \tfrac{85}{128} & \tfrac{259}{512} & -\tfrac{209}{512} \\[0.3em]
				-\tfrac{211337}{4151808} & \tfrac{1032671}{4151808} & \tfrac{161873}{259488}\\[0.3em]
				\tfrac{34211}{1037952} & -\tfrac{371405}{4151808} & \tfrac{193255}{4151808} \end{bmatrix},\right.\\
			& \;\; \left.
			\begin{bmatrix}  -\tfrac{17}{256} & \tfrac{7}{64} & \tfrac{47}{512} \\[0.3em]
				\tfrac{2775}{13568} & \tfrac{193255}{4151808} & -\tfrac{371405}{4151808}\\[0.3em]
				\tfrac{2775}{13568} & \tfrac{161873}{259488} & \tfrac{1032671}{4151808} \end{bmatrix},
			\begin{bmatrix}  -\tfrac{33}{512} & 0 & 0 \\[0.3em]
				\tfrac{34211}{1037952} & -\tfrac{32375}{1383936} & \tfrac{26125}{1383936}\\[0.3em]
				-\tfrac{211337}{4151808} & -\tfrac{259}{10176} & \tfrac{209}{10176} \end{bmatrix},
			\begin{bmatrix}  0 & 0 & 0 \\[0.3em]
				\tfrac{125}{40704} & -\tfrac{875}{172992} & -\tfrac{5875}{1383936}\\[0.3em]
				\tfrac{17}{5088} & -\tfrac{7}{1272} & -\tfrac{47}{10176} \end{bmatrix}, \right.\\
			& \; \;
			\left.
			\begin{bmatrix}  0 & 0 & 0 \\[0.3em]
				\tfrac{1375}{461312} & 0 & 0\\[0.3em]
				\tfrac{11}{3392} & 0 & 0 \end{bmatrix}
			\right\}_{[-2,4]},\\
		\tilde{b}=&\left\{ \begin{bmatrix}  \tfrac{4147}{57664} &-\tfrac{17719}{172992} & -\tfrac{2639}{21624} \\[0.3em]
			\tfrac{33}{901} & -\tfrac{47}{901} & -\tfrac{56}{901}\\[0.3em]
			0 & 0 & 0 \end{bmatrix},
		\begin{bmatrix}  \tfrac{377}{5088} & \tfrac{78793}{172992} & -\tfrac{97643}{172992} \\[0.3em]
			\tfrac{2}{53} & \tfrac{209}{901} & -\tfrac{259}{901}\\[0.3em]
			0 & 0 & 0 \end{bmatrix},
		\begin{bmatrix}  \tfrac{16211}{43248} & -\tfrac{97643}{172992} & \tfrac{78793}{172992} \\[0.3em]
		0 & \tfrac{259}{901} & -\tfrac{209}{901}\\[0.3em]
		-\tfrac{128}{477} & \tfrac{29}{53} & \tfrac{12}{53} \end{bmatrix},\right.\\
		& \;\; \left.
		\begin{bmatrix}  \tfrac{377}{5088} & -\tfrac{2639}{21624} & -\tfrac{17719}{172992} \\[0.3em]
			-\tfrac{2}{53} & \tfrac{56}{901} & \tfrac{47}{901}\\[0.3em]
			-\tfrac{482}{477} & \tfrac{12}{53} & \tfrac{29}{53} \end{bmatrix},
		\begin{bmatrix}  \tfrac{4147}{57664} & 0 & 0 \\[0.3em]
			-\tfrac{33}{901} & 0 & 0\\[0.3em]
			-\tfrac{128}{477} & 0 & 0 \end{bmatrix}
		\right\}_{[-2,2]}.
	\end{align*}
	The analytic expression of $\phi = (\phi^{1},\phi^{2},\phi^{3})^{\tp}$ is
	\begin{align*}
	\phi^{1} & := (\tfrac{36}{5}x^{3} + \tfrac{72}{5}x^{2} + \tfrac{44}{5}x + \tfrac{8}{5}) \chi_{[-1,0)} + (-\tfrac{36}{5}x^{3} + \tfrac{72}{5}x^{2} - \tfrac{44}{5}x + \tfrac{8}{5})\chi_{[0,1]},\\
	\phi^{2} & := (\tfrac{108}{5}x^3 - 36x^2 + \tfrac{72}{5} x) \chi_{[0,1]}, \quad \text{and} \quad
	\phi^{3} := (-\tfrac{108}{5}x^{3} + \tfrac{144}{5}x^2 - \tfrac{36}{5}x)\chi_{[0,1]}.
	\end{align*}
	Note that $\text{fsupp}(\phi)=\text{fsupp}(\psi)=[-1,1]$, $\text{fsupp}(\tilde{\phi})=[-2,4]$, and $\text{fsupp}(\tilde{\psi})=[-2,3]$. Furthermore, $\sm(a)=1.5$, $\sm(\tilde{a})=2.056062$, $\sr(a)=\sr(\tilde{a})=4$, and its matching filters $v,\tilde{v} \in (l_0(\Z))^{1 \times 2}$ with $\wh{v}(0)\wh{\phi}(0)=\wh{\tilde{v}}(0)\wh{\tilde{\phi}}(0)=1$ are given by $\wh{v}(0)=(\tfrac{5}{8},\tfrac{5}{8},\tfrac{5}{8})$, $\wh{v}'(0)=(0,\tfrac{5}{24},\tfrac{5}{12})\ia$, $\wh{v}''(0)=(0,-\tfrac{5}{72},-\tfrac{5}{18})$, $\wh{v}'''(0)=(0,-\tfrac{5}{216},-\tfrac{5}{27})\ia$, $\wh{\tilde{v}}(0)=(\tfrac{2}{5},\tfrac{3}{5},\tfrac{3}{5})$, $\wh{\tilde{v}}'(0)=(0,\tfrac{3}{25},\tfrac{12}{25})\ia$, $\wh{\tilde{v}}''(0)=(-\tfrac{2}{75},0,-\tfrac{9}{25})$, and $\wh{\tilde{v}}'''(0)=(0,\tfrac{6}{175},-\tfrac{48}{175})\ia$. This implies $\phi \in H^{1}(\R)$ and $\tilde{\phi} \in H^{2}(\R)$. Let $\phi^{L}:=\tfrac{5}{8}\phi^{1}\chi_{[0,\infty)}$, $\phi^{L,bc1}:=\phi^2\chi_{[0,\infty)}$, and $\phi^{L,bc2}:=\phi^3\chi_{[0,\infty)}$. Note that $\phi^{L}:=\phi^{L}(2\cdot) + \tfrac{25}{128} \phi^2(2\cdot) +  [-\tfrac{5}{128},0,\tfrac{5}{128}]\phi(2\cdot-1)$, and
	\[
	\begin{bmatrix}
		\phi^{L,bc1}\\
		\phi^{L,bc2}
	\end{bmatrix}
	= \begin{bmatrix}
		\tfrac{15}{16} & 1\\
		-\tfrac{5}{16} & 0
	\end{bmatrix}
	\begin{bmatrix}
		\phi^{L,bc1}(2\cdot)\\
		\phi^{L,bc2}(2\cdot)
	\end{bmatrix}
	+ \begin{bmatrix}
		\tfrac{9}{16} & 0 & -\tfrac{5}{16}\\
		\tfrac{9}{16} & 1 & \tfrac{15}{16}
	\end{bmatrix}
	\phi(2\cdot-1).
	\]
	The direct approach in \cite[Theorem 4.2]{HM21a} yields
	\begin{align*}
		\psi^{L} & := \phi^{L}(2\cdot) - [\tfrac{182143}{577558},\tfrac{13249}{577558}] [\phi^{L,bc1}(2\cdot),\phi^{L,bc2}(2\cdot)]^{\tp} +
		 [\tfrac{3245513}{4620464},-\tfrac{20201527}{62376264},-\tfrac{900809}{62376264}] \phi(2\cdot -1),\\
		\psi^{L,bc} & := [\tfrac{819}{3200000},-\tfrac{40281}{30800000}][\phi^{L,bc1}(2\cdot),\phi^{L,bc2}(2\cdot)]^{\tp} +
		 [\tfrac{703}{3200000},\tfrac{113269}{92400000},\tfrac{58049}{147840000}]\phi(2\cdot -1)\\
		&  \quad +
		 [-\tfrac{47}{70000},-\tfrac{29}{56250},\tfrac{619}{2475000}]\phi(2\cdot-2).
	\end{align*}
	For $J_0 \ge 1$ and $j \ge J_0$, define
	\[
	\Phi^{x}_{J_0} := \{\phi^{L,bc1}_{J_0;0},\phi^{L,bc2}_{J_0;0}\} \cup \{\phi_{J_0;k}: 1\le k \le 2^{J_0}-1\},
	\quad
	\Psi^{x}_{j} := \{\psi^{L,bc}_{j;0}\} \cup \{\psi_{j;k}:1\le k \le 2^{j}-1\} \cup \{\psi^{R,bc}_{j;2^j-1}\},
	\]
	where $\psi^{R,bc}=\psi^{L,bc}(1-\cdot)$, and
	\[
	\Phi^{y}_{J_0} := \Phi^{x}_{J_0} \cup \{\phi^{R}_{J_0;2^{J_0}-1}\}, \quad
	\Psi^{y}_{j} := \left(\Psi^{x}_{j} \backslash \{\psi^{R,bc}_{j;2^j-1}\}\right) \cup \{\psi^{R}_{j;2^j-1}\},
	\]
	where $\phi^{R}=\phi^{L}(1-\cdot)$ and $\psi^{R}=\psi^{L}(1-\cdot)$. Then, $\mathcal{B}^{x}:=\Phi^{x}_{J_0} \cup \{\Psi_{j}^x : j\ge J_0\}$ 
	and $\mathcal{B}^{y}:=\Phi^{y}_{J_0} \cup \{\Psi_{j}^y : j\ge J_0\}$
	are Riesz wavelets in $\LpI{2}$. See \cref{fig:r3} for their generators.
	
	\begin{figure}[htbp]
		\centering
		 \begin{subfigure}[b]{0.24\textwidth} \includegraphics[width=\textwidth]{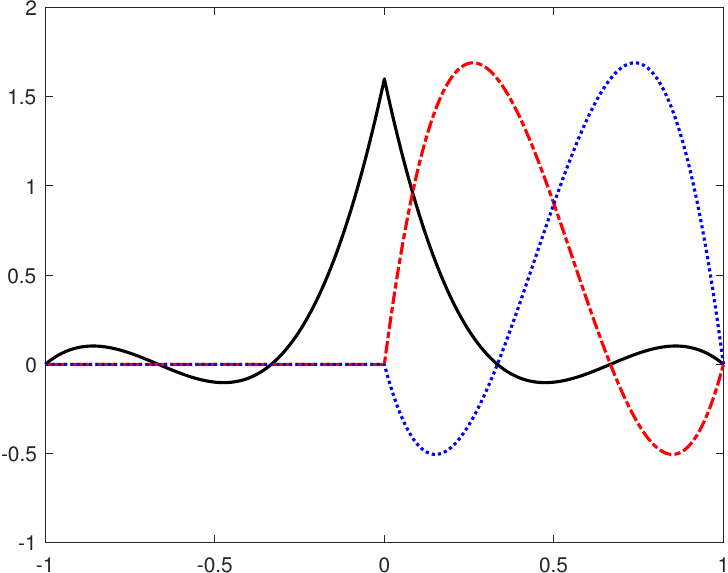}
			 \caption{$\phi=(\phi^{1},\phi^{2},\phi^{3})^{\tp}$}
		\end{subfigure}
		 \begin{subfigure}[b]{0.24\textwidth} \includegraphics[width=\textwidth]{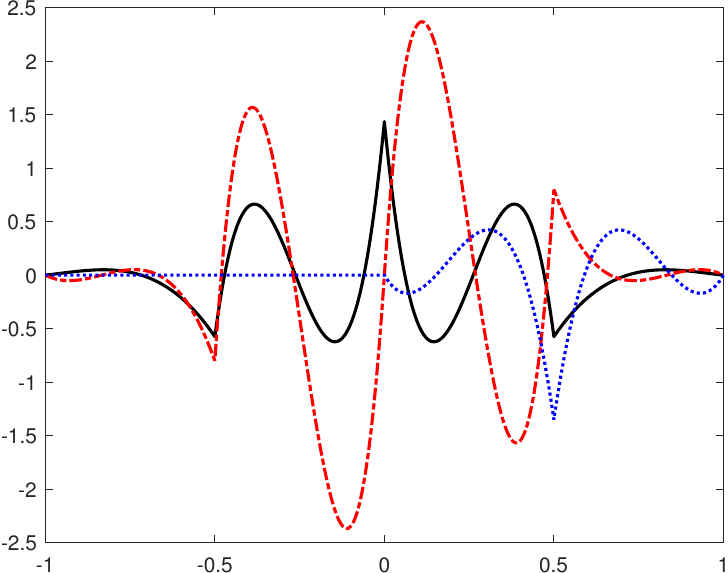}
			 \caption{$\psi=(\psi^{1},\psi^{2},\psi^{3})^{\tp}$}
		\end{subfigure}
		 \begin{subfigure}[b]{0.24\textwidth} \includegraphics[width=\textwidth]{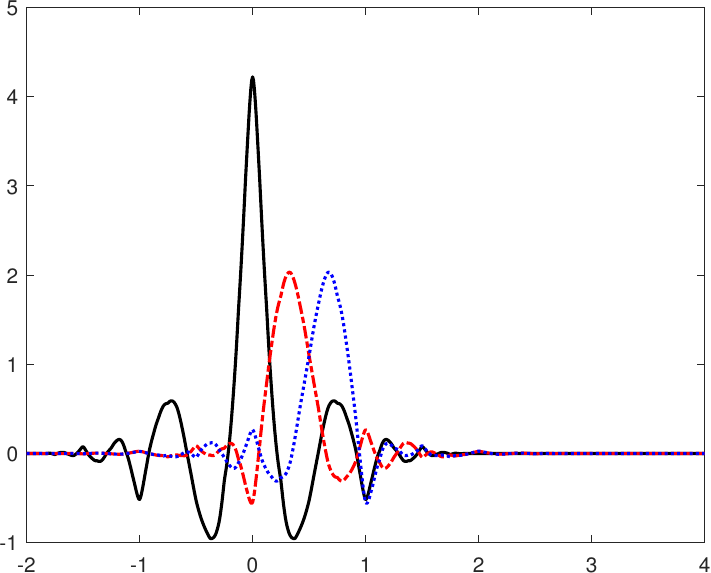}
			 \caption{$\tilde{\phi}=(\tilde{\phi}^{1},\tilde{\phi}^{2},\tilde{\phi}^{3})^{\tp}$}
		\end{subfigure}
		 \begin{subfigure}[b]{0.24\textwidth} \includegraphics[width=\textwidth]{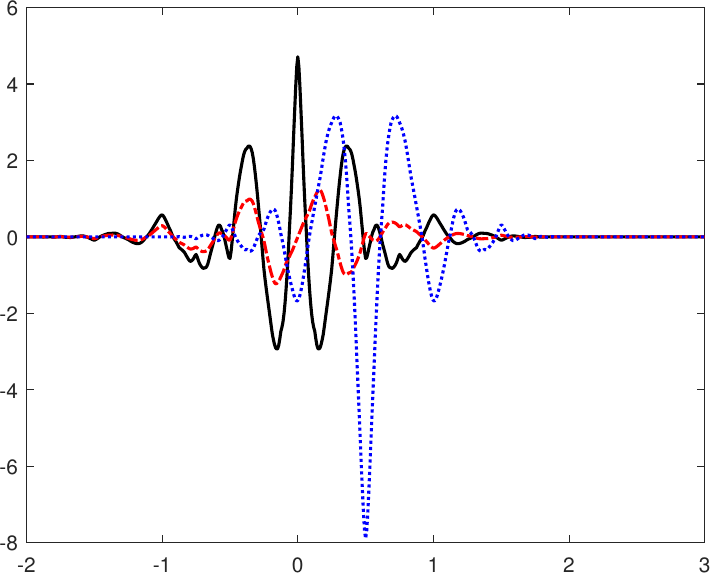}
			 \caption{$\tilde{\psi}=(\tilde{\psi}^{1},\tilde{\psi}^{2},\tilde{\psi}^{3})^{\tp}$}
		\end{subfigure}
		 \begin{subfigure}[b]{0.24\textwidth}
			 \includegraphics[width=\textwidth]{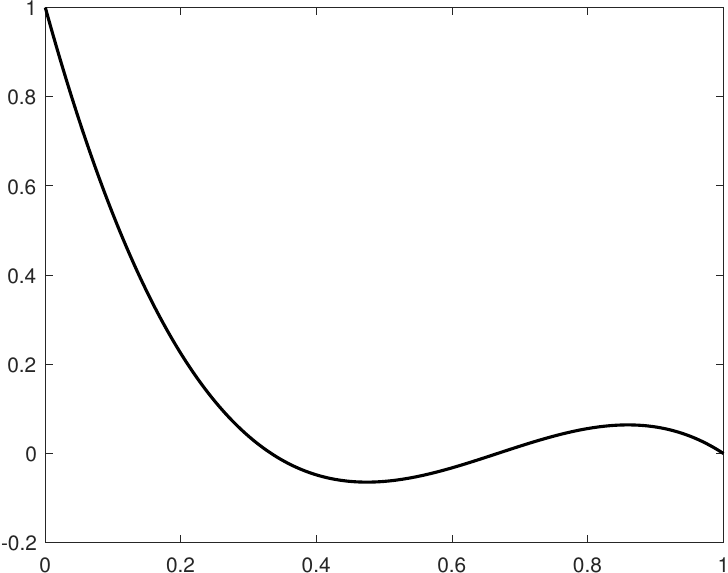}
			\caption{$\phi^{L}$}
		\end{subfigure}
		 \begin{subfigure}[b]{0.24\textwidth}
			 \includegraphics[width=\textwidth]{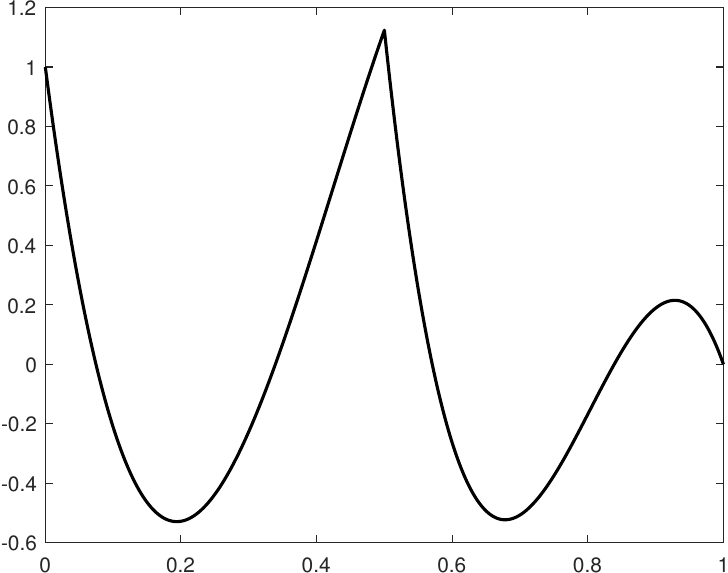}
			\caption{$\psi^{L}$}
		\end{subfigure}
		 \begin{subfigure}[b]{0.24\textwidth}
			 \includegraphics[width=\textwidth]{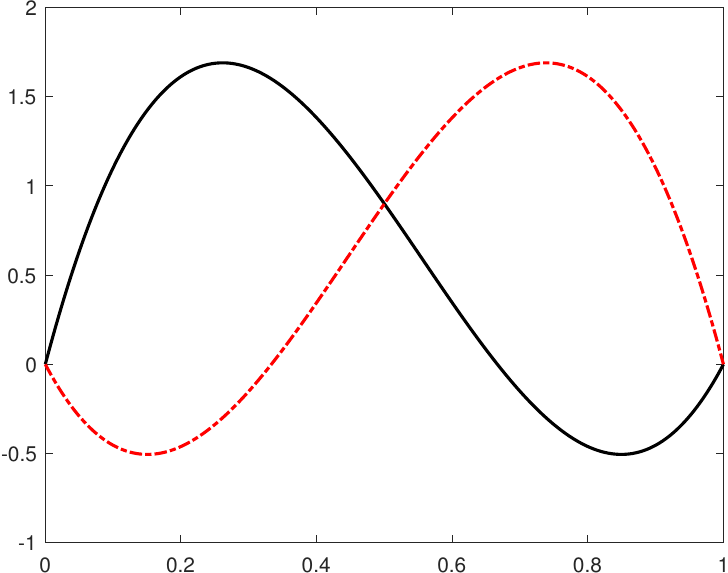}
			 \caption{$\phi^{L,bc1},\phi^{L,bc2}$}
		\end{subfigure}
		 \begin{subfigure}[b]{0.24\textwidth}
			 \includegraphics[width=\textwidth]{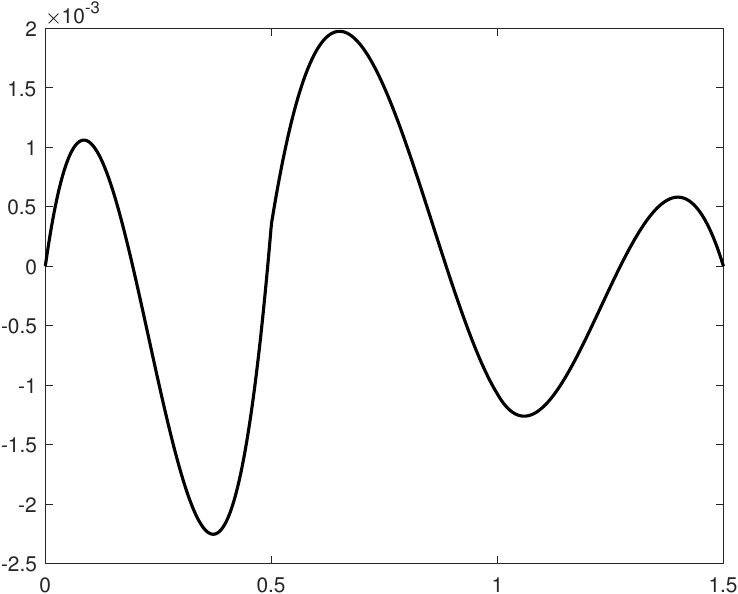}
			\caption{$\psi^{L,bc}$}
		\end{subfigure}
		\caption{The generators of Riesz wavelets $\Phi^{x}_{J_0} \cup \{\Psi_{j}^x : j\ge J_0\}$ and $\Phi^{y}_{J_0} \cup \{\Psi_{j}^y : j\ge J_0\}$ of $\LpI{2}$ 
			for $J_0 \ge 1$. The black (solid), red (dotted dashed), and blue (dotted) lines correspond to the first, second, and third components of a vector function respectively.}
		\label{fig:r3}
	\end{figure}
	\end{example}

	\section{Implementation}
	\label{sec:implement}
	
	In this section, we discuss some implementation details of our wavelet Galerkin method. By the refinability property, there exist well-defined matrices $A^{x}_{j,j'}$, $A_{j,j'}$, $A_{j,j'}^{R}$, $B^x_{j,j'}$, $B_{j,j'}$, and $B_{j,j'}^{R}$ such that the following relations hold
	\be \label{cavity:refinable}
	\Phi^{x}_{j} = A^x_{j,j'}
	\Phi^{x}_{j'}, \quad
	\Psi^{x}_{j} = B^x_{j,j'} \Phi^{x}_{j'}, \quad
	\Phi^{y}_{j} = \begin{bmatrix} A_{j,j'}\\ A^{R}_{j,j'} \end{bmatrix} \Phi^{y}_{j'}, \quad
	\text{and}
	\quad
	\Psi^{y}_{j} = \begin{bmatrix} B_{j,j'}\\ B^{R}_{j,j'} \end{bmatrix} \Phi^{y}_{j'}
	\quad
	\forall j < j'.
	\ee
	Note that $A^x_{j,j'}$ and $B^x_{j,j'}$ contain the filters of all refinable functions and wavelets satisfying the homogeneous Dirichlet boundary conditions at both endpoints. Meanwhile, $A_{j,j'}^{R}$ and $B_{j,j'}^{R}$ respectively contain the filters of right refinable functions and right wavelets satisfying no boundary conditions. For simplicity, we assume that $\phi^{R}(1)=\psi^{R}(1)=1$. It follows that $A_{j,j'}$ contains filters of left and interior refinable functions, and $B_{j,j'}$ contains filters of left and interior wavelets. For $J \ge J_0$, recall that $\mathcal{B}_{J_0,J}^{2D}:=\Phi^{2D}_{J_0} \cup \cup_{j=J_0}^{J-1} \Psi^{2D}_{j},$ where $\Phi^{2D}_{J_0}$ and $\Psi^{2D}_{j}$ are defined in \eqref{PhiPsi2D}. In our wavelet Galerkin scheme, our approximated solution is of the form $u_J=\sum_{\eta \in \mathcal{B}^{2D}_{J_0,J}} c_{\eta} \eta.$ Let $\otimes$ denote the Kronecker product if applied to matrices, $0_{m \times n}$ denote an $m\times n$ zero matrix, rows($\cdot$) denote the number of rows of a given matrix, and vec($\cdot$) denote the standard vectorization operation. Plugging the approximated solution into the weak formulation \eqref{cavity:weak}, using test functions in  $\mathcal{B}^{2D}_{J_0,J}$, and recalling the relations in \eqref{cavity:refinable}, we obtain the linear system
	\be \label{cavity:AcF}
	\left(R  \left[\langle \nabla v, \nabla w\rangle - \langle \ka^2 v,w \rangle \right]_{v,w \in (2^{-J} \Phi^{2D}_{J}) } R^{\tp} - T\right) C=F,
	\ee
	where $R:=[R_1^{\tp},\cdots,R_{2(J-J_0)+2}^{\tp}]^{\tp}$ with $R_1:=
	A^x_{J_0,J}
	\otimes A_{J_0,J}$, $R_{J-J_0+2}:=
	A^x_{J_0,J}
	\otimes A_{J_0,J}^{R}$,
	\begin{align*}
		& \begin{aligned}
		R_{\ell} :=\begin{bmatrix}
			B^x_{J_0+\ell-2,J} \otimes A_{J_0+\ell-2,J}\\
			A^x_{J_0+\ell-2,J} \otimes B_{J_0+\ell-2,J}\\
			B^x_{J_0+\ell-2,J} \otimes B_{J_0+\ell-2,J}
		\end{bmatrix}, \quad
		R_{J-J_0+\ell+1} :=\begin{bmatrix}
			B^x_{J_0+\ell-2,J} \otimes A^{R}_{J_0+\ell-2,J}\\
			A^x_{J_0+\ell-2,J} \otimes B^{R}_{J_0+\ell-2,J}\\
			B^x_{J_0+\ell-2,J} \otimes B^{R}_{J_0+\ell-2,J}
			\end{bmatrix},
			\quad
			2\le \ell \le J-J_0+1,
		\end{aligned}\\
	& S:=[S_0^\tp,\dots,S_{J-J_0}^{\tp}]^{\tp} \quad \text{with} \quad S_0:= 2^{J_0/2}A^{x}_{J_0,J}, \quad
		S_{\ell} := 2^{(J_0 + \ell -1)/2} \begin{bmatrix}
			B^x_{J_0+\ell-1,J}\\
			A^x_{J_0+\ell-1,J}\\
			B^x_{J_0+\ell-1,J}
		\end{bmatrix}, \quad
		1 \le \ell \le J-J_0,\\
		& T := \begin{bmatrix}
			 0_{(\text{rows}(R)-\text{rows}(S)) \times (\text{rows}(R)-\text{rows}(S)) } & 0_{(\text{rows}(R)-\text{rows}(S)) \times \text{rows}(S)}\\
			0_{\text{rows}(S) \times (\text{rows}(R)-\text{rows}(S))} & S\left[ \langle\mathcal{T}(\eta),\zeta\rangle_{\Gamma} \right]_{\eta,\zeta \in (2^{-J} \Phi^{x}_{J}) }S^{\tp}
		\end{bmatrix},\\
		& F := \begin{bmatrix} 0_{(\text{rows}(R)-\text{rows}(S)) \times 1}\\ S \left[\langle g, \zeta \rangle_{\Gamma} \right]_{\zeta \in (2^{-J}\Phi^{x}_J)}\end{bmatrix}
		-R \text{vec}\left([\langle f, w \rangle_{\Omega}]_{w \in (2^{-J} \Phi^{2D}_{J}) }\right),
	\end{align*}
	and $C$ denotes the coefficients $\{c_{\eta}\}_{\eta \in \mathcal{B}^{2D}_{J_0,J}}$ properly arranged in a vector form.
	
	We make some important remarks regarding the assembly of the linear system. First, we further normalize each element in $\mathcal{B}^{2D}_{J_0,J}$ by $|a(\cdot,\cdot)|^{-1/2}$, where $a(\cdot,\cdot)$ is defined in \eqref{cavity:weak}. This makes the modulus of all diagonal entries of the coefficient matrix on the left-hand side of \eqref{cavity:AcF} equal to $1$. Second, we note that the assembly of the linear system can be done efficiently by exploiting the refinability structure. The inner products are computed only for the refinable functions at the highest scale level (i.e., elements of $\Phi^{x}_J$ and $\Phi^{y}_J$). Third, following \cite[Remark 4.1]{LMS13}, we rewrite the non-local boundary condition as
	\be \label{cavity:Tv3}
	\mathcal{T}(v) = \int_{0}^{1} \ln(|x-x'|) q_{0}(x-x')v(x') dx' + \int_{0}^{1} q_1(x-x')v(x')dx' + \frac{1}{\pi} \ddashint_0^1 \frac{v(x')}{|x-x'|^2} dx',
	\ee
	where
	\[
	q_0(s):=\frac{\ia \ka H^{(1)}_1(\ka|s|)}{2|s|} + \frac{\ka J_{1}(\ka |s|)}{\pi |s|}\ln(|s|) - \frac{1}{\pi|s|^2}, \quad q_{1}(s):=- \frac{\ka J_{1}(\ka |s|)}{\pi |s|},
	\]
	and $J_1$ is the first order Bessel function of the first kind. Note that $q_0(s)$ and $q_1(s)$ are even analytic functions. The first integral in \eqref{cavity:Tv3} is only weakly singular. After properly partitioning this integral so that the weak singularity appears on an endpoint, we can use a combination of the Gauss-Legendre and double exponential quadratures to compute it. The second integral in \eqref{cavity:Tv3} can be handled by the Gauss-Legendre quadrature. Recall that if $v \in C^{1,\alpha}([c,d])$ (i.e., the first derivative of $v$ is $\alpha$-H\"older continuous on the unit interval with $0 < \alpha \le 1$), then
	\be \label{cavity:hadamard}
	\ddashint_c^d \frac{v(x')}{(x-x')^2} dx' := \lim_{\epsilon \rightarrow 0} \left( \int_{c}^{x-\epsilon} \frac{v(x')}{(x-x')^2} dx' + \int_{x+\epsilon}^{d} \frac{v(x')}{(x-x')^2} dx' - \frac{2v(x)}{\epsilon} \right).
	\ee
	See \cite{LS10,WWLS08}. Then, the third integral of \eqref{cavity:Tv3} can be exactly computed by \eqref{cavity:hadamard}, since the Riesz wavelets we employ have analytic expressions.
	
	\section{Numerical Experiments}
	\label{sec:exp}
	In what follows, we present several numerical experiments to compare the performance of our wavelet Galerkin method using $\mathcal{B}^{2D}_{J_0,J}:=\Phi^{2D}_{J_0}\cup \cup_{j=J_0}^{J-1} \Psi^{2D}_j$, and the FEM using $\Phi^{2D}_J$, where $\Phi^{2D}_{J_0}$ and $\Psi^{2D}_j$ are defined in \eqref{PhiPsi2D} of \cref{thm:H1:2D}.
We shall focus on the behavior of the coefficient matrix coming from each scheme. The relative errors reported below are in terms of $2$-norm. In case the exact solution $u$ is known, we define
	\[
	 \|u-u_J\|_{2}^2:=2^{-22}\sum_{i=1}^{2^{11}} \sum_{j=1}^{2^{11}} |u(x_i,y_j) - u_J(x_i,y_j)|^2,
	\]
which numerically approximates the true error $\|u-u_J\|_{\LpO{2}}^2$ in the $\LpO{2}$ norm using the very fine grid with the mesh size $h=2^{-11}$,
where $x_i:=ih$ and $y_j:=jh$.
If the exact solution $u$ is unknown, then we replace the exact solution $u$ above by the next level numerical solution $u_{J+1}$. 

We record the convergence rates (listed under `Order'), which are obtained by calculating
\[
\text{order} = 2 \log_{2}(\|u-u_J\|_{2}/\|u-u_{J+1}\|_{2}) (\log_2(N_{J+1}/N_{J}))^{-1},
\]
where $N_J$ corresponds to the degrees of freedom at a given scale level $J$ (or, the number of rows in the coefficient matrix), if the exact solution $u$ is known. Otherwise, if the exact solution $u$ is unknown, we replace $u-u_J$ with $u_{J+1}-u_J$. In order to accurately obtain the convergence rates of the wavelet Galerkin method, we have to use the backslash command in MATLAB to obtain the numerical solutions $u_J$ and then use them to compute the errors $\|u-u_J\|_{2}/\|u\|_{2}$ and $\|u_{J+1}-u_J\|_{2}$.
Since $\text{span}(\Phi^{2D}_J) = \text{span}(\mathcal{B}^{2D}_{J_0,J})$ and $\#\Phi^{2D}_J=\#\mathcal{B}^{2D}_{J_0,J}$, the numerical solution $u_J$ is theoretically the same if $\mathcal{B}^{2D}_{J_0,J}$ is replaced by $\Phi^{2D}_J$ and a direct solver is used to compute $u_J$.
Our numerical computation indicates that
the relative errors of the numerical solutions $u_J$ obtained from our wavelet method and the FEM using a direct solver are practically identical.
Hence, we only report a set of relative errors for a given wavelet basis and a given wavenumber.

We list the largest singular values $\sigma_{\max}$, the smallest singular values $\sigma_{\min}$, and the condition numbers (i.e., the ratio $\sigma_{\max}/\sigma_{\min}$ of the largest and smallest singular values) of the coefficient matrices coming from our wavelet method and the FEM. The `Iter' column lists the number of GMRES iterations (with zero as its initial value) needed for the relative residuals to fall below $10^{-8}$. 
	
	\subsection{Constant wavenumbers}
	This subsection contains two examples, where $\ka = \ka_0 > 0$ is a constant wavenumber.
	
	\begin{example} \label{cavity:ex:4to16}
		\normalfont Consider the model problem \eqref{cavity:model}, where $\mathcal{T}$ is defined in \eqref{cavity:Tu}, and $f$ and $g$ are chosen such that $u=\exp(xy)\sin(\ka x)\sin((\ka + \pi/2)y)$. Additionally, we let $\ka = 4\pi, 8\pi, 16\pi$. See \cref{cavity:tab:4pi,cavity:tab:8pi,cavity:tab:16pi} for the numerical results. The same problem was also considered in \cite{BS05,DSZ13}.
		
	\begin{table}[htbp]
	\begin{center}
	\begin{tabular}{c c| c c c c | c c c c | c c}
	\hline
	\hline
	\multicolumn{12}{c}{$\ka=4\pi$}\\
	\hline
	$J$ & $N_J$ & \multicolumn{4}{c}{$\Phi^{2D}_{J}$ (\cref{ex:sr3})} \vline & \multicolumn{4}{c}{$\mathcal{B}_{2,J}^{2D}$ (\cref{ex:sr3})} \vline & $\tfrac{\|u-u_J\|_2}{\|u\|_2}$ & Order \\
	\hline
	 & & $\sigma_{\max}$ & $\sigma_{\min}$ & $\tfrac{\sigma_{\max}}{\sigma_{\min}}$ & Iter  & $\sigma_{\max}$ & $\sigma_{\min}$ & $\tfrac{\sigma_{\max}}{\sigma_{\min}}$ & Iter &  &  \\
	 \hline
	 5& 4032 & 1.56 & 4.50E-4 & 3.46E+3  & 418 & 4.14 & 2.10E-2 & 1.97E+2 & 161 & 6.11E-4 &  \\
	 6& 16256 & 1.55 & 1.12E-4 & 1.39E+4  & 836 & 4.28 & 1.85E-2 & 2.32E+2 & 169 & 7.63E-5 & 2.98 \\
	 7& 65280 & 1.55 & 2.80E-5 & 5.55E+4  & 1668 & 4.39 & 1.71E-2 & 2.56E+2  & 182 & 9.53E-6 & 2.99\\
	 8& 261632 & 1.55 & 8.26E-6 & 2.22E+5 & 3325 & 4.48 & 1.64E-2 & 2.73E+2 & 188 & 1.19E-6 & 3.00\\
	 \hline
	 $J$ & $N_J$ & \multicolumn{4}{c}{$\Phi^{2D}_{J}$ (\cref{ex:hmt})} \vline & \multicolumn{4}{c}{$\mathcal{B}_{4,J}^{2D}$ (\cref{ex:hmt})} \vline & $\tfrac{\|u-u_J\|_2}{\|u\|_2}$ & Order \\
	 \hline
	 & & $\sigma_{\max}$ & $\sigma_{\min}$ & $\tfrac{\sigma_{\max}}{\sigma_{\min}}$ & Iter  & $\sigma_{\max}$ & $\sigma_{\min}$ & $\tfrac{\sigma_{\max}}{\sigma_{\min}}$ & Iter &  &  \\
	 \hline
	 4 & 1056 & 2.41 & 8.42E-3 & 2.86E+2 & 117 & 2.41 & 8.42E-3 & 2.86E+2 & 117 & 5.24E-4 &  \\
	 5 & 4160 & 2.43 & 2.03E-3 & 1.20E+3 & 235 & 3.45 & 7.62E-3 & 4.53E+2 & 188 & 3.78E-5 & 3.84 \\
	 6 & 16512 & 2.44 & 5.03E-3 & 4.85E+3 & 472 & 4.16 & 6.32E-3 & 6.57E+2 & 214 & 2.48E-6 & 3.95 \\
	 7 & 65792 & 2.44 & 1.26E-4 & 1.94E+4 & 942 & 4.26 & 6.15E-3 & 6.92E+2 & 226 & 1.58E-7 & 3.98\\
	 \hline
	 $J$ & $N_J$ & \multicolumn{4}{c}{$\Phi^{2D}_{J}$ (\cref{ex:r3})} \vline & \multicolumn{4}{c}{$\mathcal{B}_{2,J}^{2D}$ (\cref{ex:r3})} \vline & $\tfrac{\|u-u_J\|_2}{\|u\|_2}$ & Order \\
	 \hline
	 & & $\sigma_{\max}$ & $\sigma_{\min}$ & $\tfrac{\sigma_{\max}}{\sigma_{\min}}$ & Iter  & $\sigma_{\max}$ & $\sigma_{\min}$ & $\tfrac{\sigma_{\max}}{\sigma_{\min}}$ & Iter &  & \\
	 \hline
	 4 & 2256 & 2.17 & 5.39E-4 & 4.03E+3 & 444 & 4.33 & 8.56E-3 & 5.06E+2 & 168 & 2.35E-4 & \\
	 5 & 9120 & 2.16 & 1.34E-4 & 1.61E+4  & 892 & 4.51 & 8.57E-3 & 5.26E+2 & 179 & 1.48E-5 & 3.96 \\
	 6 & 36672 & 2.16 & 3.34E-5 & 6.46E+4  & 1782 & 4.63 & 8.57E-3 & 5.41E+2 & 184 & 9.28E-7 & 3.98 \\
	 7 & 147072 & 2.16 & 8.35E-6 & 2.58E+5 & 3555 & 4.73 & 8.57E-3 & 5.52E+2 & 188 & 5.82E-8 & 3.99\\
	 \hline
	\end{tabular}
	\caption{Singular values, condition numbers, relative errors, and iteration numbers of GMRES (for relative residuals to be less than $10^{-8}$) for \cref{cavity:ex:4to16} with $\ka =4\pi$. The FEM uses the standard bases $\Phi^{2D}_{J}$ while the wavelet method employs the Riesz bases $\mathcal{B}_{J_0,J}^{2D}$.
Note that $\mbox{span}(\Phi^{2D}_{J})=\mbox{span}(\mathcal{B}_{J_0,J}^{2D})$.
}
	\label{cavity:tab:4pi}
	\end{center}
	\end{table}
	\begin{table}[htbp]
	\begin{center}
		 \begin{tabular}{c c| c c c c | c c c c | c c}
			\hline
			\hline
			 \multicolumn{12}{c}{$\ka=8\pi$}\\
			\hline
			$J$ & $N_J$ & \multicolumn{4}{c}{$\Phi^{2D}_{J}$ (\cref{ex:sr3})} \vline & \multicolumn{4}{c}{$\mathcal{B}_{3,J}^{2D}$ (\cref{ex:sr3})} \vline & $\tfrac{\|u-u_J\|_2}{\|u\|_2}$ & Order \\
			\hline
			& & $\sigma_{\max}$ & $\sigma_{\min}$ & $\tfrac{\sigma_{\max}}{\sigma_{\min}}$ & Iter & $\sigma_{\max}$ & $\sigma_{\min}$ & $\tfrac{\sigma_{\max}}{\sigma_{\min}}$ & Iter &  &  \\
			\hline
			5 & 4032 & 1.58 & 5.23E-4 & 3.02E+3 & 700 & 3.93 & 1.25E-2 & 3.15E+2 & 248 & 4.39E-3 &  \\
			6 & 16256 & 1.56 & 1.24E-4 & 1.26E+4 & 1403 & 4.12 & 1.20E-2 & 3.43E+2  & 264 & 5.46E-4 & 2.99\\
			7 & 65280 & 1.55 & 3.08E-5 & 5.04E+4  & 2806 & 4.26 & 1.20E-2 & 3.55E+2  & 278 & 6.82E-5 & 2.99\\
			8 & 261632 & 1.55 & 7.70E-6 & 2.01E+5  & 5610 & 4.37 & 1.20E-2 & 3.64E+2  & 288 & 8.53E-6 & 2.99\\
			\hline
			$J$ & $N_J$ & \multicolumn{4}{c}{$\Phi^{2D}_{J}$ (\cref{ex:hmt})} \vline & \multicolumn{4}{c}{$\mathcal{B}_{4,J}^{2D}$ (\cref{ex:hmt})} \vline & $\tfrac{\|u-u_J\|_2}{\|u\|_2}$ & Order \\
			\hline
			& & $\sigma_{\max}$ & $\sigma_{\min}$ & $\tfrac{\sigma_{\max}}{\sigma_{\min}}$ & Iter  & $\sigma_{\max}$ & $\sigma_{\min}$ & $\tfrac{\sigma_{\max}}{\sigma_{\min}}$ & Iter & & \\
			\hline
			4 & 1056 & 2.30 & 1.15E-2 & 1.99E+2 &  179 & 2.30 & 1.15E-2 & 1.99E+2  & 179 & 5.25E-3 & \\
			5 & 4160 & 2.41 & 2.33E-3 & 1.03E+3 &  381 & 3.45 & 7.64E-3 & 4.52E+2  & 260 & 4.66E-4 & 3.53 \\
			6 & 16512 & 2.43 & 5.59E-4 & 4.35E+3 & 778 & 4.15 & 6.35E-3 & 6.54E+2  & 290 & 3.31E-5 & 3.84 \\
			7 & 65792 & 2.44 & 1.39E-4 & 1.76E+4 & 1562 & 4.25 & 6.17E-3 & 6.89E+2  & 310 & 2.24E-6 & 3.90 \\
			\hline
			$J$ & $N_J$ & \multicolumn{4}{c}{$\Phi^{2D}_{J}$ (\cref{ex:r3})} \vline & \multicolumn{4}{c}{$\mathcal{B}_{3,J}^{2D}$ (\cref{ex:r3})} \vline & $\tfrac{\|u-u_J\|_2}{\|u\|_2}$ & Order \\
			\hline
			& & $\sigma_{\max}$ & $\sigma_{\min}$ & $\tfrac{\sigma_{\max}}{\sigma_{\min}}$ & Iter  & $\sigma_{\max}$ & $\sigma_{\min}$ & $\tfrac{\sigma_{\max}}{\sigma_{\min}}$ & Iter &  &  \\
			\hline
			4 & 2256 & 2.21 & 6.19E-4 & 3.57E+3 & 731 & 4.03 & 2.40E-3 & 1.68E+3 & 430 & 3.17E-3 & \\
			5 & 9120 & 2.17 & 1.48E-4 & 1.46E+4  & 1482 & 4.33 & 2.38E-3 & 1.82E+3  & 460 & 2.03E-4 & 3.93 \\
			6 & 36672 & 2.16 & 3.69E-5 & 5.87E+4  & 2976 & 4.51 & 2.38E-3 & 1.90E+3 & 482 & 1.28E-5 & 3.97  \\
			7 & 147072 & 2.16 & 9.20E-6 & 2.35E+5  & 5953 & 4.61 & 2.38E-3 & 1.95E+3  & 496 & 1.00E-6 & 3.67 \\
			\hline
		\end{tabular}
		\caption{Singular values, condition numbers, relative errors, and iteration numbers of GMRES (for relative residuals to be less than $10^{-8}$) for \cref{cavity:ex:4to16} with $\ka =8\pi$. The FEM uses the standard bases $\Phi^{2D}_{J}$ while the wavelet method employs the Riesz bases $\mathcal{B}_{J_0,J}^{2D}$.
Note that $\mbox{span}(\Phi^{2D}_{J})=\mbox{span}(\mathcal{B}_{J_0,J}^{2D})$.
}
		\label{cavity:tab:8pi}
	\end{center}
	\end{table}
	\begin{table}[htbp]
		\begin{center}
			\begin{tabular}{c c| c c c c | c c c c | c c}
			\hline
			\hline
			 \multicolumn{12}{c}{$\ka=16\pi$}\\
			\hline
			$J$ & $N_J$ & \multicolumn{4}{c}{$\Phi^{2D}_{J}$ (\cref{ex:sr3})} \vline & \multicolumn{4}{c}{$\mathcal{B}_{4,J}^{2D}$ (\cref{ex:sr3})} \vline & $\tfrac{\|u-u_J\|_2}{\|u\|_2}$ & Order \\
			\hline
			& & $\sigma_{\max}$ & $\sigma_{\min}$ & $\tfrac{\sigma_{\max}}{\sigma_{\min}}$ & Iter  & $\sigma_{\max}$ & $\sigma_{\min}$ & $\tfrac{\sigma_{\max}}{\sigma_{\min}}$  & Iter & & \\
			\hline
			6 & 16256 & 1.58 & 1.55E-4 & 1.02E+4  & 2571  & 3.93 & 3.70E-3 & 1.06E+3   & 814 & 4.17E-3 &  \\
			7 & 65280 & 1.56 & 3.35E-5 & 4.66E+4  & 5158  & 4.12 & 3.25E-3 & 1.27E+3  & 854 &  5.28E-4 & 2.97 \\
			8 &  261632 & 1.55 & 8.26E-6 & 1.88E+5  & 10321  & 4.26 & 3.23E-3 & 1.32E+3   & 884 &  1.11E-4 & 2.25\\
			\hline
			$J$ & $N_J$ & \multicolumn{4}{c}{$\Phi^{2D}_{J}$ (\cref{ex:hmt})} \vline & \multicolumn{4}{c}{$\mathcal{B}_{4,J}^{2D}$ (\cref{ex:hmt})} \vline & $\tfrac{\|u-u_J\|_2}{\|u\|_2}$ & Order\\
			\hline
			& & $\sigma_{\max}$ & $\sigma_{\min}$ & $\tfrac{\sigma_{\max}}{\sigma_{\min}}$ & Iter  & $\sigma_{\max}$ & $\sigma_{\min}$ & $\tfrac{\sigma_{\max}}{\sigma_{\min}}$ & Iter & & \\
			\hline
			4 & 1056 & 5.19 & 2.98E-2 & 1.74E+2 & 363 & 5.19 & 2.98E-2 & 1.74E+2  & 363 & 5.21E-2 &  \\
			5 & 4160 & 2.30 & 3.27E-3 & 7.03E+2 & 635 & 5.19 & 7.69E-3 & 6.76E+2 & 495  & 5.01E-3 & 3.42\\
			6 & 16512 & 2.41 & 6.26E-4 & 3.85E+3 & 1381 & 5.20 & 6.46E-3 & 8.05E+2   &  558 & 4.50E-4 & 3.50 \\
			\hline
			$J$ & $N_J$ & \multicolumn{4}{c}{$\Phi^{2D}_{J}$ (\cref{ex:r3})} \vline & \multicolumn{4}{c}{$\mathcal{B}_{3,J}^{2D}$ (\cref{ex:r3})} \vline & $\tfrac{\|u-u_J\|_2}{\|u\|_2}$ & Order\\
			\hline
			& & $\sigma_{\max}$ & $\sigma_{\min}$ & $\tfrac{\sigma_{\max}}{\sigma_{\min}}$ & Iter  & $\sigma_{\max}$ & $\sigma_{\min}$ & $\tfrac{\sigma_{\max}}{\sigma_{\min}}$ & Iter &  & \\
			\hline
			4 & 2256 & 2.41 & 2.39E-3 & 1.01E+3 & 1232 & 5.43 & 8.68E-3 & 6.26E+2 & 661 & 5.15E-2 &  \\
			5 & 9120 & 2.21  & 1.71E-4 & 1.30E+4 & 2664 & 5.68 & 6.26E-3 & 9.08E+2 & 712 & 2.97E-3 & 4.08 \\
			6 & 36672 & 2.17 & 3.98E-5 & 5.46E+4 & 5385 & 5.87 & 6.06E-3 & 9.68E+2  & 736 & 2.09E-4 & 3.81\\
			\hline
			\end{tabular}
			\caption{Singular values, condition numbers, relative errors, and iteration numbers of GMRES (for relative residuals to be less than $10^{-8}$) for \cref{cavity:ex:4to16} with $\ka = 16\pi$. The FEM uses the standard bases $\Phi^{2D}_{J}$ while the wavelet method employs the Riesz bases $\mathcal{B}_{J_0,J}^{2D}$.
Note that $\mbox{span}(\Phi^{2D}_{J})=\mbox{span}(\mathcal{B}_{J_0,J}^{2D})$.
}
			\label{cavity:tab:16pi}
		\end{center}
	\end{table}
	\end{example}
	
		 \begin{example}\label{cavity:ex:32}
		\normalfont Consider the scattering problem \eqref{model:scatter}, where $\ka=32\pi$ and $\theta=0$. See \cref{cavity:tab:32pi} and \cref{fig:gmres} for the numerical results. A similar problem was considered in \cite{DSZ13}, but with $\Omega=(0,1) \times (0,0.25)$. See also \cite{BS05,HMS16}, To further improve the condition number of the coefficient matrix in our numerical experiments, we replace $\phi^{L,bc1}$, $\phi^{L,bc2}$, and $\phi^{L}$ in \cref{ex:B4} with $\phi^{L,bc1} + \tfrac{85}{100} \phi^{L,bc2}$, $-\tfrac{1}{2}\phi^{L,bc1}+\tfrac{11}{10}\phi^{L,bc2}$, and $\phi^{L} - \tfrac{13}{14}\phi^{L,bc1}+\tfrac{8}{14}\phi^{L,bc2}$ respectively. Due to the large number of iterations and lengthy computation time for the FEM, we only report the GMRES relative residuals for $\Phi^{2D}_J$. More specifically, the `Tol' column associated with $\Phi_J^{2D}$ lists the relative residuals, when GMRES is used as an iterative solver with the maximum number of iterations listed in the `Iter' column associated with $\mathcal{B}^{2D}_{J_0,J}$ for $J_0,J \in \N$ with $J_0<J$.
		\begin{table}[htbp]
			\begin{center}
				 \resizebox{\textwidth}{!}{\begin{tabular}{c c| c c c c | c c c c | c c}
						\hline
						\hline
						 \multicolumn{12}{c}{$\ka=32\pi$}\\
						\hline
						$J$ & $N_J$ & \multicolumn{4}{c}{$\Phi^{2D}_{J}$ (\cref{ex:B3})} \vline & \multicolumn{4}{c}{$\mathcal{B}_{6,J}^{2D}$ (\cref{ex:B3})} \vline & $\|u_J-u_{J+1}\|_2$ & Order\\
						\hline
						& & $\sigma_{\max}$ & $\sigma_{\min}$ & $\tfrac{\sigma_{\max}}{\sigma_{\min}}$ & Tol  & $\sigma_{\max}$ & $\sigma_{\min}$ & $\tfrac{\sigma_{\max}}{\sigma_{\min}}$  & Iter &  & \\
						\hline
						6 & 4160 & 6.96 & 1.37E-2 & 5.07E+2 & $<$1E-8 & 6.96 & 1.37E-2 & 5.07E+2 & 1132 & 9.60E-1 & 4.30 \\
						7 & 16512 & 1.93 & 8.58E-4 & 2.25E+3  & 1.06E-7  & 6.96 & 5.45E-3  & 1.28E+3  & 1727 & 4.95E-2 & 4.09 \\
						8 & 65792 & 1.84 & 1.30E-4 & 1.42E+4  & 2.03E-4 & 6.96 & 2.78E-3  & 2.50E+3  & 1960 & 2.94E-3 & \\
						9 & 262656 & 1.83 & 3.13E-5 & 5.85E+4  & 2.03E-3 & 6.96 & 2.28E-3  & 3.05E+3  & 2027 &  &  \\
						\hline
						$J$ & $N_J$ & \multicolumn{4}{c}{$\Phi^{2D}_{J}$ (\cref{ex:B4})} \vline & \multicolumn{4}{c}{$\mathcal{B}_{6,J}^{2D}$ (\cref{ex:B4})} \vline & $\|u_J-u_{J+1}\|_2$ & Order \\
						\hline
						& & $\sigma_{\max}$ & $\sigma_{\min}$ & $\tfrac{\sigma_{\max}}{\sigma_{\min}}$ & Tol  & $\sigma_{\max}$ & $\sigma_{\min}$ & $\tfrac{\sigma_{\max}}{\sigma_{\min}}$  & Iter &  & \\
						\hline
						6 & 4290 & 3.41E+1 & 2.80E-2 & 1.22E+3 & $<$1E-8  & 3.41E+1 &  2.80E-2 & 1.22E+3  & 1195 & 1.09E-1 & 6.55 \\
						7 & 16770 & 3.06 & 9.02E-4 & 3.39E+3 & 1.21E-6 & 3.41E+1 & 1.05E-2 & 3.23E+3 & 1484 & 1.25E-3 & 2.64 \\
						8 & 66306 & 3.17 & 2.16E-4 & 1.47E+4 & 3.55E-4 & 3.41E+1 & 8.36E-3 & 4.07E+3 & 1680 & 2.04E-4  &  \\
						9 & 263682 & 3.20 & 5.37E-5 & 5.96E+4 & 2.19E-3 & 3.41E+1 & 8.24E-3  & 4.15E+3 &  1745 &  &  \\
						\hline
						$J$ & $N_J$ & \multicolumn{4}{c}{$\Phi^{2D}_{J}$ (\cref{ex:sr3})} \vline & \multicolumn{4}{c}{$\mathcal{B}_{5,J}^{2D}$ (\cref{ex:sr3})} \vline & $\|u_J-u_{J+1}\|_2$ & Order\\
						\hline
						& & $\sigma_{\max}$ & $\sigma_{\min}$ & $\tfrac{\sigma_{\max}}{\sigma_{\min}}$  & Tol  & $\sigma_{\max}$ & $\sigma_{\min}$ & $\tfrac{\sigma_{\max}}{\sigma_{\min}}$  & Iter & & \\
						\hline
						6 & 16256 & 1.67 & 6.25E-4 & 2.67E+3  & 8.84E-5 & 3.57 & 2.80E-3 & 1.26E+3 & 2176 & 5.83E-1 & 3.78 \\
						7 & 65280& 1.58 & 5.52E-5 & 2.86E+4  & 1.33E-3 & 3.93 & 1.32E-3 & 2.98E+3  & 2362 & 4.21E-2 & 3.90\\
						8 & 261632 & 1.56 & 8.94E-6 & 1.74E+5  & 2.48E-3 & 4.12 & 8.70E-4 & 4.74E+3  & 2483 & 2.80E-3 & \\
						9 & 1047552 & 1.55 & 2.15E-6 & 7.22E+5  & 8.26E-3 & 4.26 & 8.41E-4 & 5.07E+3  & 2591 &  & \\
						\hline
						$J$ & $N_J$ & \multicolumn{4}{c}{$\Phi^{2D}_{J}$ (\cref{ex:hmt})} \vline & \multicolumn{4}{c}{$\mathcal{B}_{5,J}^{2D}$ (\cref{ex:hmt})} \vline & $\|u_J-u_{J+1}\|_2$ & Order\\
						\hline
						& & $\sigma_{\max}$ & $\sigma_{\min}$ & $\tfrac{\sigma_{\max}}{\sigma_{\min}}$ & Tol  & $\sigma_{\max}$ & $\sigma_{\min}$ & $\tfrac{\sigma_{\max}}{\sigma_{\min}}$  & Iter & & \\
						\hline
						6 & 16512& 2.30 & 1.02E-3 & 2.25E+3  & 1.21E-6 & 5.23 & 3.42E-3 & 1.53E+3 & 1524 &  2.26E-2 & 4.80 \\
						7 & 65792 & 2.41 & 1.64E-4 & 1.47E+4  & 7.85E-4 & 5.23 & 2.66E-3 & 1.97E+3 & 1722 &  8.19E-4 & 2.46 \\
						8 & 262656 & 2.43 & 3.90E-5 & 6.24E+4 & 2.66E-3 & 5.24 & 2.64E-3 & 1.99E+3  & 1781 & 1.49E-4 &\\
						9 & 1049600 & 2.44 & 9.65E-6 & 2.53E+5 & 6.57E-3 & 5.25 & 2.64E-3 & 1.99E+3 & 1873 &  & \\
						\hline
						$J$ & $N_J$ & \multicolumn{4}{c}{$\Phi^{2D}_{J}$ (\cref{ex:r3})} \vline & \multicolumn{4}{c}{$\mathcal{B}_{4,J}^{2D}$ (\cref{ex:r3})} \vline & $\|u_J-u_{J+1}\|_2$ & Order\\
						\hline
						& & $\sigma_{\max}$ & $\sigma_{\min}$ & $\tfrac{\sigma_{\max}}{\sigma_{\min}}$ & Tol  & $\sigma_{\max}$ & $\sigma_{\min}$ & $\tfrac{\sigma_{\max}}{\sigma_{\min}}$ & Iter & &\\
						\hline
						6 & 36672 & 2.21 & 4.92E-5 & 4.50E+4  & 1.46E-3 & 5.71 & 1.79E-3 & 3.19E+3  & 2435 & 1.21E-2 & 4.74\\
						7 & 147072 & 2.17 & 1.04E-5 & 2.09E+5  & 2.87E-3 & 5.90 & 1.57E-3 & 3.76E+3  & 2512 & 4.48E-4 & 1.78 \\
						8 & 589056 & 2.16 & 2.57E-6 & 8.42E+5  & 9.33E-3 & 6.02 & 1.56E-3 & 3.85E+3  & 2558 & 1.30E-4 & \\
						9 & 2357760 & 2.16 & 6.41E-7 & 3.37E+6  & 1.91E-2 & 6.09 & 1.57E-3 & 3.89E+3  & 2598 &  & \\
						\hline
				\end{tabular}}
				\caption{Singular values, condition numbers, relative errors, and iteration numbers of GMRES (for relative residuals to be less than $10^{-8}$) for \cref{cavity:ex:32} with $\ka = 32\pi$.
The `Tol' column associated with $\Phi_j^{2D}$ lists the relative residuals, when GMRES is used as an iterative solver with the maximum number of iterations listed in the `Iter' column associated with $\mathcal{B}_{J_0,J}^{2D}$ for $J_0 \in \N$. The FEM uses the standard bases $\Phi^{2D}_{J}$ while the wavelet method employs the Riesz bases $\mathcal{B}_{J_0,J}^{2D}$.
Note that $\mbox{span}(\Phi^{2D}_{J})=\mbox{span}(\mathcal{B}_{J_0,J}^{2D})$.
}
				\label{cavity:tab:32pi}
			\end{center}
		\end{table}
		\begin{figure}[htbp]
			\centering
			 \begin{subfigure}[b]{0.19\textwidth} \includegraphics[width=\textwidth]{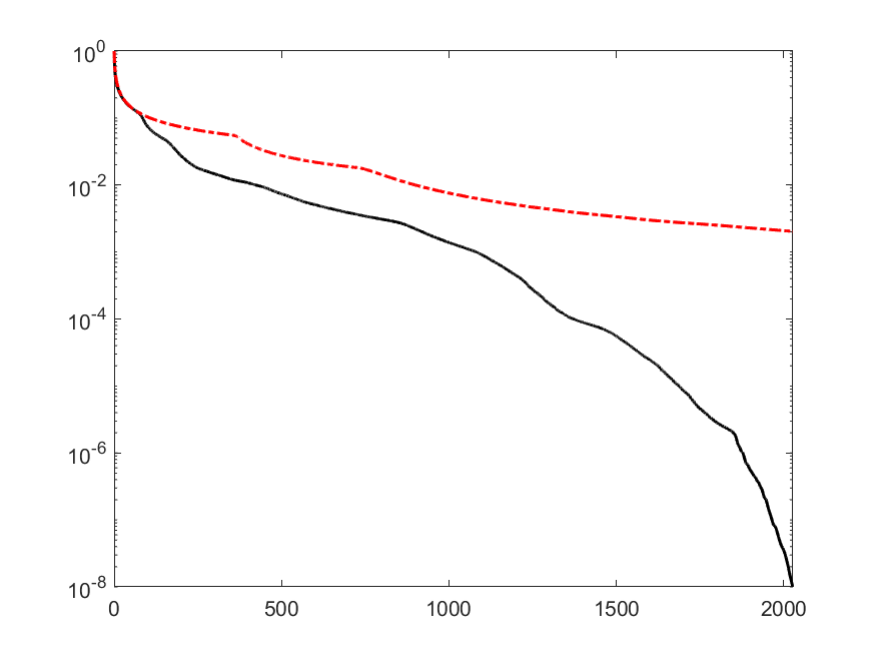}
			\end{subfigure}
			 \begin{subfigure}[b]{0.19\textwidth} \includegraphics[width=\textwidth]{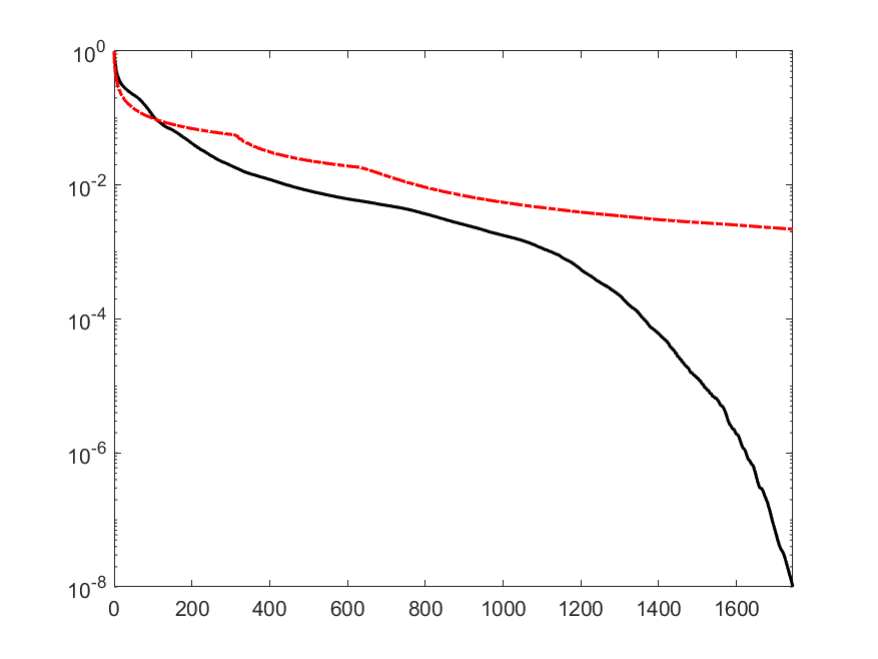}
			\end{subfigure}
			 \begin{subfigure}[b]{0.19\textwidth} \includegraphics[width=\textwidth]{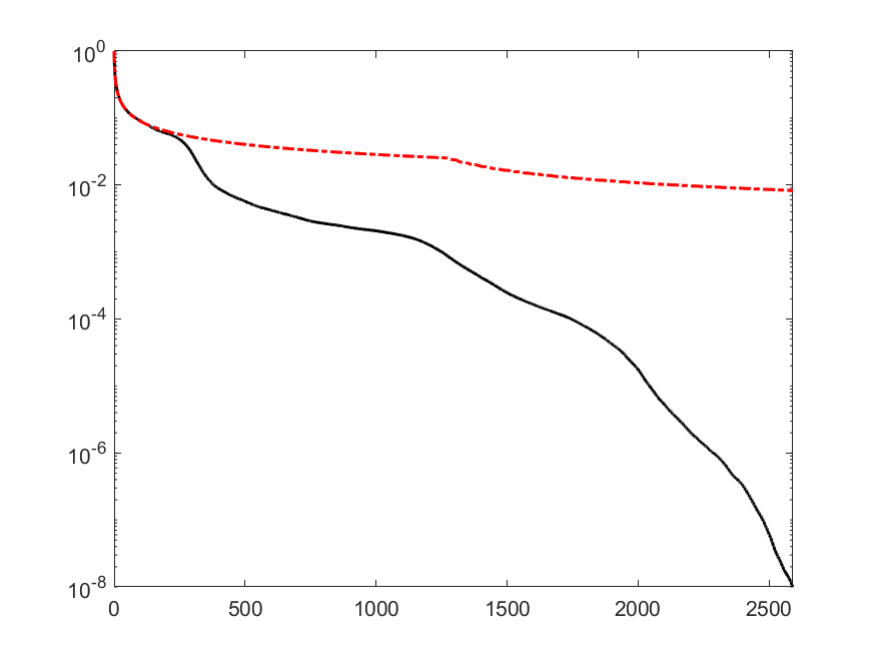}
			\end{subfigure}
			 \begin{subfigure}[b]{0.19\textwidth} \includegraphics[width=\textwidth]{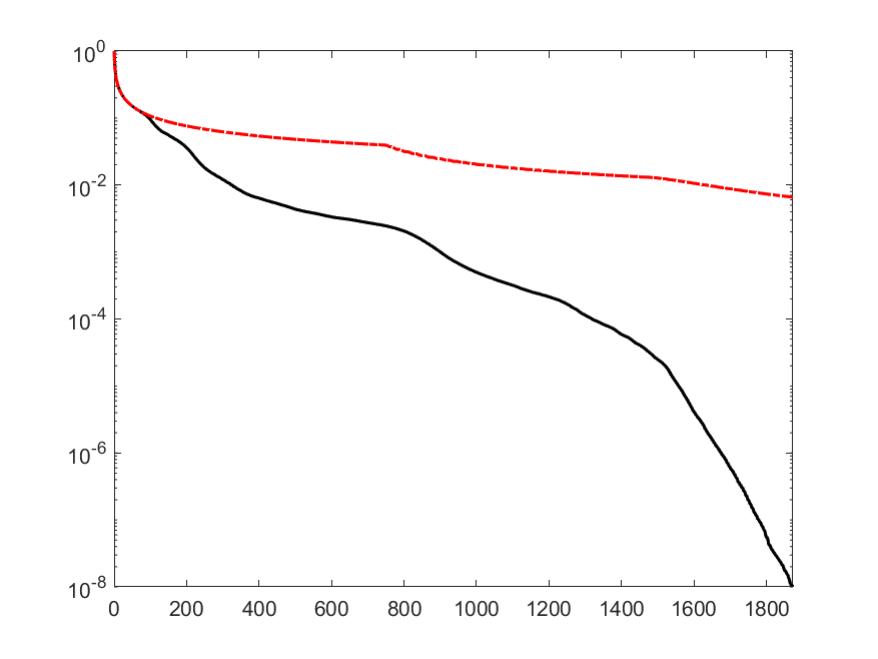}
			\end{subfigure}
			 \begin{subfigure}[b]{0.19\textwidth}
				 \includegraphics[width=\textwidth]{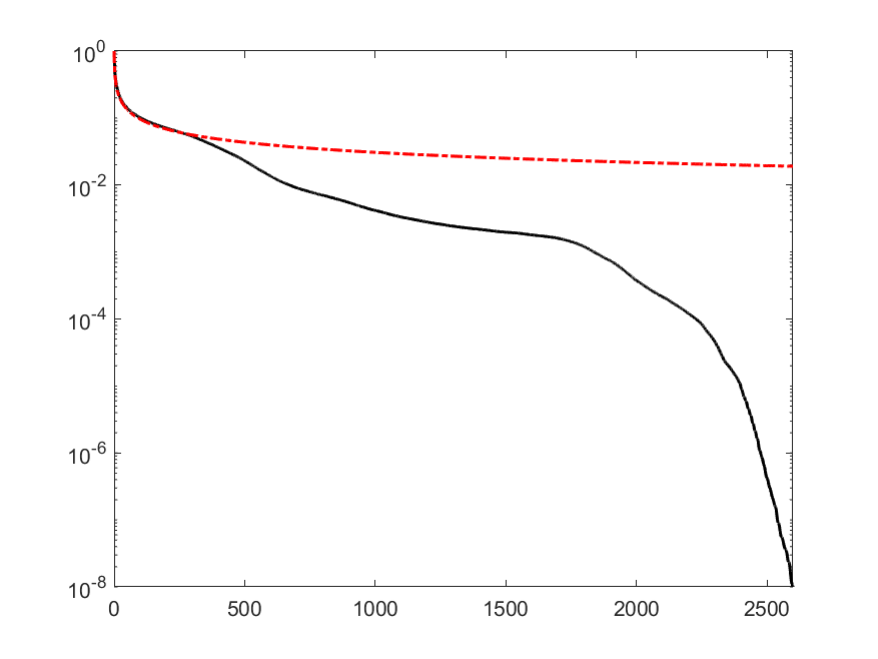}
			\end{subfigure}
			\caption{GMRES relative residuals for \cref{cavity:ex:32} with the maximum number of iterations displayed in the `Iter' column of \cref{cavity:tab:32pi} and $J=9$. \cref{ex:B3,ex:B4,ex:sr3,ex:hmt,ex:r3} are displayed in order from left to right. The black (solid) and red (dotted dashed) lines correspond to GMRES residuals of our wavelet method and the FEM respectively.}
			\label{fig:gmres}
		\end{figure}
	\end{example}
	
	We now discuss the results of our numerical experiments observed in \cref{cavity:ex:4to16,cavity:ex:32}. First, we observe that the largest singular values of coefficient matrices of our wavelet method and the FEM do not change much as the scale level $J$ increases (or equivalently the mesh size decreases). Second, the smallest singular values of coefficient matrices of the wavelet Galerkin method are uniformly bounded away from zero: they seem to converge to a positive number as the mesh size decreases. This is in sharp contrast to the smallest singular values of the FEM coefficient matrices, which seem to become arbitrarily small as the mesh size decreases. In particular, the smallest singular values are approximately a quarter of what they were before as we halve the grid size of each axis. Not surprisingly, the condition numbers of the coefficient matrices of the FEM quadruple as we increase the scale level, while those of the wavelet Galerkin method plateau. When an iterative scheme is used (here, we used GMRES), we see two distinct behaviours. In the FEM, the number of iterations needed for the GMRES relative residuals to fall below $10^{-8}$ doubles as we increase the scale level, while fixing the wavenumber. On the other hand, in the wavelet Galerkin method, the number of iterations needed for the GMRES relative residuals to fall below $10^{-8}$ is practically independent of the size of the coefficient matrix; moreover, we often see situations, where only a tenth (or even less) of the number of iterations is needed. In \cref{cavity:tab:32pi}, we see that the GMRES relative residuals of the FEM coefficient matrix fail to be within $10^{-8}$ at the given maximum iterations in the `Iter' column, while those of the wavelet Galerkin method are within $10^{-8}$. See \cref{fig:gmres}. The convergence rates in \cref{cavity:tab:4pi,cavity:tab:8pi,cavity:tab:16pi} are in accordance with the approximation orders of the bases. Meanwhile, the convergence rates in \cref{cavity:tab:32pi} are affected by the corner singularities near $\Gamma$. This behaviour was also documented in \cite{HMS16}.
	
	\subsection{Variable wavenumbers} This subsection contains four examples, where the wavenumbers are variable.
	
	\begin{example} \label{cavity:ex:variable}
		\normalfont
		Consider the scattering problem \eqref{model:scatter}, where $\kappa_0 = 16 \pi$, $\theta=0$, and $\varepsilon_r := 0.25 \chi_{\{r(x,y) \le 1/4\}} + \chi_{\Omega \cap \{1/4 < r(x,y) < 3/8\}} (-383.75 + 0.75 (8640 r(x,y) -57600 r(x,y)^2 + 189440 r(x,y)^3 -307200 r(x,y)^4 + 196608 r(x,y)^5)) + \chi_{\Omega \cap \{r(x,y) \ge 3/8\}}$ with $r(x,y):=\sqrt{(x-1/2)^2 + (y-1/2)^2}$. Note that for this example the wavenumber $\ka = \ka_0 \sqrt{\varepsilon_r}$ varies continuously in both $x$ and $y$ directions. Additionally, it is non-separable. See \cref{cavity:tab:variable} for numerical results and \cref{fig:kappa2_uapprox_var} for relevant plots.
		\begin{table}[htbp]
			\begin{center}
				 \resizebox{\textwidth}{!}{\begin{tabular}{c c| c c c c | c c c c | c c}
					\hline
					$J$ & $N_J$ & \multicolumn{4}{c}{$\Phi^{2D}_{J}$ (\cref{ex:B2})} \vline & \multicolumn{4}{c}{$\mathcal{B}_{2,J}^{2D}$ (\cref{ex:B2})} \vline & $\|u_{J}-u_{J+1}\|_2$ & Order \\
					\hline
					& & $\sigma_{\max}$ & $\sigma_{\min}$ & $\tfrac{\sigma_{\max}}{\sigma_{\min}}$ & Iter  & $\sigma_{\max}$ & $\sigma_{\min}$ & $\tfrac{\sigma_{\max}}{\sigma_{\min}}$ & Iter &  &  \\
					\hline
					6& 4032 & 1.58 & 1.70E-3 & 9.33E+2 & 647 & 2.02E+1 & 8.66E-3 & 2.33E+3 & 406 & 2.33 & 1.97 \\
					7& 16256 & 1.52 & 5.04E-4 & 3.01E+3 & 1284 & 2.03E+1 & 8.13E-3 & 2.49E+3 & 436  &  5.89E-1 & 2.05  \\
					8& 65280 & 1.50 & 1.28E-4 & 1.17E+4 & 2558 & 2.03E+1 & 8.38E-3 & 2.42E+3 & 449 & 1.41E-1 &  2.01  \\
					9& 261632 & 1.50 & 3.23E-5 & 4.64E+4 & 5102 & 2.03E+1 & 8.47E-3 &  2.40E+3 & 459 & 3.50E-2 &   \\
					\hline
				\end{tabular}}
				\caption{Singular values, condition numbers, relative errors, and iteration numbers of GMRES (for relative residuals to be less than $10^{-8}$) for \cref{cavity:ex:variable}. The FEM uses the standard bases $\Phi^{2D}_{J}$ while the wavelet method employs the Riesz bases $\mathcal{B}_{J_0,J}^{2D}$. Note that $\mbox{span}(\Phi^{2D}_{J})=\mbox{span}(\mathcal{B}_{J_0,J}^{2D})$.}
				 \label{cavity:tab:variable}
			\end{center}
		\end{table}
		\begin{figure}[htbp]
			\centering
			 \begin{subfigure}[b]{0.32\textwidth} \includegraphics[width=\textwidth]{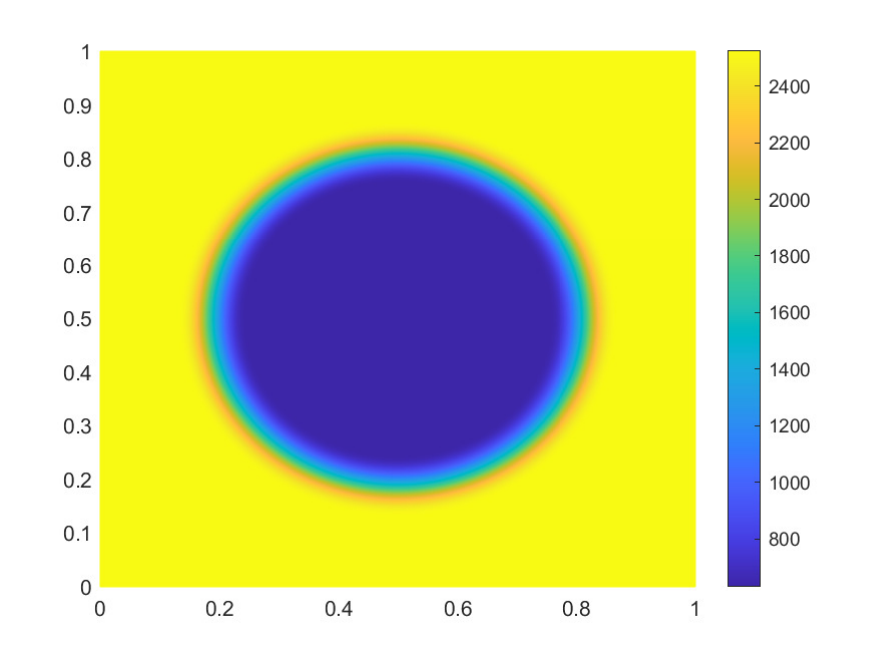}
			\end{subfigure}
			 \begin{subfigure}[b]{0.32\textwidth} \includegraphics[width=\textwidth]{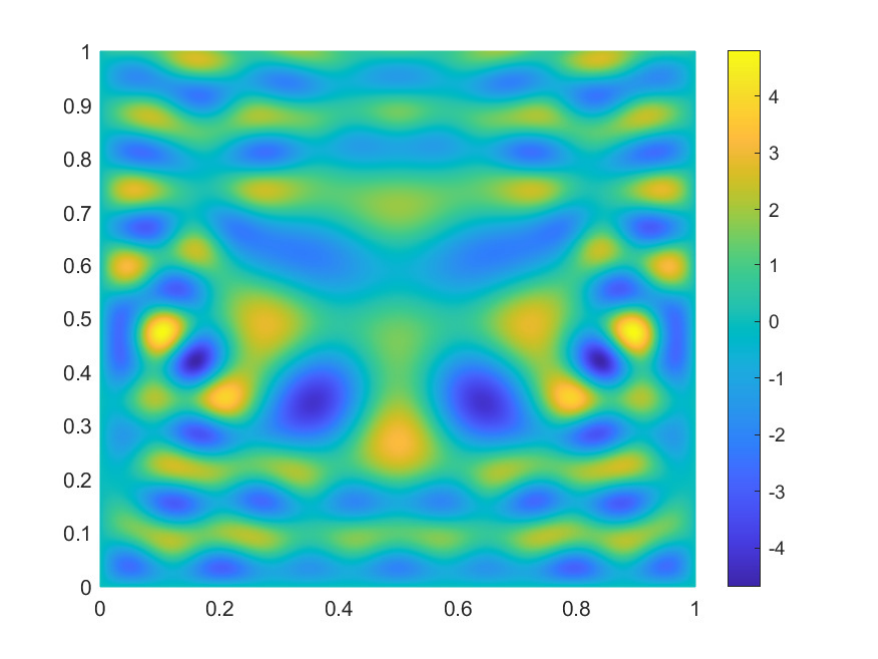}
			\end{subfigure}
			 \begin{subfigure}[b]{0.32\textwidth} \includegraphics[width=\textwidth]{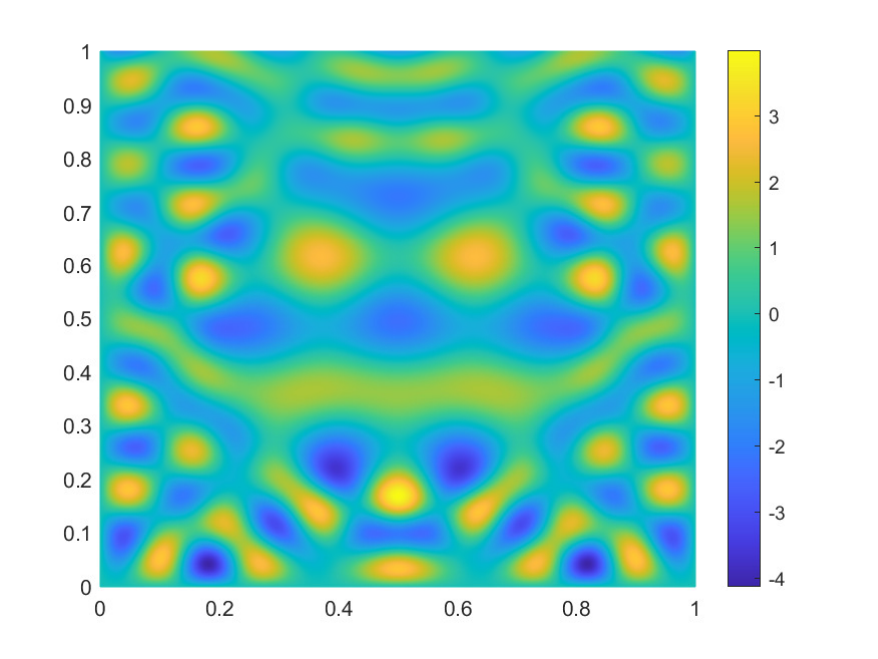}
			\end{subfigure}
			\caption{$\ka_0^2 \varepsilon_r$ (left), the real part of $u_{10}$ (middle), and the imaginary part of $u_{10}$ (right) of \cref{cavity:ex:variable} constructed by using the 2D wavelet basis constructed by taking the tensor product of the biorthogonal wavelet in \cref{ex:B2}.}
			 \label{fig:kappa2_uapprox_var}
		\end{figure}
	\end{example}
	\begin{example} \label{cavity:ex:variable:hmt}
		\normalfont
		Consider the model problem \eqref{cavity:model}, where $\mathcal{T}$ is defined in \eqref{cavity:Tu}, and $f$ and $g$ are chosen such that $u=\exp(xy)\sin(4\pi x)\sin((4\pi + \pi/2)y)$, $\kappa_0 = 4\pi$, and $\kappa^2 = \kappa_0^2 \varepsilon_r$ with $\varepsilon_r := 0.25 \chi_{\{r(x,y) \le 1/4\}} + \chi_{\Omega \cap \{1/4 < r(x,y) < 3/8\}} (-383.75 + 0.75 (8640 r(x,y) -57600 r(x,y)^2 + 189440 r(x,y)^3 -307200 r(x,y)^4 + 196608 r(x,y)^5)) + \chi_{\Omega \cap \{r(x,y) \ge 3/8\}}$ with $r(x,y):=\sqrt{(x-1/2)^2 + (y-1/2)^2}$. Note that for this example the wavenumber $\ka = \ka_0 \sqrt{\varepsilon_r}$ varies continuously in both $x$ and $y$ directions. Additionally, it is non-separable. The variable wavenumber $\kappa$ in this example is similar to \eqref{cavity:ex:variable} except with a different $\kappa_0$. See \cref{cavity:tab:variable:hmt} for numerical results.
		\begin{table}[htbp]
			\begin{center}
				\begin{tabular}{c c| c c c c | c c c c | c c}
						\hline
						$J$ & $N_J$ & \multicolumn{4}{c}{$\Phi^{2D}_{J}$ (\cref{ex:hmt})} \vline & \multicolumn{4}{c}{$\mathcal{B}_{2,J}^{2D}$ (\cref{ex:hmt})} \vline & $\tfrac{\|u-u_{J}\|_2}{\|u\|_2}$ & Order \\
						\hline
						& & $\sigma_{\max}$ & $\sigma_{\min}$ & $\tfrac{\sigma_{\max}}{\sigma_{\min}}$ & Iter  & $\sigma_{\max}$ & $\sigma_{\min}$ & $\tfrac{\sigma_{\max}}{\sigma_{\min}}$ & Iter &  &  \\
						\hline
						4 & 1056 & 2.41 & 5.17E-3 & 4.66E+2 & 107 & 2.41 & 5.17E-3 & 4.66E+2 & 107 & 5.24E-4 & \\
						5& 4160 & 2.43 & 1.26E-3 & 1.93E+3 & 215 & 3.45 & 5.17E-3 & 6.67E+2 & 175 & 3.78E-5 & 3.83 \\
						6& 16512 & 2.44 & 3.18E-4 & 7.80E+3 & 431 & 4.16 & 5.18E-3 & 8.03E+2 & 213 & 2.48E-6 & 3.95 \\
						7& 65792 & 2.44 & 7.81E-5 & 3.13E+4 & 861 & 4.26 & 5.18E-3 & 8.22E+2 & 223 & 1.60E-7 & 3.97 \\
						8& 262656 & 2.44 & 1.95E-5 & 1.25E+5 & 1718 & 4.56 & 5.18E-3 &8.81E+2 & 230 & 1.74E-8 & 3.20 \\
						\hline
				\end{tabular}
				\caption{Singular values, condition numbers, relative errors, and iteration numbers of GMRES (for relative residuals to be less than $10^{-8}$) for \cref{cavity:ex:variable:hmt}. The FEM uses the standard bases $\Phi^{2D}_{J}$ while the wavelet method employs the Riesz bases $\mathcal{B}_{J_0,J}^{2D}$. Note that $\mbox{span}(\Phi^{2D}_{J})=\mbox{span}(\mathcal{B}_{J_0,J}^{2D})$.}
				 \label{cavity:tab:variable:hmt}
			\end{center}
		\end{table}
	\end{example}
	\begin{example} \label{cavity:ex:layered}
		\normalfont
		Consider the scattering problem \eqref{model:scatter}, where $\ka_0=16\pi$, $\theta=0$, and $\varepsilon_r = 2 \chi_{\{0\le y < 1/3\}} +  1.5 \chi_{\{1/3 \le y < 2/3\}} + \chi_{\{2/3 \le y \le 1\}}$. This example was considered for instance in \cite{DLS15}, and also in \cite{BS05} with a shifted rectangular domain. Note that for this example the wavenumber $\ka = \ka_0 \sqrt{\varepsilon_r}$ is piecewise constant with horizontal straight-line interfaces, and varies only along the $y$ direction. See \cref{cavity:tab:layered} for numerical results and \cref{fig:kappa2_uapprox_layered} for relevant plots.
		\begin{table}[htbp]
			\begin{center}
				 \resizebox{\textwidth}{!}{\begin{tabular}{c c| c c c c | c c c c | c c}
					\hline
					$J$ & $N_J$ & \multicolumn{4}{c}{$\Phi^{2D}_{J}$ (\cref{ex:B2})} \vline & \multicolumn{4}{c}{$\mathcal{B}_{2,J}^{2D}$ (\cref{ex:B2})} \vline & $\|u_{J}-u_{J+1}\|_2$ & Order \\
					\hline
					& & $\sigma_{\max}$ & $\sigma_{\min}$ & $\tfrac{\sigma_{\max}}{\sigma_{\min}}$ & Iter  & $\sigma_{\max}$ & $\sigma_{\min}$ & $\tfrac{\sigma_{\max}}{\sigma_{\min}}$ & Iter &  &  \\
					\hline
					6& 4032 & 1.69 & 1.21E-4 & 1.40E+4 & 819 & 2.53E+1 & 2.95E-4 & 8.59E+4 & 749 & 1.85 & 1.56 \\
					7& 16256 & 1.54 & 7.67E-4 & 2.01E+3 & 1556 & 2.54E+1 & 5.53E-3 & 4.58E+3 & 783 & 6.21E-1 & 1.94 \\
					8& 65280 & 1.51 & 4.97E-5 & 3.04E+4 & 3086 & 2.54E+1 & 1.91E-3 & 1.33E+4 & 824 & 1.61E-1 & 1.98 \\
					9& 261632 & 1.50 & 3.89E-5 & 3.86E+4 & 6143 & 2.54E+1 & 5.96E-3 & 4.26E+3 & 833 & 4.08E-2 &  \\
					\hline
					$J$ & $N_J$ & \multicolumn{4}{c}{$\Phi^{2D}_{J}$ (\cref{ex:hmt})} \vline & \multicolumn{4}{c}{$\mathcal{B}_{2,J}^{2D}$ (\cref{ex:hmt})} \vline & $\|u_{J}-u_{J+1}\|_2$ & Order \\
					\hline
					& & $\sigma_{\max}$ & $\sigma_{\min}$ & $\tfrac{\sigma_{\max}}{\sigma_{\min}}$ & Iter  & $\sigma_{\max}$ & $\sigma_{\min}$ & $\tfrac{\sigma_{\max}}{\sigma_{\min}}$ & Iter &  &  \\
					\hline
					5 & 4160 & 5.69 & 1.51E-2 & 3.76E+2 & 630  & 1.08E+1 & 7.69E-3 & 1.40E+3 & 689 & 4.20E-2 & 4.85 \\
					6 & 16512 & 2.41 & 1.83E-3 & 1.31E+3 & 1082 & 1.08E+1 & 6.46E-3 & 1.67E+3 & 750 & 1.43E-3 & 3.13 \\
					7 & 65792 & 2.43 & 4.17E-4 & 5.84E+3 & 2184 & 1.08E+1 & 6.26E-3 & 1.72E+3 & 768 & 1.63E-4 & \\
					\hline
				\end{tabular}}
				\caption{Singular values, condition numbers, relative errors, and iteration numbers of GMRES (for relative residuals to be less than $10^{-8}$) for \cref{cavity:ex:layered}. The FEM uses the standard bases $\Phi^{2D}_{J}$ while the wavelet method employs the Riesz bases $\mathcal{B}_{J_0,J}^{2D}$. Note that $\mbox{span}(\Phi^{2D}_{J})=\mbox{span}(\mathcal{B}_{J_0,J}^{2D})$.}
				 \label{cavity:tab:layered}
			\end{center}
		\end{table}
		\begin{figure}[htbp]
			\centering
			 \begin{subfigure}[b]{0.32\textwidth} \includegraphics[width=\textwidth]{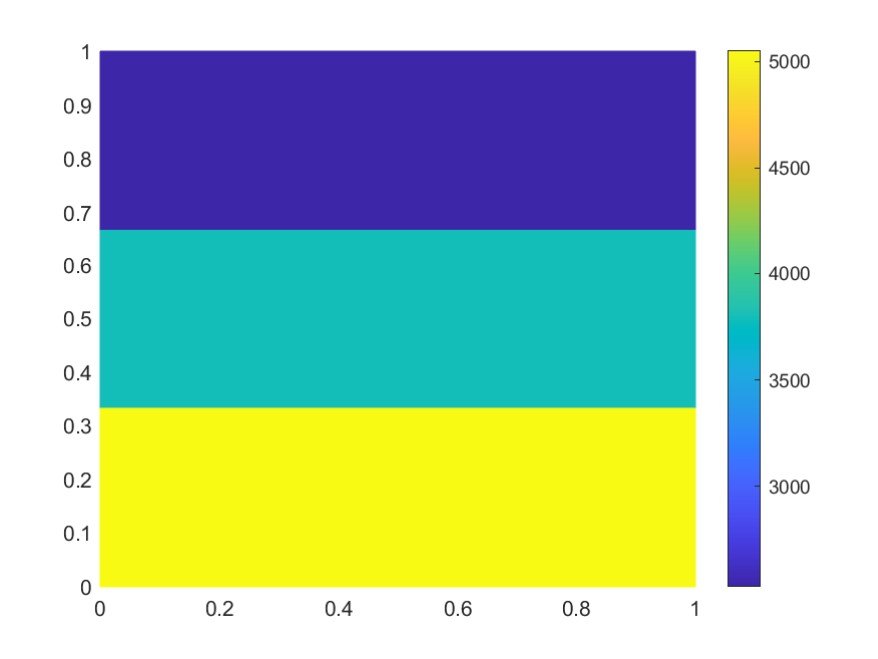}
			\end{subfigure}
			 \begin{subfigure}[b]{0.32\textwidth} \includegraphics[width=\textwidth]{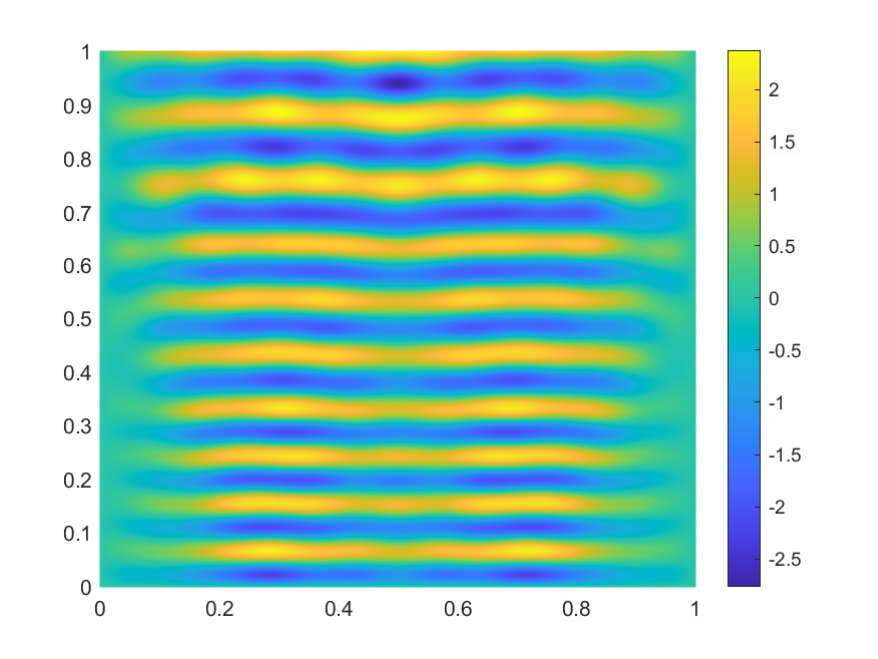}
			\end{subfigure}
			 \begin{subfigure}[b]{0.32\textwidth} \includegraphics[width=\textwidth]{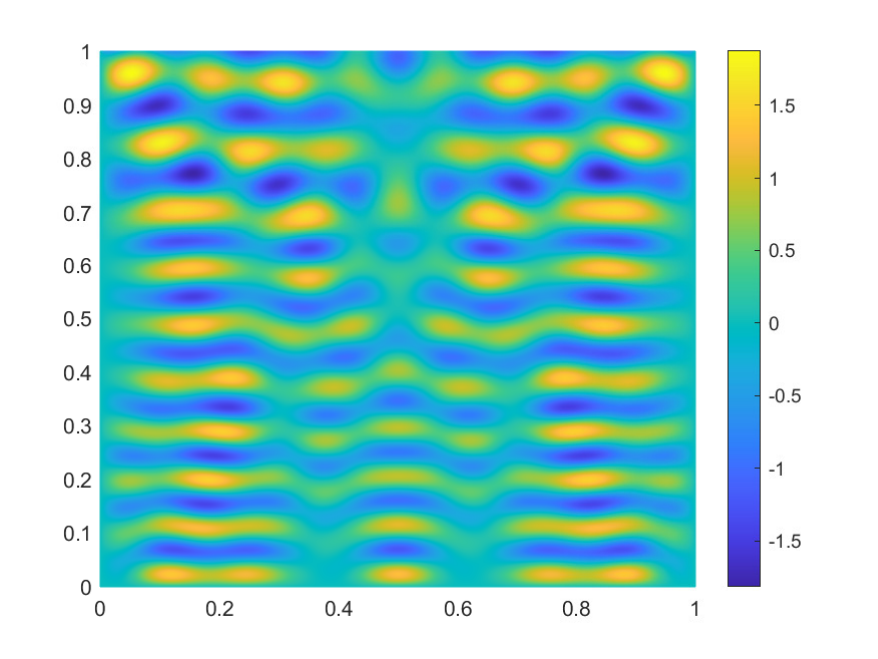}
			\end{subfigure}
			\caption{$\ka_0^2 \varepsilon_r$ (left), the real part of $u_{10}$ (middle), and the imaginary part of $u_{10}$ (right) of \cref{cavity:ex:layered} constructed by using the 2D wavelet basis constructed by taking the tensor product of the biorthogonal wavelet in \cref{ex:B2}.}
			 \label{fig:kappa2_uapprox_layered}
		\end{figure}
	\end{example}
	\begin{example} \label{cavity:ex:star}
		\normalfont
		Consider the scattering problem \eqref{model:scatter}, where $\ka_0=16\pi$, $\theta=0$, and
		\[
		\varepsilon_r = 0.1 \chi_{\{r(x,y) \le 0.2 + 0.08 \sin(\text{atan2}(y-1/2,x-1/2))\}} +
		\chi_{\Omega \cap \{r(x,y) > 0.2 + 0.08 \sin(\text{atan2}(y-1/2,x-1/2))\}}
		\]
		with $r(x,y):=\sqrt{(x-1/2)^2 + (y-1/2)^2}$. Note that for this example the wavenumber $\ka = \ka_0 \sqrt{\varepsilon_r}$ is piecewise constant with a flower-shaped interface, and is non-separable. See \cref{cavity:tab:star} for numerical results and \cref{fig:kappa2_uapprox_star} for relevant plots.
		\begin{table}[htbp]
			\begin{center}
				 \resizebox{\textwidth}{!}{\begin{tabular}{c c| c c c c | c c c c | c c}
					\hline
					$J$ & $N_J$ & \multicolumn{4}{c}{$\Phi^{2D}_{J}$ (\cref{ex:B2})} \vline & \multicolumn{4}{c}{$\mathcal{B}_{2,J}^{2D}$ (\cref{ex:B2})} \vline & $\|u_{J}-u_{J+1}\|_2$ & Order \\
					\hline
					& & $\sigma_{\max}$ & $\sigma_{\min}$ & $\tfrac{\sigma_{\max}}{\sigma_{\min}}$ & Iter  & $\sigma_{\max}$ & $\sigma_{\min}$ & $\tfrac{\sigma_{\max}}{\sigma_{\min}}$ & Iter &  &  \\
					\hline
					6& 4032 & 1.58 & 2.35E-3 & 6.74E+2 & 722 & 1.20E+1 & 9.74E-3 & 1.23E+3 & 470 & 1.42& 1.35 \\
					7& 16256 & 1.52 & 3.86E-4 & 3.94E+3 & 1438 & 1.20E+1 & 6.15E-3 & 1.95E+3 & 506 & 5.57E-1 & 2.06 \\
					8& 65280 & 1.50 & 1.05E-4 & 1.44E+4 & 2897 & 1.20E+1 & 6.34E-3 & 1.90E+3 & 528 & 1.32E-1 & 1.98 \\
					9& 261632 & 1.50 & 2.55E-5 & 5.88E+4 & 5731 & 1.20E+1 & 6.11E-3 & 1.97E+3 & 539 & 3.35E-2 &  \\
					\hline
				\end{tabular}}
				\caption{Singular values, condition numbers, relative errors, and iteration numbers of GMRES (for relative residuals to be less than $10^{-8}$) for \cref{cavity:ex:star}. The FEM uses the standard bases $\Phi^{2D}_{J}$ while the wavelet method employs the Riesz bases $\mathcal{B}_{J_0,J}^{2D}$. Note that $\mbox{span}(\Phi^{2D}_{J})=\mbox{span}(\mathcal{B}_{J_0,J}^{2D})$.}
				\label{cavity:tab:star}
			\end{center}
		\end{table}
		\begin{figure}[htbp]
			\centering
			 \begin{subfigure}[b]{0.32\textwidth} \includegraphics[width=\textwidth]{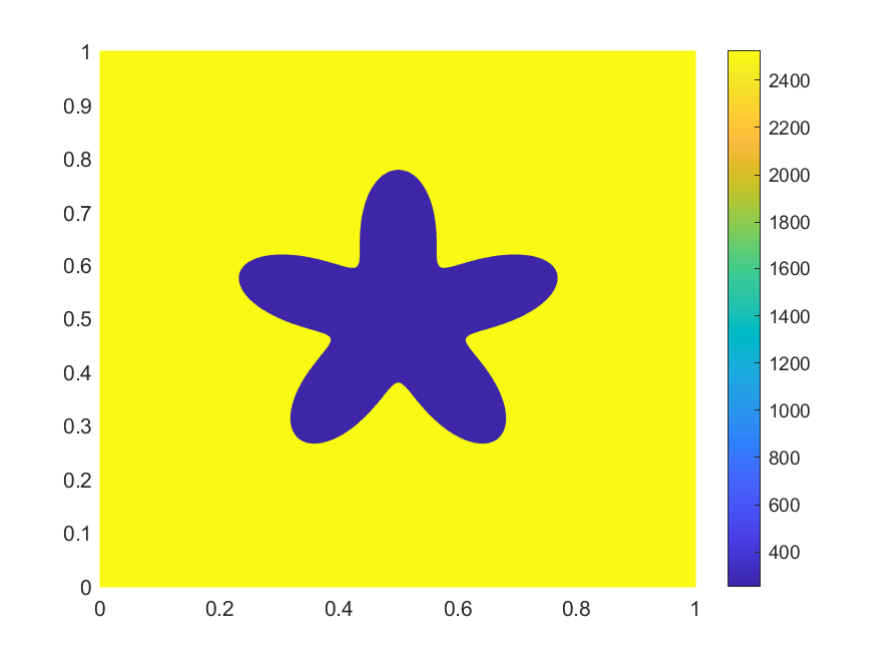}
			\end{subfigure}
			 \begin{subfigure}[b]{0.32\textwidth} \includegraphics[width=\textwidth]{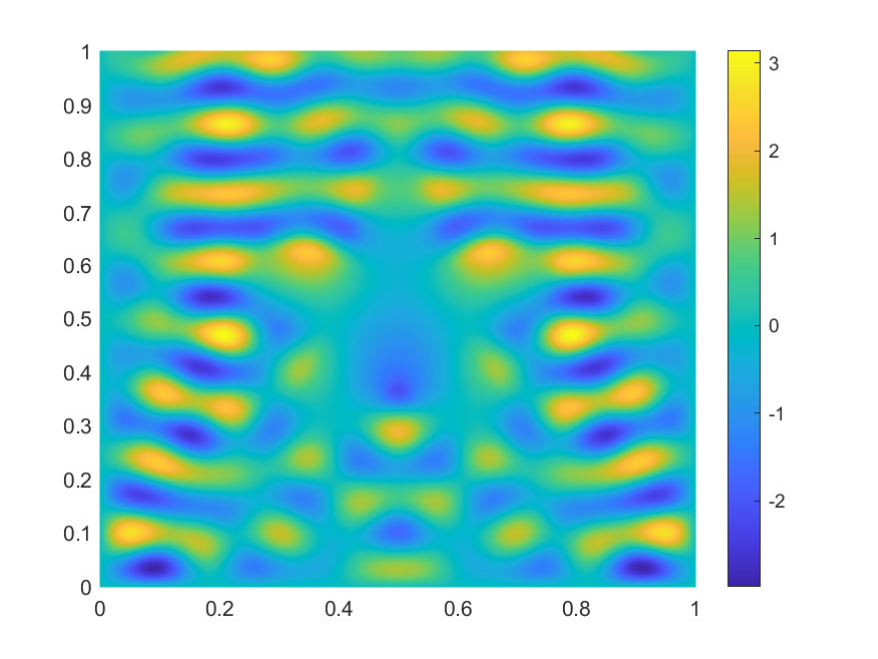}
			\end{subfigure}
			 \begin{subfigure}[b]{0.32\textwidth} \includegraphics[width=\textwidth]{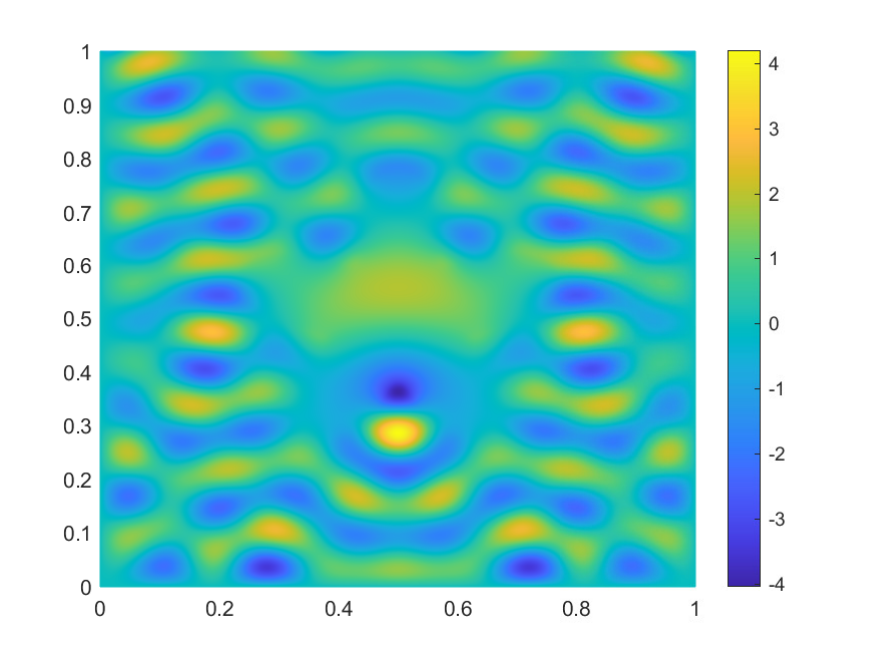}
			\end{subfigure}
			\caption{$\ka_0^2 \varepsilon_r$ (left), the real part of $u_{10}$ (middle), and the imaginary part of $u_{10}$ (right) of \cref{cavity:ex:star} constructed by using the 2D wavelet basis constructed by taking the tensor product of the biorthogonal wavelet in \cref{ex:B2}.}
			 \label{fig:kappa2_uapprox_star}
		\end{figure}
	\end{example}
	The above examples show that our wavelet Galerkin method can handle variable wavenumbers, which are continuous and piecewise smooth. However, for discontinuous wavenumbers such as \cref{cavity:ex:layered}, a high-order scheme experiences a reduced convergence rate at a higher scale level, which is caused by corner singularities (where the nonlocal and Dirichlet boundary conditions meet) and singularities near the discontinuities of the wavenumber (i.e., the interface). Being a Riesz basis of the Sobolev space $H^1(\Omega)$, wavelets have the potential to recover optimal convergence rates (with respect to the approximation order of our basis) by effectively capturing and dealing with these two types of singularities. We can apply the ideas in \cite{HM24} for elliptic interface problems to our high-order wavelet Galerkin scheme to capture and treat the singularities near discontinuities of the wavenumber, and add more basis functions at higher scale levels in the approximated solution near the corners, while preserving the uniformly bounded condition numbers. We shall address this issue in our future work.
	
	We conclude this section by comparing our paper with \cite{BS05} for the special case that $\gep_r=\gep_r(y)$ (and hence $\kappa^2=\kappa^2_0\gep_r(y)$) only depends on the variable $y$. Using the fast Fourier transform in the horizontal and a Gaussian elimination in the vertical direction, the authors in \cite{BS05} proposed a fast solver and a preconditioning strategy to solve the problem by essentially only dealing with the 1D problem (primarily due to the special structure and independence along the horizontal variable of $\gep_r$ and $\kappa$ to untangle the nonlocal boundary condition on $\Gamma$). In their numerical experiments, they use a standard second-order finite difference method (FDM), although they mention that their algorithm is applicable to other discretizations as well.

	Even though we consider the same model problem in \eqref{cavity:model}, our paper has a different focus than theirs. Our contribution is in the discretization aspect of the problem with an emphasis on a high-order method that can handle variable wavenumbers and achieve uniformly bounded condition numbers. We construct a 2D wavelet basis and use it in the Galerkin framework. Since it is a well-behaving basis in the Sobolev space $\mathcal{H}$ (in the sense that it is a Riesz basis thus \eqref{f:H1:2D}-\eqref{rz:H1:2D} hold), one direct consequence is that the linear system obtained from discretizing the problem has a uniformly bounded condition number independent of the scale level or the number of terms used in the approximated solution. At this moment, we do not have a fast solver or a further preconditioning strategy for the linear system obtained by our wavelet Galerkin method. So, when it comes to solving a linear system obtained from a discretization, the algorithm in \cite{BS05}  for the case that $\gep_r=\gep_r(y)$ has a much better computational efficiency.
	
	Inspired by the interesting idea in \cite{BS05} of using the fast Fourier transform (FFT) in the horizontal direction and a Gaussian elimination in the vertical direction to untangle the nonlocal boundary condition on $\Gamma$, we could exploit the separability of the original problem for the special case that $\gep_r=\gep_r(y)$. Similar to the approach introduced in \cite{BS05}, we only need to solve 1D problems, which can be done by using a method that is parallelizable and works for arbitrarily high wavenumbers such as the Dirac Assisted Tree (DAT) method \cite{HM21a}. An advantage of using DAT is that we can use any discretization method such as a 1D finite difference method with arbitrarily high accuracy or a 1D wavelet Galerkin method. Note that our wavelet method also has the separability by using tensor product of 1D wavelets. Hence, the technique in \cite{BS05} taking advantages of the separability of the original problem may be still applicable by employing special derivative-orthogonal wavelets in \cite{hm19} along the vertical direction. Combining this separation of variables method as outlined in \cite{BS05} with DAT or a tensor-product wavelet method is expected to yield a discretization method with desirable computational efficiency and small uniformly bounded condition numbers. These issues will be discussed in our future work.
	
	\section{Conclusion}
	In this paper, we presented a high-order wavelet Galerkin method for solving the model problem \eqref{cavity:model} with variable wavenumbers. We provided a self-contained proof showing that the tensor product of two 1D Riesz wavelets in the appropriate Sobolev space forms a 2D Riesz wavelet in the appropriate Sobolev space. Consequently, the wavelet coefficient matrix obtained from the discretization has a uniformly bounded condition number. In our future work, we shall improve our method to deal with corner singularities and the singularity along the discontinuity of the wavenumber. Additionally, based on the observation in \cite{BS05}, we can take the Fourier transform of the model problem and solve the resulting 1D problems using DAT combined with 1D high-order FDM or 1D derivative-orthogonal wavelet Galerkin method.
	
	\section{Appendix: Proofs of \cref{thm:H1,thm:H1:2D}}
	\label{sec:appendix}
	To prove our theoretical results in \cref{thm:H1,thm:H1:2D}, we need the following result.
	
	\begin{theorem}\label{thm:bessel}
		Let $\eta$ be a function supported inside $[0,1]^d$ and $\eta\in H^\gep(\R^d)$ for some $\gep>0$. Then there exists a positive constant $C$ such that
		\be \label{bessel:eta}
		\sum_{j\in \Z} |\la f,2^{dj/2}\eta(2^j\cdot)\ra|^2\le C \|f\|^2_{\LpId{2}},\qquad \forall\; f\in \LpId{2}.
		\ee
	\end{theorem}
	
	\begin{proof}
		Define $g:=\eta-\eta(\cdot+e_1)$, where $e_1:=(1,0,\ldots,0)^\tp\in \R^d$.
		Note that $\eta(\cdot+e_1)$ is supported outside $(0,1)^d$.
		Because $\eta\in H^\gep(\R^d)$, we have $g\in H^\gep(\R^d)$ and $g$ is a compactly supported function satisfying $\int_{\R^d} g(x) dx=0$. By \cite[Theorem~2.3]{han03}, there exists a positive constant $C$ such that
		\be \label{bessel:g}
		\sum_{j\in \Z} \sum_{k\in \Z} |\la f, g_{j;k}\ra|^2 \le C\|f\|_{\dLp{2}}^2,\qquad \forall\; f\in \dLp{2},
		\ee
		where $g_{j;k}:=2^{dj/2}g(2^j\cdot-k)$.
		Taking $f\in \LpId{2}$ in \eqref{bessel:g} and noting that $\la f, g_{j;0}\ra=\la f, \eta_{j;0}\ra$, we trivially deduce from \eqref{bessel:g} that \eqref{bessel:eta} holds.
		
		For the convenience of the reader, we provide a self-contained proof here. Note that $\wh{g}(\xi)=(1-e^{i\xi\cdot e_1})\wh{\eta}(\xi)$ and hence
		\[
		|\wh{g}(\xi)| \le \min(\|\xi\|, 2) |\wh{\eta}(\xi)|\le \theta(\xi) |\wh{\eta}(\xi)| (1+\|\xi\|^2)^{\gep/2}
		\quad \mbox{with}\quad
		 \theta(\xi):=\min(\|\xi\|,2)(1+\|\xi\|^2)^{-\gep/2}.
		\]
		For $f\in \LpId{2}$, note that $\la f, \eta(2^j\cdot+e_1)\ra=0$ because $\eta(2^j\cdot+e_1)$ is supported outside $(0,1)^d$ for all $j\in \Z$.
		For $f\in \LpId{2}$,
		by the Plancherel Theorem, we have
		\[
		\la f, 2^{j/2}\eta(2^j\cdot)\ra=
		\la f, g_{j;0}\ra=(2\pi)^{-d}\la \wh{f},\wh{g_{j;0}}\ra.
		\]
		Noting that $\wh{g_{j;0}}(\xi)=2^{-dj/2}\wh{g}(2^{-j}\xi)$ and using the Cauchy-Schwarz inequality,
		we have
		\begin{align*}
			& |\la \wh{f},\wh{g_{j;0}}\ra|^2
			\le
			2^{-dj} \left(\int_{\R^d} |\wh{f}(\xi) \theta(2^{-j}\xi) \wh{\eta}(2^{-j}\xi)| (1+\|2^{-j}\xi\|^2)^{\gep/2} d\xi\right)^2 \\
			& \quad \le \left(2^{-dj}  \int_{\R^d}
			|\wh{\eta}(2^{-j}\xi)|^2 (1+\|2^{-j}\xi\|^2)^\gep d\xi \right) \left(\int_{\R^d}
			|\wh{f}(\xi)|^2 \theta^2(2^{-j}\xi) d\xi\right)=C_1 \int_{\R^d}
			|\wh{f}(\xi)|^2 \theta^2(2^{-j}\xi) d\xi,
		\end{align*}
		where $C_1:=\int_{\R^d} |\wh{\eta}(\xi)|^2 (1+\|\xi\|^2)^\gep d\xi<\infty$ because $\eta\in H^\gep(\R^d)$.
		Therefore, for $f\in \LpId{2}$,
		\[
		\sum_{j\in \Z} |\la f,2^{j/2}\eta(2^j\cdot)\ra|^2
		\le (2\pi)^{-2d} C_1 \int_{\R^d} |\wh{f}(\xi)|^2 \sum_{j\in \Z} \theta^2(2^{-j}\xi)d\xi\le C_1 \|f\|_{\LpId{2}}^2
		\Big\| \sum_{j\in \Z} \theta^2(2^j\cdot)\Big\|_{\dLp{\infty}}.
		\]
		For any $\xi\in \R^d \bs \{0\}$, there is a unique integer $J\in \Z$ such that $2^{J}<\|\xi\|\le 2^{J+1}$. Hence,
		$2^{J+j}<\|2^{j}\xi\|\le 2^{J+j+1}$ for all $j\in \Z$ and
		\[
		\sum_{j\in \Z} \theta^2(2^j \xi)
		=\sum_{j\le -J}  \theta^2(2^{j}\xi)+\sum_{j>-J}  \theta^2(2^{j}\xi)
		\le \sum_{j\le -J} 2^{2(j+J+1)}
		+4\sum_{j>-J}  2^{-2(j+J)\gep}
		 =5+\frac{1}{3}+\frac{2^{2-2\gep}}{1-2^{-2\gep}}.
		\]
		Consequently, we conclude that
		\eqref{bessel:eta} holds with $C:=C_1 (6+\frac{2^{2-2\gep}}{1-2^{-2\gep}})<\infty$ by $\gep>0$.
	\end{proof}
	
	\bp[Proof of \cref{thm:H1}] Recall that $\cI:=(0,1)$. Let us first point out a few key ingredients that we will use in our proof. First, we note that all the functions $f$ in $H^{1,x}(\cI)$ and $H^{1,y}(\cI)$ satisfy $f(0)=0$. Our proof here only uses $f(0)=0$ and $f(1)$ could be arbitrary.
	Hence, it suffices for us to prove the claim for $H^{1,x}(\cI)$ and the same proof can be applied to $H^{1,y}(\cI)$. Second, because $\phi$ is a compactly supported refinable vector function in $H^1(\R)$ with a finitely supported mask $a\in \lrs{0}{r}{r}$. By \cite[Corollary~5.8.2]{hanbook} (c.f. \cite[Theorem~2.2]{han03}), we must have $\sm(\phi)>1$ and hence $\sr(a)\ge 2$ must hold. Hence, there exists $\gep>0$ such that all $[\phi]', [\phi^L]', [\phi^R]', [\psi]', [\psi^L]', [\psi^R]'$ belong to $H^{\gep}(\R)$ (see \cite{HM21a}).
	Because $\sr(a)\ge 2$, the dual wavelet $\tilde{\psi}$ must have at least order two vanishing moments, i.e., $\vmo(\tilde{\psi})\ge 2$. We define
	\be \label{mtpsi}
	\mathring{\tilde{\psi}}(x):=
	\int_{x}^\infty \tilde{\psi}(t) dt,\qquad x\in \R.
	\ee
	Because $\tilde{\psi}$ is compactly supported and $\vmo(\tilde{\psi})\ge 2$, we conclude that the new function $\mathring{\tilde{\psi}}$ must be a compactly supported function and $\vmo(\mathring{\tilde{\psi}})\ge 1$.
	Third, although all our constructed wavelets have vanishing moments, except the necessary condition $\vmo(\tilde{\psi})\ge 2$, our proof does not assume that $\psi^L, \psi^R, \psi, \tilde{\psi}^L, \tilde{\psi}^R$ have any order of vanishing moments at all.
	
	Let $\{c_\alpha\}_{\alpha\in \Phi_{J_0}}\cup \{c_{\beta_j}\}_{j\ge J_0, \beta_j\in \Psi_j}$ be a finitely supported sequence. We define a function $f$ as in
	\eqref{f:H1}. Since the summation is finite, the function $f$ is well defined and $f\in H^{1,x}(\cI)$.
	Since $(\tilde{\mathcal{B}},\mathcal{B})$ is a locally supported biorthogonal wavelet in $\LpI{2}$ and $H^{1,x}(\cI)\subseteq \LpI{2}$,
	we have
	\be \label{fL2}
	f:=\sum_{\alpha\in \Phi_{J_0}} c_\alpha 2^{-J_0} \alpha+\sum_{j=J_0}^\infty \sum_{\beta_j\in \Psi_j} c_{\beta_j} 2^{-j} \beta_j=
	\sum_{\alpha\in \Phi_{J_0}} \la f, \tilde{\alpha}\ra \alpha+\sum_{j=J_0}^\infty \sum_{\beta_j\in \Psi_j} \la f, \tilde{\beta_j}\ra \beta_j,
	\ee
	because we deduce from the biorthogonality of $(\tilde{\mathcal{B}},\mathcal{B})$ that
	\[
	\la f, \tilde{\alpha}\ra=c_\alpha 2^{-J_0},\qquad \la f, \tilde{\beta}_j\ra=2^{-j} c_{\beta_j},\qquad j\ge J_0.
	\]
	Because $(\tilde{\mathcal{B}},\mathcal{B})$ is a locally supported biorthogonal wavelet in $\LpI{2}$, there must exist positive constants $C_3$ and $C_4$, independent of $f$ and $\{c_\alpha\}_{\alpha\in \Phi_{J_0}}\cup \{c_{\beta_j}\}_{j\ge J_0, \beta_j\in \Psi_j}$, such that
	\be \label{rz:L2}
	C_3 \Big(\sum_{\alpha \in \Phi_{J_0}} 2^{-2J_0}|c_\alpha|^2+
	\sum_{j=J_0}^\infty\sum_{\beta_j \in \Psi_j} 2^{-2j}|c_{\beta_j}|^2\Big)
	\le \|f\|^2_{\LpI{2}}
	\le C_4 \Big(\sum_{\alpha \in \Phi_{J_0}} 2^{-2J_0}|c_\alpha|^2+\sum_{j=J_0}^\infty\sum_{\beta_j\in \Psi_j} 2^{-2j}|c_{\beta_j}|^2\Big).
	\ee
	We now prove \eqref{rz:H1}. From the definition $\psi_{j;k}:=2^{j/2}\psi(2^j\cdot-k)$, it is very important to notice that
	\[
	[\psi_{j;k}]':=
	[2^{j/2}\psi(2^j\cdot-k)]'=
	2^j 2^{j/2} \psi'(2^j\cdot-k)=2^j [\psi']_{j;k}.
	\]
	Define
	\[
	\mathcal{\breve{B}}:=\{ g' \setsp g\in \mathcal{B}\}=\breve{\Phi}_{J_0}\cup \cup_{j=J_0}^\infty \breve{\Psi}_j,
	\]
	where
	\be \label{phipsi:breve}
	\begin{aligned}
		 &\breve{\Phi}_{J_0}:=\{[\phi^{L}]'_{J_0;0}\}\cup
		\{[\phi]'_{J_0;k} \setsp n_{l,\phi} \le k\le 2^{J_0}-n_{h,\phi}\}\cup \{ [\phi^{R}]'_{J_0;2^{J_0}-1}\},\\
		 &\breve{\Psi}_{j}=\{[\psi^{L}]'_{j;0}\}\cup
		\{[\psi]'_{j;k} \setsp n_{l,\psi} \le k\le 2^{j}-n_{h,\psi}\}\cup \{ [\psi^{R}]'_{j;2^{j}-1}\}.
	\end{aligned}
	\ee
	That is, we obtain
	$\breve{\mathcal{B}}$ by replacing all the generators $\phi^L, \phi, \phi^R, \psi^R, \psi, \psi^R$ in $\mathcal{B}$ by new generators
	$[\phi^L]', [\phi]', [\phi^R]', [\psi^R]', [\psi]', [\psi^R]'$, respectively.
	Note that all the elements in $\breve{\mathcal{B}}$ belongs to $H^\gep(\R)$.
	From \eqref{f:H1}, noting that $\{c_\alpha\}_{\alpha\in \Phi_{J_0}}\cup \{c_{\beta_j}\}_{j\ge J_0, \beta_j\in \Psi_j}$ is finitely supported,
	we have
	\[
	f'=\sum_{\alpha \in \breve{\Phi}_{J_0}} c_\alpha \alpha+\sum_{j=J_0}^\infty
	\sum_{\beta_j\in \breve{\Psi}_j} c_{\beta_j} \beta_j.
	\]
	Because every element in $\breve{\mathcal{B}}$ trivially has at least one vanishing moment due to derivatives, by \cref{thm:bessel} with
	\cite[Theorem~2.6]{HM21a} or \cite[Theorem~2.3]{han03}, the system $\breve{\mathcal{B}}$ must be a Bessel sequence in $\LpI{2}$ and thus there exists a positive constant $C_5$,  independent of $f'$ and $\{c_\alpha\}_{\alpha\in \Phi_{J_0}}\cup \{c_{\beta_j}\}_{j\ge J_0, \beta_j\in \Psi_j}$, such that (e.g., see \cite[(4.2.5) in Proposition~4.2.1]{hanbook})
	\[
	\|f'\|^2_{\LpI{2}}
	=\Big\|
	\sum_{\alpha \in \breve{\Phi}_{J_0}} c_\alpha \alpha+\sum_{j=J_0}^\infty
	\sum_{\beta_j\in \breve{\Psi}_j} c_{\beta_j} \beta_j\Big\|^2_{\LpI{2}}
	\le C_5 \Big(\sum_{\alpha \in \Phi_{J_0}} |c_\alpha|^2+\sum_{j=J_0}^\infty \sum_{\beta_j\in \Psi_j} |c_{\beta_j}|^2\Big).
	\]
	Therefore, we conclude from the above inequality and \eqref{rz:L2} that
	the upper bound in \eqref{rz:H1} holds with $C_2:=C_4+C_5<\infty$, where we also used $J_0\in \N\cup\{0\}$ so that $0\le 2^{-2j}\le 1$ for all $j\ge J_0$.
	
	We now prove the lower bound of \eqref{rz:H1}.
	Define
	\be \label{mL2}
	[\eta]^\circ(x):=\int_x^1 \eta(t) dt,\qquad
	x\in [0,1], \eta\in \LpI{2}.
	\ee
	Because $f(0)=0$ and $[\tilde{\beta}_j]^\circ(1):=\int_1^1 \tilde{\beta}_j(t)dt=0$, we have
	\[
	\la f, \tilde{\beta}_j\ra=\int_0^1 f(t) \ol{\tilde{\beta}_j(t)} dt=-\int_0^1 f(t) d \ol{[\tilde{\beta}_j]^\circ(t)}=
	 -f(t)\ol{[\tilde{\beta}_j]^\circ(t)}|_{t=0}^{t=1}
	+\int_{0}^1 f'(t) \ol{[\tilde{\beta}_j]^\circ(t)}dt=\la f', [\tilde{\beta}_j]^\circ\ra.
	\]
	Similarly, we have $\la f, \tilde{\alpha}\ra=\la f', [\tilde{\alpha}]^\circ\ra$. It is important to notice that all these identities hold true for any general function $f\in H^{1,x}(\cI)$.
	Therefore,
	\be \label{cbetaj}
	c_\alpha=2^{J_0}\la f, \tilde{\alpha}\ra
	=\la f',2^{J_0} [\tilde{\alpha}]^\circ \ra
	\quad \mbox{and}\quad
	c_{\beta_j}=2^{j}\la f, \tilde{\beta}_j\ra=
	\la f', 2^{j} [\tilde{\beta}_j]^\circ\ra.
	\ee
	It is also very important to notice that if $\tilde{\psi}_{j,k}$ is supported inside $[0,1]$, then
	\[
	2^j [\tilde{\psi}_{j;k}]^\circ(x):=2^j \int_{x}^1 2^{j/2}\tilde{\psi}(2^j t-k)dt=\mathring{\tilde{\psi}}_{j;k}(x),
	\]
	where the function $\mathring{\tilde{\psi}}$ is defined in \eqref{mtpsi}. Define
	\[
	\breve{\tilde{\mathcal{B}}}:=
	2^{J_0} [\tilde{\Phi}_{J_0}]^\circ
	\cup \cup_{j=J_0}^\infty 2^j [\tilde{\Psi}_j]^\circ=
	\breve{\tilde{\Phi}}_{J_0}\cup \cup_{j=J_0}^\infty \breve{\tilde{\Psi}}_j,
	\]
	where
	\be \label{dphipsi:breve}
	\begin{aligned}
		&\breve{\tilde{\Phi}}_{J_0}:=\{
		 \mathring{\tilde{\phi}}^{L}_{J_0;0}\}\cup
		\{\mathring{\tilde{\phi}}_{J_0;k} \setsp n_{l,\tilde{\phi}} \le k\le 2^{J_0}-n_{h,\tilde{\phi}}\}\cup \{ \mathring{\tilde{\phi}}^{R}
		_{J_0;2^{J_0}-1}\},\\
		&\breve{\Psi}_{j}=
		\{\mathring{\tilde{\psi}}^{L}
		_{j;0}\}\cup
		\{\mathring{\tilde{\psi}}_{j;k} \setsp n_{l,\tilde{\psi}} \le k\le 2^{j}-n_{h,\tilde{\psi}}\}\cup \{ \mathring{\tilde{\psi}}^{R}_{j;2^{j}-1}\}.
	\end{aligned}
	\ee
	That is, we obtain
	$\breve{\tilde{\mathcal{B}}}$ by replacing the generators $\tilde{\phi}^L, \tilde{\phi}, \tilde{\phi}^R, \tilde{\psi}^R, \tilde{\psi}, \tilde{\psi}^R$ in $\tilde{\mathcal{B}}$ by new generators
	$\mathring{\tilde{\phi}}^L$, $\mathring{\tilde{\phi}}$, $\mathring{\tilde{\phi}}^R$, $\mathring{\tilde{\psi}}^R$, $\mathring{\tilde{\psi}}$, $\mathring{\tilde{\psi}}^R$, respectively. Note that all elements in $\breve{\tilde{\mathcal{B}}}$ belong to $H^{1/2}(\R)$.
	As we discussed before, it is very important to notice that $\vmo(\tilde{\psi})\ge 2$ and consequently, the new function $\mathring{\tilde{\psi}}$ has compact support and $\vmo(\mathring{\tilde{\psi}})\ge 1$.
		Hence, combining \cref{thm:bessel} with \cite[Theorem~2.3]{han03} for $\mathring{\tilde{\psi}}$, we conclude that the system $\breve{\tilde{\mathcal{B}}}$ must be a Bessel sequence in $\LpI{2}$ and therefore, there exists a positive constant $C_6$, independent of $f'$ and $\{c_\alpha\}_{\alpha\in \Phi_{J_0}}\cup \{c_{\beta_j}\}_{j\ge J_0, \beta_j\in \Psi_j}$, such that
		\be \label{C6}
		\sum_{\alpha \in \Phi_{J_0}} |c_\alpha|^2+\sum_{j=J_0}^\infty\sum_{\beta\in \Psi_j} |c_{\beta_j}|^2
		=\sum_{\tilde{\alpha} \in \breve{\tilde{\Phi}}_{J_0}} |\la f', \tilde{\alpha}\ra|^2
		 +\sum_{j=J_0}^\infty\sum_{\tilde{\beta}_j\in \breve{\tilde{\Psi}}_j} |\la f', \tilde{\beta}_j\ra|^2\le C_6\|f'\|^2_{\LpI{2}},
		\ee
		where we used the identities in \eqref{cbetaj}.
		Therefore, noting that $\|f'\|^2_{\LpI{2}}\le \|f\|^2_{\LpI{2}}+\|f'\|^2_{\LpI{2}}
		=\|f\|^2_{H^1(\cI)}$ trivially holds,
		we conclude from the above inequality that
		the lower bound in \eqref{rz:H1} holds with $C_1:= \frac{1}{C_6} <\infty$.
		Therefore, we prove that \eqref{rz:H1} holds with $f$ defined in \eqref{f:H1} for all finitely supported sequences $\{c_\alpha\}_{\alpha\in \Phi_{J_0}}\cup \{c_{\beta_j}\}_{j\ge J_0, \beta_j\in \Psi_j}$.
		Now by the standard density argument,
		for any square summable sequence $\{c_\alpha\}_{\alpha\in \Phi_{J_0}}\cup \{c_{\beta_j}\}_{j\ge J_0, \beta_j\in \Psi_j}$ satisfying
		\be \label{seq:l2}
		\sum_{\alpha\in \Phi_{J_0}}
		|c_\alpha|^2+\sum_{j=J_0}^\infty \sum_{c_{\beta_j}\in \Psi_j}
		|c_{\beta_j}|^2<\infty,
		\ee
		we conclude that
		\eqref{rz:H1} holds and the series in \eqref{f:H1} absolutely converges in $H^{1,x}(\cI)$.
		
		Because $(\tilde{\mathcal{B}},\mathcal{B})$ is a locally supported biorthogonal wavelet in $\LpI{2}$ and $H^{1,x}(\cI)\subseteq \LpI{2}$,
		for any $f\in H^{1,x}(\cI)$, we have
		\be \label{general:f}
		f=\sum_{\alpha\in \Phi_{J_0}} c_\alpha 2^{-J_0} \alpha+\sum_{j=J_0}^\infty \sum_{\beta_j\in \Psi_j} c_{\beta_j} 2^{-j} \beta_j
		\quad \mbox{with}\quad c_\alpha:=2^{J_0}\la f, \tilde{\alpha}\ra,\quad c_{\beta_j}:=2^{j}\la f, \tilde{\beta}_j\ra,\quad j\ge J_0
		\ee
		with the series converging in $\LpI{2}$.
		Note that we already proved \eqref{cbetaj} for any $f\in H^{1,x}(\cI)$ and hence \eqref{C6} must hold true. In particular, we conclude from \eqref{C6} that the coefficients of $f\in H^{1,x}(\cI)$ in \eqref{general:f} must satisfy \eqref{seq:l2}. Consequently, by \eqref{rz:H1}, the series in \eqref{general:f} must converge absolutely in $H^{1,x}(\cI)$.
		This proves that $\mathcal{B}_{H^{1,x}}$ is a Riesz basis of $H^{1,x}(\cI)$.
	\ep
	
	\bp[Proof of \cref{thm:H1:2D}]
Recall that $\cI:=(0,1)$ and $\Omega:=\cI^2=(0,1)^2$.
	Let $\{c_\alpha\}_{\alpha\in \Phi^{2D}_{J_0}}\cup \{c_{\beta_j}\}_{j\ge J_0, \beta_j\in \Psi_j^{2D}}$ be a finitely supported sequence. We define a function $f$ as in
	\eqref{f:H1:2D}. Since the summation is finite, the function $f$ is well defined and $f\in \mathcal{H}(\Omega) \subseteq \LpO{2}$. Define $\tilde{\mathcal{B}}^{2D}$ by adding $\sim$ to all elements in $\mathcal{B}^{2D}$. Note that $(\tilde{\mathcal{B}}^{2D},\mathcal{B}^{2D})$ is a biorthogonal wavelet in $\LpO{2}$, because it is formed by taking the tensor product of two biorthogonal wavelets in $\LpI{2}$. Hence,
	\be \label{fL2:2D}
	f:=\sum_{\alpha\in \Phi_{J_0}^{2D}} c_\alpha 2^{-J_0} \alpha+\sum_{j=J_0}^\infty \sum_{\beta_j\in \Psi_j^{2D}} c_{\beta_j} 2^{-j} \beta_j=
	\sum_{\alpha\in \Phi_{J_0}^{2D}} \la f, \tilde{\alpha}\ra \alpha+\sum_{j=J_0}^\infty \sum_{\beta_j\in \Psi_j^{2D}} \la f, \tilde{\beta_j}\ra \beta_j,
	\ee
	because we deduce from the biorthogonality of $(\tilde{\mathcal{B}}^{2D},\mathcal{B}^{2D})$ that
	\[
	\la f, \tilde{\alpha}\ra=c_\alpha 2^{-J_0},\qquad \la f, \tilde{\beta}_j\ra=2^{-j} c_{\beta_j},\qquad j\ge J_0.
	\]
	Because $(\tilde{\mathcal{B}}^{2D},\mathcal{B}^{2D})$ is a biorthogonal wavelet in $\LpO{2}$, there must exist positive constants $C_3$ and $C_4$, independent of $f$ and $\{c_\alpha\}_{\alpha\in \Phi_{J_0}^{2D}}\cup \{c_{\beta_j}\}_{j\ge J_0, \beta_j\in \Psi_j^{2D}}$, such that
	\be \label{rz:L2:2D}
	C_3 \Big(\sum_{\alpha \in \Phi_{J_0}^{2D}} 2^{-2J_0}|c_\alpha|^2+
	\sum_{j=J_0}^\infty\sum_{\beta_j \in \Psi_j^{2D}} 2^{-2j}|c_{\beta_j}|^2\Big)
	\le \|f\|^2_{\LpO{2}}
	\le C_4 \Big(\sum_{\alpha \in \Phi_{J_0}^{2D}} 2^{-2J_0}|c_\alpha|^2+\sum_{j=J_0}^\infty\sum_{\beta_j\in \Psi_j^{2D}} 2^{-2j}|c_{\beta_j}|^2\Big).
	\ee
	To prove \eqref{rz:H1:2D}, it is enough to consider $\tfrac{\partial}{\partial y}f$, since the argument used for $\tfrac{\partial}{\partial x}f$ is identical. From \eqref{f:H1:2D}, noting that $\{c_\alpha\}_{\alpha\in \Phi_{J_0}^{2D}}\cup \{c_{\beta_j}\}_{j\ge J_0, \beta_j\in \Psi_j^{2D}}$ is finitely supported,
	we have
	\[
	\tfrac{\partial}{\partial y}f=\sum_{\alpha \in \breve{\Phi}_{J_0}^{2D,y}} c_\alpha \alpha+\sum_{j=J_0}^\infty
	\sum_{\beta_j\in \breve{\Psi}_j^{2D,y}} c_{\beta_j} \beta_j,
	\]
	where
	\[
	\breve{\Phi}^{2D,y}_{J_0}:=\{ 2^{-J_0} \Phi^{x}_{J_0} \otimes \breve{\Phi}^{y}_{J_0}\} \quad \text{and} \quad
	\breve{\Psi}^{2D,y}_{j}:= \{ 2^{-j} \Phi^{x}_j \otimes \breve{\Psi}^{y}_j,  2^{-j} \Psi^{x}_{j}  \otimes \breve{\Phi}^{y}_{j}, 2^{-j} \Psi^{x}_j \otimes \breve{\Psi}^{y}_j\},
	\]
	$\breve{\Phi}^{y}_{J_0}$ and $\breve{\Psi}^{y}_j$ for $j \ge J_0$ are defined as in \eqref{phipsi:breve} with $\phi^{R}$ and $\psi^{R}$ replaced by $\phi^{R,y}$ and $\psi^{R,y}$ respectively. Every element in $\breve{\Phi}^{y}_{J_0} \cup \cup_{j=J_0}^{\infty} \breve{\Psi}^{y}_j$ has at least one vanishing moment and belongs to $H^\gep(\R)$ for some $\gep>0$. Hence, every element in $\breve{\Phi}_{J_0}^{2D,y} \cup \cup_{j=J_0}^{\infty}  \breve{\Psi}^{2D,y}_{j}$ has at least one vanishing moment and belongs to $H^\gep(\R^2)$ for some $\gep>0$. By
	\cref{thm:bessel} and
	\cite[Theorem~2.3]{han03}, the system $\breve{\Phi}_{J_0}^{2D,y} \cup \cup_{j=J_0}^{\infty}  \breve{\Psi}^{2D,y}_{j}$ must be a Bessel sequence in $\LpO{2}$. That is, there exists a positive constant $C_5$, independent of $\tfrac{\partial}{\partial y} f$ and $\{c_\alpha\}_{\alpha\in \Phi_{J_0}^{2D}}\cup \{c_{\beta_j}\}_{j\ge J_0, \beta_j\in \Psi_j^{2D}}$, such that
	\[
	\|\tfrac{\partial}{\partial y}f\|^2_{\LpO{2}}
	=\Big\|
	\sum_{\alpha \in \breve{\Phi}_{J_0}^{2D,y}} c_\alpha \alpha+\sum_{j=J_0}^\infty
	\sum_{\beta_j\in \breve{\Psi}_j^{2D,y}} c_{\beta_j} \beta_j\Big\|^2_{\LpO{2}}
	\le C_5 \Big(\sum_{\alpha \in \Phi_{J_0}^{2D}} |c_\alpha|^2+\sum_{j=J_0}^\infty \sum_{\beta_j\in \Psi_j^{2D}} |c_{\beta_j}|^2\Big).
	\]
	The upper bound of \eqref{rz:H1:2D} is now proved by applying a similar argument to $\tfrac{\partial}{\partial x} f$, and appealing to the above inequality and \eqref{rz:L2:2D}.
	
	Next, we prove the lower bound of \eqref{rz:H1:2D}. Similar to \eqref{mL2},
	for functions in $\LpO{2}$, we define
	\[
	[\eta]^\circ(x,y):= \int_{y}^{1} f(x,t) dt, \quad (x,y)\in [0,1]^2, \eta \in \LpO{2}.
	\]
	Since $f(x,0)=0$ for all $x \in [0,1]$, we have $\la f, \tilde{\beta}_j \ra = \la f, [\tilde{\beta}_j]^{\circ} \ra$ and $\la f, \tilde{\alpha}_j \ra = \la f, [\tilde{\alpha}_j]^{\circ} \ra$ by recalling the tensor product structure of $\tilde{\beta}_j$ and $\tilde{\alpha}$. Therefore,
	\be \label{cbetaj:2D}
	c_\alpha=2^{J_0}\la f, \tilde{\alpha}\ra
	=\la \tfrac{\partial}{\partial y}f,2^{J_0} [\tilde{\alpha}]^\circ \ra
	\quad \mbox{and}\quad
	c_{\beta_j}=2^{j}\la f, \tilde{\beta}_j\ra=
	\la \tfrac{\partial}{\partial y}f, 2^{j} [\tilde{\beta}_j]^\circ\ra.
	\ee
	Define
	\[
	\breve{\tilde{\Phi}}^{2D,y}_{J_0}:=\{ 2^{-J_0} \tilde{\Phi}^{x}_{J_0} \otimes \breve{\tilde{\Phi}}^{y}_{J_0}\} \quad \text{and} \quad
	\breve{\tilde{\Psi}}^{2D,y}_{j}:= \{ 2^{-j} \tilde{\Phi}^{x}_j \otimes \breve{\tilde{\Psi}}^{y}_j,  2^{-j} \tilde{\Psi}^{x}_{j}  \otimes \breve{\tilde{\Phi}}^{y}_{j}, 2^{-j} \tilde{\Psi}^{x}_j \otimes \breve{\tilde{\Psi}}^{y}_j\},
	\]
	where $\breve{\tilde{\Phi}}^{y}_{J_0}$ and $\breve{\tilde{\Psi}}^{y}_j$ for $j \ge J_0$ are defined as in \eqref{dphipsi:breve} with $\mathring{\tilde{\phi}}^{R}$ and $\mathring{\tilde{\psi}}^{R}$ replaced by $\mathring{\tilde{\phi}}^{R,y}$ and $\mathring{\tilde{\psi}}^{R,y}$ respectively. Note that all elements in $\breve{\tilde{\Phi}}^{2D,y}_{J_0}$ and $\breve{\tilde{\Psi}}^{2D,y}_{j}$ must belong to $H^\gep(\R^2)$ for some $\gep>0$.
	Applying \cref{thm:bessel} and \cite[Theorem~2.3]{han03},
	we have that the system $\breve{\tilde{\Phi}}_{J_0}^{2D,y} \cup \cup_{j=J_0}^{\infty} \breve{\tilde{\Psi}}^{2D,y}_{j}$ is a Bessel sequence in $\LpO{2}$ and therefore, there exists a positive constant $C_6$, independent of $\tfrac{\partial}{\partial y}f$ and $\{c_\alpha\}_{\alpha\in \Phi_{J_0}^{2D}}\cup \{c_{\beta_j}\}_{j\ge J_0, \beta_j\in \Psi_j^{2D}}$, such that
	\be \label{C6:2D}
	\sum_{\alpha \in \Phi_{J_0}^{2D}} |c_\alpha|^2+\sum_{j=J_0}^\infty\sum_{\beta\in \Psi_j^{2D}} |c_{\beta_j}|^2
	=\sum_{\tilde{\alpha} \in \breve{\tilde{\Phi}}_{J_0}^{2D,y}} |\la \tfrac{\partial}{\partial y} f, \tilde{\alpha}\ra|^2
	 +\sum_{j=J_0}^\infty\sum_{\tilde{\beta}_j\in \breve{\tilde{\Psi}}_j^{2D,y}} |\la \tfrac{\partial}{\partial y} f, \tilde{\beta}_j\ra|^2\le C_6\|\tfrac{\partial}{\partial y} f\|^2_{\LpO{2}},
	\ee
	where we used the identities in \eqref{cbetaj:2D}. The lower bound of \eqref{rz:H1:2D} is now proved with $C_1:=\frac{1}{C_6}<\infty$ by the trivial inequality
	$\|\frac{\partial}{\partial y} f\|^2_{\LpO{2}}\le \|f\|^2_{H^1(\Omega)}$.
	The remaining of this proof now follows the proof of \cref{thm:H1} with appropriate modifications for the 2D setting.
	\ep


\begin{thebibliography}{99}
		
		\bibitem{ABW02} H.~Ammari, G.~Bao, and A.~W.~Wood, Analysis of the electromagnetic scattering from a cavity. \emph{Japan J. Indust. Appl. Math.} \textbf{19} (2002), 301-310.
		
		\bibitem{BL14} G.~Bao and J.~Lai, Radar cross section reduction of a cavity in the ground plane. \emph{Commun. Comput. Phys.} \textbf{15} (2014), no. 4, 895-910.
		
		\bibitem{BS05} G.~Bao and W.~Sun, A fast algorithm for the electromagnetic scattering from a large cavity. \emph{SIAM J. Sci. Comput.} \textbf{27} (2005), no. 2, 553-574.
		
		\bibitem{BY16} G.~Bao and K.~Yun, Stability for the electromagnetic scattering from large cavities. \emph{Arch. Rational Mech. Anal.} \textbf{220} (2016), 1003-1044.
		
		\bibitem{BYZ12} G.~Bao, K.~Yun, and Z.~Zou, Stability of the scattering from a large electromagnetic cavity in two dimensions. \emph{SIAM J. Math. Anal.} \textbf{44} (2012), no.1, 383-404.
		
		\bibitem{cer19}
		D.~\v{C}ern\'a, Wavelets on the interval and their applications, Habilitation thesis at Masaryk University, (2019).
		
		\bibitem{CS12} N.~Chegini and R.~Stevenson, The adaptive tensor product wavelet scheme: sparse matrices and the application to singularly perturbed problems. \emph{IMA J. Numer. Anal.} \textbf{32} (2012), no. 1, 75-104.

		\bibitem{cohbook}
		A.~Cohen, Numerical Analysis of Wavelet Methods. Elsevier, Amsterdam (2003).
		
		\bibitem{cdf92} A.~Cohen, I.~Daubechies, and J.~C.~Feauveau, Biorthogonal bases of compactly supported wavelets. \emph{Comm. Pure Appl. Math.}
		\textbf{45} (1992), 485--560.
		
		\bibitem{dah96}
		W.~Dahmen, Stability of multiscale transformations. \emph{J. Fourier Anal. Appl.} \textbf{4} (1996), 341--362.
		
		\bibitem{dah97}
		W.~Dahmen, Wavelet and multiscale methods for operator equations. \emph{Acta Numer.} \textbf{6} (1997), 55--228.

		\bibitem{dhjk00}
		W.~Dahmen, B.~Han, R.-Q.~Jia, and A.~Kunoth, Biorthogonal multiwavelets on the interval: cubic Hermite splines. \emph{Constr. Approx.} \textbf{16} (2000), 221--259.		
		
		\bibitem{dku99}
		W.~Dahmen, A.~Kunoth and K.~Urban, Biorthogonal spline wavelets on the interval---stability and moment conditions. \emph{Appl. Comput. Harmon. Anal.} \textbf{6} (1999), 132--196.
	
		\bibitem{DS10} T.~J.~Dijkema and R.~Stevenson, A sparse Laplacian in tensor product wavelet coordinates. \emph{Numer. Math.} \textbf{115} (2010), 433-449.
		
		\bibitem{DLS15} K.~Du, B.~Li, and W.~Sun. A numerical study on the stability of a class of Helmholtz problems. \emph{J. Comput. Phys.} \textbf{287} (2015), 46-59.
		
		\bibitem{DLSY18} K.~Du, B.~Li, W.~Sun, and H.~Yang. Electromagnetic scattering from a cavity embedded in an impedance ground plane. \emph{Math. Methods in Applied Sciences} \textbf{41} (2018), 7748-7765.
		
		\bibitem{DSZ13} K.~Du, W.~Sun, and X.~Zhang, Arbitrary high-order $C^{0}$ tensor product Galerkin finite element methods for the electromagnetic scattering from a large cavity. \emph{J. Comput. Phys.} \textbf{242} (2013), 181-195.
		
		\bibitem{han01}
		B.~Han, Approximation properties and construction of Hermite interpolants and biorthogonal multiwavelets. \emph{J. Approx. Theory} \textbf{110} (2001), 18--53.
		
		\bibitem{han03}
		B.~Han, Compactly supported tight wavelet frames and orthonormal wavelets of exponential decay with a general dilation matrix. \emph{J. Comput. Appl. Math.} \textbf{155} (2003), 43--67.

		\bibitem{han12}
		B.~Han, Nonhomogeneous wavelet systems in high dimensions. \emph{Appl. Comput. Harmon. Anal.} \textbf{32} (2012), 169--196.
		
		\bibitem{hanbook}
		B.~Han, Framelets and wavelets: Algorithms, analysis, and applications. \emph{Applied and Numerical Harmonic Analysis}. Birkh\"auser/Springer, Cham, 2017. xxxiii + 724 pp.
		
		\bibitem{hm18}
		B.~Han and M.~Michelle, Construction of wavelets and framelets on a bounded interval. \emph{Anal. Appl.} \textbf{16} (2018), 807--849.
	
\bibitem{hm19}
B.~Han and M.~Michelle, Derivative-orthogonal Riesz wavelets in Sobolev spaces with applications to differential equations.
\emph{Appl. Comput. Harmon. Anal.} \textbf{47} (2019), 759--794.
	
		\bibitem{HM21a} B.~Han and M.~Michelle, Wavelets on intervals derived from arbitrary compactly supported biorthogonal multiwavelets. \emph{Appl. Comp. Harmon. Anal.} \textbf{53} (2021), 270-331.


		\bibitem{HM22} B.~Han and M.~Michelle, Sharp wavenumber-explicit stability bounds for 2D Helmholtz equations. \emph{SIAM J. Numer. Anal.} \textbf{60} (2022), no. 4, 1985-2013.

		\bibitem{HM24} B.~Han and M.~Michelle, Galerkin scheme using biorthogonal wavelets on intervals for elliptic interface problems. Preprint, (2024), \url{https://arXiv.org/abs/2410.16596/}.


		\bibitem{HMW21} B.~Han, M.~Michelle, and Y.~S.~Wong, Dirac assisted tree method for 1D heterogeneous Helmholtz equations with arbitrary variable  wave numbers. \emph{Comput. Math. Appl.} \textbf{97} (2021), 416-438.


 		
		\bibitem{HMS16} L.~Hu, L.~Ma, and J.~Shen, Efficient spectral-Galerkin method and analysis for elliptic PDEs with non-local boundary conditions. \emph{J. Sci. Comput.} \textbf{68} (2016), 417-437.
		
		\bibitem{K01} A.~Kunoth, Wavelet methods -- elliptic boundary value problems and control problems. \emph{Advances in Numerical Mathematics}. Vieweg+Teubner Verlag Wiesbaden, 2001. x + 141 pp.
		
		\bibitem{LC13} B.~Li and X.~Chen, Wavelet-based numerical analysis: a review and classification. \emph{Finite Elem. Anal. Des.} \textbf{81} (2014), 14-31.
	
		\bibitem{LMS13} H.~Li, H.~Ma, and W.~Sun, Legendre spectral Galerkin method for electromagnetic scattering from large cavities. \emph{SIAM J. Numer. Anal.} \textbf{51} (2013), no. 1, 353-376.
		
		\bibitem{LS10} B.~Li and W.~Sun, Newton-Cotes rules for Hadamard finite-part integrals on an interval. \emph{IMA J. Numer. Anal.} \textbf{30} (2010), 1235-1255.
		
		
		\bibitem{MS11} J.~M.~Melenk and S. Sauter, Wavenumber explicit convergence analysis for Galerkin discretizations of the Helmholtz equation. \emph{SIAM J. Numer. Anal.} \textbf{49} (2011), no. 3, 1210-1243.
		
		\bibitem{S09} R.~Stevenson, Adaptive wavelet methods for solving operator equations: an overview. \emph{Multiscale, Nonlinear, and Adaptive Approximation}. Springer, Berlin, Heidelberg, 2009, 543-597.
		
		\bibitem{U09} K.~Urban, Wavelet methods for elliptic partial differential equations. \emph{Numerical Mathematics and Scientific Computation}. Oxford University Press, Oxford, 2009. xxvii + 480 pp.
		
		\bibitem{WWLS08} J.~Wu, Y.~Wang, W.~Li, and W.~Sun, Toeplitz-type approximations to the Hadamard integral operator and their applications to electromagnetic cavity problems. \emph{Appl. Numer. Math.} \textbf{58} (2008), 101-121.
		
		\bibitem{ZQT11} M.~Zhao, Z.~Qiao, and T.~Tang, A fast high order method for electromagnetic scattering by large open cavities. \emph{J. Comput. Math.} \textbf{29} (2011), no. 3, 287-304.
	\end{thebibliography}
\end{document}